\documentclass[12pt]{amsart}

\usepackage{amsmath}
\usepackage{amssymb}
\usepackage{amscd}
\usepackage{color}
\usepackage{hyperref}
\usepackage{mathrsfs}
\usepackage{upgreek}
\usepackage{eucal}
\usepackage{mathtools}
\usepackage{enumitem}
\usepackage{comment}
\usepackage{extpfeil, extarrows}

\usepackage[normalem]{ulem}

\usepackage{float}
\restylefloat{table}

\usepackage{xy}
\xyoption{all}

\newenvironment{renumerate}{\begin{enumerate}[label={\textup{(\roman*)}}]}{\end{enumerate}}
\newenvironment{aenumerate}{\begin{enumerate}[label={\textup{(\alph*)}}]}{\end{enumerate}}

\makeatletter
\newcommand*{\da@rightarrow}{\mathchar"0\hexnumber@\symAMSa 4B }
\newcommand*{\xdashrightarrow}[2][]{%
\mathrel{%
\mathpalette{\da@xarrow{#1}{#2}{}\da@rightarrow{\,}{}}{}%
}%
}
\newcommand*{\da@xarrow}[7]{%
\sbox0{$\ifx#7\scriptstyle\scriptscriptstyle\else\scriptstyle\fi#5#1#6\m@th$}%
\sbox2{$\ifx#7\scriptstyle\scriptscriptstyle\else\scriptstyle\fi#5#2#6\m@th$}%
\sbox4{$#7\dabar@\m@th$}%
\dimen@=\wd0 %
\ifdim\wd2 >\dimen@
\dimen@=\wd2 %
\fi
\count@=2 %
\def\da@bars{\dabar@\dabar@}%
\@whiledim\count@\wd4<\dimen@\do{%
\advance\count@\@ne
\expandafter\def\expandafter\da@bars\expandafter{%
\da@bars
\dabar@ 
}%
}%
\mathrel{#3}%
\mathrel{%
\mathop{\da@bars}\limits
\ifx\\#1\\%
\else
_{\copy0}%
\fi
\ifx\\#2\\%
\else
^{\copy2}%
\fi
}%
\mathrel{#4}%
}
\makeatother

\newcommand{\nc}{\newcommand}

\nc{\rv}{\upupsilon}

\nc{\CC}{{\mathbb{C}}}
\nc{\LL}{{\mathbb{L}}}
\nc{\RR}{{\mathbb{R}}}
\renewcommand{\P}{{\mathbb{P}}}
\nc{\OO}{{\mathbb{O}}}

\nc{\QQ}{{\mathbb{Q}}}
\nc{\ZZ}{{\mathbb{Z}}}

\nc{\cA}{{\mathcal{A}}}
\nc{\cB}{{\mathcal{B}}}
\nc{\cC}{{\mathcal{C}}}
\nc{\cD}{{\mathcal{D}}}
\nc{\cE}{{\mathcal{E}}}
\nc{\tcE}{{\tilde{\mathcal{E}}}}
\nc{\cF}{{\mathcal{F}}}
\nc{\tcF}{{\tilde{\mathcal{F}}}}
\nc{\cG}{{\mathcal{G}}}
\nc{\cH}{{\mathcal{H}}}
\nc{\cI}{{\mathcal{I}}}
\nc{\cJ}{{\mathcal{J}}}
\nc{\cK}{{\mathcal{K}}}
\nc{\cL}{{\mathcal{L}}}
\nc{\cM}{{\mathcal{M}}}
\nc{\cN}{{\mathcal{N}}}
\nc{\cO}{{\mathcal{O}}}
\nc{\cP}{{\mathcal{P}}}
\nc{\cQ}{{\mathcal{Q}}}
\nc{\cR}{{\mathcal{R}}}
\nc{\cS}{{\mathcal{S}}}
\nc{\cT}{{\mathcal{T}}}
\nc{\cU}{{\mathcal{U}}}
\nc{\tcU}{{\tilde{\mathcal{U}}}}
\nc{\cV}{{\mathcal{V}}}
\nc{\cW}{{\mathcal{U}}}
\nc{\cWW}{{\mathcal{W}}}
\nc{\cX}{{\mathcal{X}}}
\nc{\cY}{{\mathcal{Y}}}
\nc{\cZ}{{\mathcal{Z}}}

\nc{\rc}{{\mathrm{c}}}
\nc{\rd}{{\mathrm{d}}}
\nc{\rf}{{\mathsf{f}}}
\nc{\rch}{{\mathrm{ch}}}
\nc{\rtd}{{\mathrm{td}}}

\nc{\rB}{{\mathrm{B}}}
\nc{\rC}{{\mathrm{C}}}
\nc{\rF}{{\mathrm{F}}}
\nc{\rG}{{\mathrm{G}}}
\nc{\rH}{{\mathrm{H}}}
\nc{\rK}{{\mathrm{K}}}
\nc{\rL}{{\mathrm{L}}}
\nc{\rM}{{\mathrm{M}}}
\nc{\rP}{{\mathrm{P}}}
\nc{\rR}{{\mathrm{R}}}
\nc{\rS}{{\mathrm{S}}}
\nc{\rT}{{\mathrm{T}}}
\nc{\rW}{{\mathrm{W}}}
\nc{\rX}{{\mathrm{X}}}
\nc{\rZ}{{\mathrm{Z}}}

\nc{\rQ}{{\mathrm{Q}}}

\nc{\bA}{{\mathbf{A}}}
\nc{\bB}{{\mathbf{B}}}
\nc{\bC}{{\mathbf{C}}}
\nc{\bD}{{\mathbf{D}}}
\nc{\bE}{{\mathbf{E}}}
\nc{\bF}{{\mathbf{F}}}
\nc{\bG}{{\mathbf{G}}}
\nc{\bH}{{\mathbf{H}}}
\nc{\bI}{{\mathbf{I}}}
\nc{\bJ}{{\mathbf{J}}}
\nc{\bK}{{\mathbf{K}}}
\nc{\bL}{{\mathbf{L}}}
\nc{\bM}{{\mathbf{M}}}
\nc{\bN}{{\mathbf{N}}}
\nc{\bO}{{\mathbf{O}}}
\nc{\bP}{{\mathbf{P}}}
\nc{\bQ}{{\mathbf{Q}}}
\nc{\bR}{{\mathbf{R}}}
\nc{\bS}{{\mathbf{S}}}
\nc{\bT}{{\mathbf{T}}}
\nc{\bU}{{\mathbf{U}}}
\nc{\bV}{{\mathbf{V}}}
\nc{\bW}{{\mathbf{W}}}
\nc{\bX}{{\mathbf{X}}}
\nc{\bY}{{\mathbf{Y}}}
\nc{\bZ}{{\mathbf{Z}}}

\nc{\ba}{{\mathbf{a}}}
\nc{\bb}{{\mathbf{b}}}
\nc{\bc}{{\mathbf{c}}}
\nc{\bd}{{\mathbf{d}}}
\nc{\be}{{\mathbf{e}}}
\nc{\bg}{{\mathbf{g}}}
\nc{\bh}{{\mathbf{h}}}
\nc{\bi}{{\mathbf{i}}}
\nc{\bj}{{\mathbf{j}}}
\nc{\bk}{{\mathbf{k}}}
\nc{\bl}{{\mathbf{l}}}
\nc{\bm}{{\mathbf{m}}}
\nc{\bn}{{\mathbf{n}}}
\nc{\bo}{{\mathbf{o}}}
\nc{\bp}{{\mathbf{p}}}
\nc{\bq}{{\mathbf{q}}}
\nc{\br}{{\mathbf{r}}}
\nc{\bs}{{\mathbf{s}}}
\nc{\bt}{{\mathbf{t}}}
\nc{\bu}{{\mathbf{u}}}
\nc{\bv}{{\mathbf{v}}}
\nc{\bw}{{\mathbf{w}}}
\nc{\bx}{{\mathbf{x}}}
\nc{\by}{{\mathbf{y}}}
\nc{\bz}{{\mathbf{z}}}

\nc{\fA}{{\mathfrak{A}}}
\nc{\fB}{{\mathfrak{B}}}
\nc{\fC}{{\mathfrak{C}}}
\nc{\fD}{{\mathfrak{D}}}
\nc{\fE}{{\mathfrak{E}}}
\nc{\fF}{{\mathfrak{F}}}
\nc{\fG}{{\mathfrak{G}}}
\nc{\fH}{{\mathfrak{H}}}
\nc{\fI}{{\mathfrak{I}}}
\nc{\fJ}{{\mathfrak{J}}}
\nc{\fK}{{\mathfrak{K}}}
\nc{\fL}{{\mathfrak{L}}}
\nc{\fM}{{\mathfrak{M}}}
\nc{\fN}{{\mathfrak{N}}}
\nc{\fO}{{\mathfrak{O}}}
\nc{\fP}{{\mathfrak{P}}}
\nc{\fQ}{{\mathfrak{Q}}}
\nc{\fR}{{\mathfrak{R}}}
\nc{\fS}{{\mathfrak{S}}}
\nc{\fT}{{\mathfrak{T}}}
\nc{\fU}{{\mathfrak{U}}}
\nc{\fV}{{\mathfrak{V}}}
\nc{\fW}{{\mathfrak{W}}}
\nc{\fX}{{\mathfrak{X}}}
\nc{\fY}{{\mathfrak{Y}}}
\nc{\fZ}{{\mathfrak{Z}}}

\nc{\fa}{{\mathfrak{a}}}
\nc{\fb}{{\mathfrak{b}}}
\nc{\fc}{{\mathfrak{c}}}
\nc{\fd}{{\mathfrak{d}}}
\nc{\fe}{{\mathfrak{e}}}
\nc{\ff}{{\mathfrak{f}}}
\nc{\fg}{{\mathfrak{g}}}
\nc{\fh}{{\mathfrak{h}}}
\nc{\fj}{{\mathfrak{j}}}
\nc{\fk}{{\mathfrak{k}}}
\nc{\fl}{{\mathfrak{l}}}
\nc{\fm}{{\mathfrak{m}}}
\nc{\fn}{{\mathfrak{n}}}
\nc{\fo}{{\mathfrak{o}}}
\nc{\fp}{{\mathfrak{p}}}
\nc{\fq}{{\mathfrak{q}}}
\nc{\fr}{{\mathfrak{r}}}
\nc{\fs}{{\mathfrak{s}}}
\nc{\ft}{{\mathfrak{t}}}
\nc{\fu}{{\mathfrak{u}}}
\nc{\fv}{{\mathfrak{v}}}
\nc{\fw}{{\mathfrak{w}}}
\nc{\fx}{{\mathfrak{x}}}
\nc{\fy}{{\mathfrak{y}}}
\nc{\fz}{{\mathfrak{z}}}

\nc{\sA}{{\mathsf{A}}}
\nc{\sB}{{\mathsf{B}}}
\nc{\sC}{{\mathsf{C}}}
\nc{\sD}{{\mathsf{D}}}
\nc{\sE}{{\mathsf{E}}}
\nc{\sF}{{\mathsf{F}}}
\nc{\sG}{{\mathsf{G}}}
\nc{\sH}{{\mathsf{H}}}
\nc{\sI}{{\mathsf{I}}}
\nc{\sJ}{{\mathsf{J}}}
\nc{\sK}{{\mathsf{K}}}
\nc{\sL}{{\mathsf{L}}}
\nc{\sM}{{\mathsf{M}}}
\nc{\sN}{{\mathsf{N}}}
\nc{\sO}{{\mathsf{O}}}
\nc{\sP}{{\mathsf{P}}}
\nc{\sQ}{{\mathsf{Q}}}
\nc{\sR}{{\mathsf{R}}}
\nc{\sS}{{\mathsf{S}}}
\nc{\sT}{{\mathsf{T}}}
\nc{\sU}{{\mathsf{U}}}
\nc{\sV}{{\mathsf{V}}}
\nc{\sW}{{\mathsf{W}}}
\nc{\sX}{{\mathsf{X}}}
\nc{\sY}{{\mathsf{Y}}}
\nc{\sZ}{{\mathsf{Z}}}

\nc{\sa}{{\mathsf{a}}}
\nc{\sd}{{\mathsf{d}}}
\nc{\se}{{\mathsf{e}}}
\nc{\sg}{{\mathsf{g}}}
\nc{\sh}{{\mathsf{h}}}
\nc{\si}{{\mathsf{i}}}
\nc{\sj}{{\mathsf{j}}}
\nc{\sk}{{\mathsf{k}}}
\nc{\sn}{{\mathsf{n}}}
\nc{\so}{{\mathsf{o}}}
\nc{\sq}{{\mathsf{q}}}
\nc{\sr}{{\mathsf{r}}}
\nc{\st}{{\mathsf{t}}}
\nc{\su}{{\mathsf{u}}}
\nc{\sv}{{\mathsf{v}}}
\nc{\sw}{{\mathsf{w}}}
\nc{\sx}{{\mathsf{x}}}
\nc{\sy}{{\mathsf{y}}}
\nc{\sz}{{\mathsf{z}}}

\nc{\oA}{{\overline{A}}}
\nc{\oB}{{\overline{B}}}
\nc{\oC}{{\overline{C}}}
\nc{\oD}{{\overline{D}}}
\nc{\oE}{{\overline{E}}}
\nc{\oF}{{\overline{F}}}
\nc{\oG}{{\overline{G}}}
\nc{\oH}{{\overline{H}}}
\nc{\oI}{{\overline{I}}}
\nc{\oJ}{{\overline{J}}}
\nc{\oK}{{\overline{K}}}
\nc{\oL}{{\overline{L}}}
\nc{\oM}{{\overline{M}}}
\nc{\oN}{{\overline{N}}}
\nc{\oO}{{\overline{O}}}
\nc{\oP}{{\overline{P}}}
\nc{\oQ}{{\overline{Q}}}
\nc{\oR}{{\overline{R}}}
\nc{\oS}{{\overline{S}}}
\nc{\oT}{{\overline{T}}}
\nc{\oU}{{\overline{U}}}
\nc{\oV}{{\overline{V}}}
\nc{\oW}{{\overline{W}}}
\nc{\oX}{{\overline{X}}}
\nc{\oY}{{\overline{Y}}}
\nc{\oZ}{{\overline{Z}}}

\nc{\oa}{{\overline{a}}}
\nc{\ob}{{\overline{b}}}
\nc{\oc}{{\overline{c}}}
\nc{\od}{{\overline{d}}}
\nc{\of}{{\overline{f}}}
\nc{\og}{{\overline{g}}}
\nc{\oh}{{\overline{h}}}
\nc{\oi}{{\overline{i}}}
\nc{\oj}{{\overline{j}}}
\nc{\ok}{{\overline{k}}}
\nc{\ol}{{\overline{l}}}
\nc{\om}{{\overline{m}}}
\nc{\on}{{\overline{n}}}
\nc{\oo}{{\overline{o}}}
\nc{\op}{{\overline{p}}}
\nc{\oq}{{\overline{q}}}
\nc{\os}{{\overline{s}}}
\nc{\ot}{{\overline{t}}}
\nc{\ou}{{\overline{u}}}
\nc{\ov}{{\overline{v}}}
\nc{\ow}{{\overline{w}}}
\nc{\ox}{{\overline{x}}}
\nc{\oy}{{\overline{y}}}
\nc{\oz}{{\overline{z}}}

\nc{\tA}{{\tilde{A}}}
\nc{\tB}{{\tilde{B}}}
\nc{\tC}{{\tilde{C}}}
\nc{\tD}{{\tilde{D}}}
\nc{\tE}{{\tilde{E}}}
\nc{\tF}{{\tilde{F}}}
\nc{\tG}{{\tilde{G}}}
\nc{\tH}{{\tilde{H}}}
\nc{\tI}{{\tilde{I}}}
\nc{\tJ}{{\tilde{J}}}
\nc{\tK}{{\tilde{K}}}
\nc{\tL}{{\tilde{L}}}
\nc{\tM}{{\tilde{M}}}
\nc{\tN}{{\tilde{N}}}
\nc{\tO}{{\tilde{O}}}
\nc{\tP}{{\tilde{P}}}
\nc{\tQ}{{\tilde{Q}}}
\nc{\tR}{{\tilde{R}}}
\nc{\tS}{{\tilde{S}}}
\nc{\tT}{{\tilde{T}}}
\nc{\tU}{{\tilde{U}}}
\nc{\tV}{{\tilde{V}}}
\nc{\tW}{{\widetilde{W}}}
\nc{\tX}{{\tilde{X}}}
\nc{\tY}{{\tilde{Y}}}
\nc{\tZ}{{\tilde{Z}}}

\nc{\ta}{{\tilde{a}}}
\nc{\tb}{{\tilde{b}}}
\nc{\tc}{{\tilde{c}}}
\nc{\td}{{\tilde{d}}}
\nc{\te}{{\tilde{e}}}
\nc{\tf}{{\tilde{f}}}
\nc{\tg}{{\tilde{g}}}
\nc{\ti}{{\tilde{i}}}
\nc{\tj}{{\tilde{j}}}
\nc{\tk}{{\tilde{k}}}
\nc{\tl}{{\tilde{l}}}
\nc{\tm}{{\tilde{m}}}
\nc{\tn}{{\tilde{n}}}
\nc{\tp}{{\tilde{p}}}
\nc{\tq}{{\tilde{q}}}
\nc{\tr}{{\tilde{r}}}
\nc{\ts}{{\tilde{s}}}
\nc{\tu}{{\tilde{u}}}
\nc{\tv}{{\tilde{v}}}
\nc{\tw}{{\tilde{w}}}
\nc{\tx}{{\tilde{x}}}
\nc{\ty}{{\tilde{y}}}
\nc{\tz}{{\tilde{z}}}

\nc{\hA}{{\hat{A}}}
\nc{\hB}{{\hat{B}}}
\nc{\hC}{{\hat{C}}}
\nc{\hD}{{\hat{D}}}
\nc{\hE}{{\hat{E}}}
\nc{\hF}{{\hat{F}}}
\nc{\hG}{{\hat{G}}}
\nc{\hH}{{\hat{H}}}
\nc{\hI}{{\hat{I}}}
\nc{\hJ}{{\hat{J}}}
\nc{\hK}{{\hat{K}}}
\nc{\hL}{{\hat{L}}}
\nc{\hM}{{\hat{M}}}
\nc{\hN}{{\hat{N}}}
\nc{\hO}{{\hat{O}}}
\nc{\hP}{{\hat{P}}}
\nc{\hQ}{{\hat{Q}}}
\nc{\hR}{{\hat{R}}}
\nc{\hS}{{\widehat{S}}}
\nc{\hT}{{\hat{T}}}
\nc{\hU}{{\widehat{U}}}
\nc{\hV}{{\hat{V}}}
\nc{\hW}{{\hat{W}}}
\nc{\hX}{{\hat{X}}}
\nc{\hY}{{\hat{Y}}}
\nc{\hZ}{{\hat{Z}}}

\nc{\ha}{{\hat{a}}}
\nc{\hb}{{\hat{b}}}
\nc{\hc}{{\hat{c}}}
\nc{\hd}{{\hat{d}}}
\nc{\he}{{\hat{e}}}
\nc{\hf}{{\widehat{f}}}
\nc{\hg}{{\hat{g}}}
\nc{\hh}{{\hat{h}}}
\nc{\hi}{{\hat{i}}}
\nc{\hj}{{\hat{j}}}
\nc{\hk}{{\hat{k}}}
\nc{\hl}{{\hat{l}}}
\nc{\hm}{{\hat{m}}}
\nc{\hn}{{\hat{n}}}
\nc{\ho}{{\hat{o}}}
\nc{\hp}{{\hat{p}}}
\nc{\hq}{{\hat{q}}}
\nc{\hr}{{\hat{r}}}
\nc{\hs}{{\hat{s}}}
\nc{\hu}{{\hat{u}}}
\nc{\hv}{{\hat{v}}}
\nc{\hw}{{\hat{w}}}
\nc{\hx}{{\hat{x}}}
\nc{\hy}{{\hat{y}}}
\nc{\hz}{{\hat{z}}}

\nc{\eps}{\varepsilon}
\nc{\lan}{\big\langle}
\nc{\ran}{\big\rangle}
\nc{\kk}{{\Bbbk}}
\newcommand{\sm}{{\mathrm{sm}}}

\nc{\uZ}{\underline{\ZZ}}

\nc{\EE}{\mathbb{E}}

\nc{\et}{{\mathrm{\acute{e}t}}}
\nc{\num}{{\mathrm{num}}}

\nc{\xrightiso}{ \xrightarrow{\ \raisebox{-0.5ex}[0ex][0ex]{$\sim$}\ }}

\def\bw#1#2{\textstyle{\bigwedge\hskip-0.9mm^{#1}}\hskip0.2mm{#2}}

\DeclareMathOperator{\Aut}{\mathrm{Aut}}

\DeclareMathOperator{\Hom}{\mathrm{Hom}}
\DeclareMathOperator{\Ext}{\mathrm{Ext}}

\DeclareMathOperator{\Bl}{\mathrm{Bl}}

\DeclareMathOperator{\Sing}{\mathrm{Sing}}
\DeclareMathOperator{\Pic}{\mathrm{Pic}}

\DeclareMathOperator{\Br}{\mathrm{Br}}

\DeclareMathOperator{\Sym}{\mathrm{Sym}}
\DeclareMathOperator{\Ker}{\mathrm{Ker}}
\DeclareMathOperator{\Coker}{\mathrm{Coker}}
\DeclareMathOperator{\Ima}{\mathrm{Im}}
\DeclareMathOperator{\Cone}{\mathrm{Cone}}

\DeclareMathOperator{\pr}{\mathrm{pr}}

\DeclareMathOperator{\Gr}{\mathrm{Gr}}

\DeclareMathOperator{\OGr}{\mathrm{OGr}}

\DeclareMathOperator{\LGr}{\mathrm{LGr}}

\DeclareMathOperator{\GL}{\mathrm{GL}}
\DeclareMathOperator{\SL}{\mathrm{SL}}

\DeclareMathOperator{\SP}{\mathrm{SP}}
\DeclareMathOperator{\SO}{\mathrm{SO}}
\DeclareMathOperator{\Spin}{\mathrm{Spin}}

\DeclareMathOperator{\id}{\mathrm{id}}
\DeclareMathOperator{\ev}{\mathrm{ev}}
\DeclareMathOperator{\coev}{\mathrm{coev}}
\DeclareMathOperator{\rank}{\mathrm{rk}}
\DeclareMathOperator{\codim}{\mathrm{codim}}

\nc{\bkk}{{\overline{\kk}}}
\newcommand{\g}{{\mathrm{g}}}

\theoremstyle{plain}

\newtheorem{theorem}{Theorem}[section]

\newtheorem{lemma}[theorem]{Lemma}
\newtheorem{proposition}[theorem]{Proposition}
\newtheorem{corollary}[theorem]{Corollary}

\theoremstyle{definition}
\newtheorem{definition}[theorem]{Definition}

\theoremstyle{remark}
\newtheorem{remark}[theorem]{Remark}

\makeatletter
\def\@tocline#1#2#3#4#5#6#7{\relax
  \ifnum #1>\c@tocdepth 
  \else
    \par \addpenalty\@secpenalty\addvspace{#2}%
    \begingroup \hyphenpenalty\@M
    \@ifempty{#4}{%
      \@tempdima\csname r@tocindent\number#1\endcsname\relax
    }{%
      \@tempdima#4\relax
    }%
    \parindent\z@ \leftskip#3\relax \advance\leftskip\@tempdima\relax
    \rightskip\@pnumwidth plus4em \parfillskip-\@pnumwidth
    #5\leavevmode\hskip-\@tempdima
      \ifcase #1
       \or\or \hskip 1em \or \hskip 2em \else \hskip 3em \fi%
      #6\nobreak\relax
    \dotfill\hbox to\@pnumwidth{\@tocpagenum{#7}}\par
    \nobreak
    \endgroup
  \fi}
\makeatother

\makeatletter
\@namedef{subjclassname@2020}{%
  \textup{2020} Mathematics Subject Classification}
\makeatother

\title{Mukai models of Fano varieties}

\begin{document}

\author{Arend Bayer}
\address{{\sloppy
\parbox{0.99\textwidth}{
School of Mathematics and Maxwell Institute for Mathematical Sciences, \\
University of Edinburgh, JCMB, Peter Guthrie Tait Road, Edinburgh EH9 3FD, UK
}\vspace{1.2mm}}}
\email{arend.bayer@ed.ac.uk}

\author{Alexander Kuznetsov}
\address{{\sloppy
\parbox{0.99\textwidth}{
Algebraic Geometry Section, Steklov Mathematical Institute of Russian Academy of Sciences\\
8 Gubkin str., Moscow 119991 Russia
\\[3pt]
Laboratory of Algebraic Geometry, NRU Higher School of Economics, Russian Federation
}\vspace{1.2mm}}}
\email{akuznet@mi-ras.ru}

\author{Emanuele Macr\`i}
\address{{\sloppy
\parbox{0.99\textwidth}{
Universit\'e Paris-Saclay, CNRS, Laboratoire de Math\'ematiques d'Orsay\\
Rue Michel Magat, B\^at. 307, 91405 Orsay, France
}\vspace{1.2mm}}}
\email{emanuele.macri@universite-paris-saclay.fr}

\subjclass[2020]{14J28, 14J30, 14J45, 14J60}
\keywords{Fano varieties, Mukai varieties, K3 surfaces, Mukai vector bundles}
\thanks{A.B. was partially supported by the EPSRC grant EP/R034826/1 and the ERC grant ERC-2018-CoG819864-WallCrossAG. 
A.K. was partially supported by the HSE University Basic Research Program. 
E.M. was partially supported by the ERC grant ERC-2020-SyG-854361-HyperK}

\begin{abstract}
We give a self-contained and simplified proof of Mukai's classification 
of prime Fano threefolds of index~1 and genus~$g \ge 6$ 
with at most factorial terminal singularities, and of its extension to higher-dimension. 
\end{abstract}

\maketitle

\tableofcontents
\setcounter{tocdepth}{2}


\section{Introduction}\label{sec:intro}

In this paper we show that a Fano variety of coindex~$3$ and genus~$g \in \{6,7,8,9,10,12\}$ 
with mild singularities over an algebraically closed field~$\kk$ of characteristic~$0$
is a complete intersection in an iterated cone over one of the {\sf Mukai varieties}~$\rM_g$, thus proving the theorem announced by Mukai in~\cite[Theorem~2]{Mukai:PNAS} for smooth manifolds
and (partially) its extension to singular threefolds suggested in~\cite[Theorem~6.5]{Muk:New}. 

More precisely, for~$g \in \{6,7,8,9,10\}$, we consider the following homogeneous varieties:
\begin{itemize}
\item 
$\rM_6 = \Gr(2,5) = \SL_5/\rP_2 \subset \P^{9}$, 
the Grassmannian of 2-dimensional subspaces in~$\kk^5$;
\item 
$\rM_7 = \OGr_+(5,10) = \SO_{10}/\rP_5 \subset \P^{15}$, 
the connected component of the Grassmannian of isotropic 5-dimensional subspaces in~$\kk^{10}$ 
with respect to a non-degenerate quadratic form;
\item 
$\rM_8 = \Gr(2,6) = \SL_6/\rP_2 \subset \P^{14}$, 
the Grassmannian of 2-dimensional subspaces in~$\kk^6$;
\item 
$\rM_9 = \LGr(3,6) = \SP_6/\rP_3 \subset \P^{13}$, 
the Grassmannian of isotropic 3-dimensional subspaces in~$\kk^{6}$ with respect to a symplectic form;
\item 
$\rM_{10} = \rG_2 / \rP_2 \subset \P^{13}$, 
the adjoint Grassmannian of the simple algebraic group of type~$\rG_2$.
\end{itemize}
Here~$\rP_k$ stands for a maximal parabolic subgroup corresponding to the $k$-th vertex of the Dynkin diagram
of the respective simple group
and the ambient projective space is the projectivization of the corresponding fundamental representation.
Note that all the above varieties are \emph{rigid}.

When~$g = 12$, instead of a single rigid variety, there is a family of Mukai threefolds defined as
\begin{itemize}
\item 
$\rM_{12} \subset \Gr(3,7)$, 
the subvariety parameterizing all 3-dimensional subspaces in~$\kk^7$ isotropic for a triple of $2$-forms~$(\sigma_1,\sigma_2,\sigma_3)$ 
such that any linear combination~$\sum a_i \sigma_i$ has rank~$6$.
\end{itemize}

The following table lists the dimensions~$n_g \coloneqq \dim(\rM_g)$ of these varieties and~$N_g$ of their ambient projective spaces:
\begin{equation}
\label{eq:ng}
\begin{array}{|c|c|c|c|c|c|c|}
\hline
g & 6 & 7 & 8 & 9 & 10 & 12
\\
\hline
\rM_g & \Gr(2,5) & \OGr_+(5,10) & \Gr(2,6) & \LGr(3,6) & \rG_2/\rP_2 & \rM_{12}
\\
\hline
n_g & 6 & 10 & 8 & 6 & 5 & 3
\\
\hline
N_g & 9 & 15 & 14 & 13 & 13 & 13
\\
\hline
\end{array}
\end{equation}
Note that~$\rM_g \subset \P^{N_g}$ has codimension~$g - 2$ for~$g \ge 7$ and~$g - 3 = 3$ for~$g = 6$.

If~$K \subset W$ is a linear subspace and~$Y \subset \P(W/K)$ is a projective subvariety,
we denote by~$\Cone_{\P(K)}(Y) \subset \P(W)$
the cone over~$Y$ with vertex~$\P(K)$.
If~$\dim(K) > 1$ we sometimes call~$\Cone_{\P(K)}(Y)$ an {\sf iterated cone}.
If~$\dim(K) = 1$, so that~$\P(K)$ is a point, we abbreviate this to just~$\Cone(Y)$.
If~$K = 0$, we have~$\Cone_{\P^{-1}}(Y) = Y$.
Our main theorem is the following:

\begin{theorem}
\label{thm:main}
Let~$X$ be a Fano variety of dimension~$n \ge 3$ with at most factorial terminal singularities 
over an algebraically closed field of characteristic zero such that
\begin{equation}
\label{eq:prime-fano}
\Pic(X) = \ZZ \cdot H,
\qquad 
-K_X = (n - 2)H,
\qquad\text{and}\qquad 
H^n = 2g - 2
\end{equation}
for~$g \in \{6,7,8,9,10,12\}$.

\begin{enumerate}[label={\textup{(\alph*)}}, wide]
\item
If~$g \ge 7$, there is a linear subspace~$\P^{n_0 + g - 2} \subset \P^{n_g + g - 2} = \P^{N_g}$ such that
\begin{equation*}
X_0 = \rM_g \cap \P^{n_0 + g - 2}
\end{equation*}
is a dimensionally transverse intersection, 
a Fano variety with at most factorial terminal singularities
of dimension~$n_0$ with~\mbox{$3 \le n_0 \le n_g$}, and
\begin{equation*}
X = \Cone_{\P^{n-n_0-1}}(X_0)
\end{equation*}
is an iterated cone over~$X_0$.
Moreover, if~$n_0 < n$ then~$n_0 \ge 4$.
Finally, $X$ is a local complete intersection if and only if~$n = n_0$, i.e., $X = \rM_g \cap \P^{n + g - 2}$; 
in this case we have~$n \le n_g$.

\item
If~$g = 6$, there is a linear subspace~$\P^{n_0 + 3} \subset \P^9$ such that
\begin{equation*}
Y_0 = \Gr(2,5) \cap \P^{n_0 + 3}
\end{equation*}
is a smooth quintic del Pezzo variety of dimension~$n_0$ with~\mbox{$3 \le n_0 \le 6$}, and
\begin{equation*}
X = \Cone_{\P^{n-n_0}}(Y_0) \cap Q
\end{equation*}
is a dimensionally transverse intersection of an iterated cone over~$Y_0$ with a quadric.
Moreover, if~$Q$ contains the vertex of the cone then~$n_0 \ge 4$, 
and if~$Q$ contains the vertex with multiplicity~$2$ then~$n_0 \ge 5$.
Finally, $X$ is a local complete intersection if and only if~$Q$ does not intersect the vertex of the cone,
hence~$X = \Cone(\Gr(2,5)) \cap \P^{n + 4} \cap Q$; 
in this case we have~$n \le 6$.
\end{enumerate}

The above representation of~$X$ is canonical and unique up to automorphisms of~$\rM_g$. 
\end{theorem}

\begin{remark}
The dimension of~$X$ can be arbitrarily high (because of the cones), except for the case~$g = 12$, where it is bounded by~3.
On the other hand, for local complete intersection varieties (in particular for smooth varieties), 
the dimension is bounded by~$n_g$.
\end{remark}

The canonical divisor class of a factorial variety 
(as any other Weil divisor class) is Cartier,
and terminal singularities are Cohen--Macaulay,
hence any factorial terminal variety is Gorenstein.
Moreover, any Fano threefold with terminal Gorenstein singularities is smoothable by~\cite[Theorem~11]{Namikawa},
so if it satisfies~\eqref{eq:prime-fano} and~$g \ge 6$ 
then the Iskovskih genus bound~$g \in \{6,7,8,9,10,12\}$ holds by~\cite[Theorem~4.6.7]{Fano-book}.
Finally, for Fano varieties of dimension~$n \ge 4$ the same genus bound follows by induction, see Corollary~\ref{cor:genus-bound}.
Thus, Theorem~\ref{thm:main} classifies all Fano varieties satisfying~\eqref{eq:prime-fano} with~$g \ge 6$.

Theorem~\ref{thm:main} is complemented by Iskovskikh's classification of Fano varieties~$X$ 
satisfying~\eqref{eq:prime-fano} with~$g \le 5$.
Such a variety is a complete intersection in a weighted projective space:
\begin{equation*}
X \subset \P(1^{n+1},3),
\qquad 
X \subset \P(1^{n+2},2),
\qquad 
X \subset \P^{n+2},
\qquad
X \subset \P^{n+3}
\end{equation*}
of multidegree~$(6)$, $(2,4)$, $(2,3)$, and~$(2,2,2)$, respectively.
(Here we do not need cones, because the cone over a (weighted) projective space is again a (weighted) projective space.)

The significance of the Mukai--Iskovskikh classification, 
in particular of prime Fano threefolds, is hard to overstate: 
it is an essential ingredient in the study of their moduli spaces, 
their relations to curves and K3 surfaces, 
or, more recently, their K-stability. 
However, despite its influence, a complete proof of Mukai's classification 
even in the smooth case for dimension three has not appeared in the literature before, 
see Section~\ref{sec:History} for an extended discussion of the history.

\subsection{Outline of the proof}

If~$X$ is a Fano variety as in Theorem~\ref{thm:main}, 
we check in Lemma~\ref{lem:hva} that the ample generator~$H$ of~$\Pic(X)$ 
is very ample and gives an embedding~\mbox{$X \subset \P^{n + g - 2}$}.
Moreover, we show that if~$\P^g \subset \P^{n + g - 2}$ is a very general linear subspace then
\begin{equation*}
S \coloneqq X \cap \P^g
\end{equation*}
is a smooth K3 surface 
with $\Pic(S) = \ZZ \cdot H\vert_S$ and $(H\vert_S)^2 = 2g - 2$.
In analogy with the terminology for Fano threefolds, we say that~$S$ is a {\sf prime K3 surface} of genus~$g$.
Following the ideas of Mukai, we deduce Theorem~\ref{thm:main} from the following result.

\begin{theorem}
\label{thm:prime-k3}
Let~$S$ be a smooth prime $K3$ surface of genus~$g \in \{6,7,8,9,10,12\}$
over an algebraically closed field of characteristic zero.
\begin{enumerate}[label={\textup{(\alph*)}}, wide]
\item 
If~$g \ge 7$ there is an embedding~$S \hookrightarrow \rM_g \subset \P^{N_g} = \P^{n_g + g - 2}$ such that
\begin{equation*}
S = \rM_g \cap \P^{g}
\end{equation*}
is a transverse intersection, where~$\P^{g} \subset \P^{n_g + g - 2}$ is the linear span of~$S$.
\item 
If~$g = 6$ there is an embedding~$S \hookrightarrow \rM_6 = \Gr(2,5) \subset \P^9$ such that
\begin{equation*}
S = \Gr(2,5) \cap \P^6 \cap Q
\end{equation*}
is a transverse intersection, where~$\P^6 \subset \P^9$ is the linear span of~$S$ and~$Q$ is a quadric.
\end{enumerate}

In both cases the embedding~$S \hookrightarrow \rM_g$ is unique up to the natural action of~$\Aut(\rM_g)$.
\end{theorem}

More precisely, to prove Theorem~\ref{thm:main} we consider
for each~$g \in \{6,7,8,9,10,12\}$ a pair of integers~$(r,s)$ defined in the following table
\begin{equation}
\label{eq:rs}
\begin{array}{|c|c|c|c|c|c|c|}
\hline
g & 6 & 7 & 8 & 9 & 10 & 12
\\ \hline
(r,s) & (2,3) & (5,5) & (2,4) & (3,3) & (2,5) & (3,4)
\\ \hline
\end{array}
\end{equation}
For~$g \ne 7$ they provide a factorization~$g = r \cdot s$.
In terms of this pair our argument for~$g \ge 8$ splits into the following three steps
(the modifications necessary for~$g \in \{6,7\}$ will be explained later):
\begin{enumerate}[label={\textup{{\bf\arabic*.}}}, ref={\bf\arabic*}]
\item 
\label{step:prime-k3}
For a prime K3 surface~$(S,H)$ of genus~$g \in \{8,9,10,12\}$ 
we consider the unique stable vector bundle~$\cU_S$ on~$S$ of rank~$r$ with~$\rc_1(\cU_S) = -H$ and~$\upchi(\cU_S) = r + s$
(we call it {\sf the Mukai bundle}),
verify that it gives a closed embedding~\mbox{$\gamma_S \colon S \hookrightarrow \rM_g \subset \Gr(r,r+s)$},
and prove that this embedding induces an equality~$S = \rM_g \cap \P^g$ (thus, proving Theorem~\ref{thm:prime-k3}). 
\item 
\label{step:extend-gamma}
We prove that the Mukai bundle~$\cU_S$ on a very general K3 linear section~$S = X \cap \P^g$
extends to a reflexive sheaf~$\cU_X$ on~$X$, 
and that the morphism~$S \to \rM_g$ extends to a rational map~$X \dashrightarrow \rM_g$.
\item 
\label{step:extend-description}
We apply a general extension result, Proposition~\ref{prop:s-x},
that allows us to obtain a description of~$X$ from the description of~$S$.
\end{enumerate}
We explain below the ideas used in these steps.

\subsubsection*{Step~\ref{step:prime-k3}}

The construction of the Mukai bundle~$\cU_S$ is standard and quite general;
in particular, it works for prime K3 surfaces (and even for some Brill--Noether general K3 surfaces, see~\cite[Theorem~3.4]{BKM})
of any genus~$g \ge 2$ and all factorizations~$g = r \cdot s$;
we recall this construction in Proposition~\ref{prop:muk-bundles}
and check that~$\cU_S$ induces a closed embedding
\begin{equation*}
\gamma_S \colon S \to \Gr(r, r + s)
\end{equation*}
(we call it {\sf the Gushel morphism}) in Lemma~\ref{lem:gamma-x-s}.

Further, for~$g \in \{8,9,10,12\}$ we note that
the Mukai variety~$\rM_g$ has an embedding
\begin{equation*}
\rM_g \subset \Gr(r, r + s)
\end{equation*}
(where~$(r,s)$ are defined in~\eqref{eq:rs})
and can be identified inside~$\Gr(r,r+s)$ as the zero locus of a global section 
of a certain vector bundle~$\cE_0$ (see Section~\ref{subsec:MukaiVarieties8-12} for its definition)
that satisfies an appropriate \emph{non-degeneracy condition}
(see Definition~\ref{def:nondegenerate} and Lemma~\ref{lem:rm-9-10}).
To show that the Gushel morphism~$\gamma_S$ factors through~$\rM_g$, 
we check that a section of~$\cE_0$ vanishes on~$\gamma_S(S)$,
(this uses a cohomology vanishing of independent interest, see Lemma~\ref{lem:g10H1vanishing}),
and that this section satisfies the required non-degeneracy condition (Proposition~\ref{prop:sigma}).
Our key technical insight here is Lemma~\ref{lem:no-schubert}, 
which says that prime K3 surfaces are never contained in a Schubert divisor inside the Grassmannian. 
This reduces the verification of the non-degeneracy to a short case-by-case analysis.

Next, we identify~$S \subset \rM_g$ with a transverse linear section of~$\rM_g$.
The fact that~$S \subset \rM_g \cap \P^g$ is quite easy (see Corollary~\ref{cor:s-mukai}),
while the proof that this inclusion is an equality is more tricky; this is done in Proposition~\ref{prop:m-cap-p}.
Our proof is based on Lemma~\ref{lem:yes-schubert}, 
which in a combination with Lemma~\ref{lem:no-schubert} bounds the singularities of linear sections of~$\rM_g$ containing~$S$.
Combining this with Proposition~\ref{lem:zl-deg} 
we deduce that the linear section~$Y \coloneqq \rM_g \cap \P^g$ containing~$S$ 
is either a Cohen--Macaulay surface or a smooth threefold.
In the first case, by computing the degree we conclude that~$S = Y$, as required.
In the second case,
we prove that~$S = Y \cap Q$, where~$Q$ is a quadric, 
and show that a general hyperplane section of~$S$ cannot be a Brill--Noether general curve (see Corollary~\ref{cor:y-cap-q}), 
in contradiction to the celebrated result of Lazarsfeld~\cite[Corollary~1.4]{Laz}.

\subsubsection*{Step~\ref{step:extend-gamma}}

Let~$X$ be a Fano variety as in Theorem~\ref{thm:main}, still assuming~$g \in \{8,9,10,12\}$,
and let~\mbox{$S \coloneqq X \cap \P^g$} be a prime K3 surface linear section
(as we mentioned above, the existence of~$S$ follows from factoriality of~$X$, 
see Lemma~\ref{lem:hva}).
To extend the Gushel morphism~\mbox{$\gamma_S \colon S \to \rM_g$} to a rational map from~$X$,
we extend the Mukai bundle~$\cU_S$ to a maximal Cohen--Macaulay sheaf~$\cU_{X}$ on~$X$
(so that~$\cU_X$ is locally free on the smooth locus~$X_\sm \subset X$)
such that~$\cU_{X}^\vee$ is globally generated with~$\rH^0(X, \cU_{X}^\vee) = \rH^0(S, \cU_S^\vee)$.
For this we use two different approaches for~$n = 3$ and~$n \ge 4$.

In the case~$n = 3$ we use~\cite[Theorem~5.3]{BKM}, 
and as a result we obtain an extension of~$\gamma_S$ to a morphism~$\gamma_{X_\sm} \colon X_\sm \to \Gr(r,r+s)$
such that~$\gamma_{X_\sm}\vert_S = \gamma_S$.
To check that the image is contained in~$\rM_g$ 
we verify in Lemma~\ref{lem:s-x} the injectivity of the restriction morphism
\begin{equation*}
\rH^0(X_\sm, \gamma_{X_\sm}^*\cE_0) \to \rH^0(S, \gamma_S^*\cE_0),
\end{equation*}
which implies that any section of~$\cE_0$ on~$\Gr(r,r+s)$ vanishing on~$\gamma_S(S)$ also vanishes on~$\gamma_{X_\sm}(X_\sm)$.
This defines the required rational map~$\gamma_X \colon X \dashrightarrow \rM_g$.

Before going to higher dimensions, we prove that the rational map~$\gamma_X$ is in fact regular,
hence the the sheaf~$\cU_X$ is locally free on the entire~$X$.
We will explain this part of the argument at the end of Step~\ref{step:extend-description}.
We also note that~\cite[Corollary~5.9]{BKM}
proves that~$\cU_X$ is \emph{exceptional}.

To construct the rational Gushel map~$X \dashrightarrow \Gr(r, r+s)$, 
where~$X$ is a Fano variety of dimension~$n \ge 4$ we use a deformation argument.
Namely, we chose a very general K3 surface linear section~$S \subset X$ 
and consider the family~$\{\cX_u\}_{u \in U}$ (where~$U$ is an open subset of~$\P^{n-3}$) 
of all mildly singular Fano threefold linear sections~$\cX_u = X \cap \P^{g+1}$
of~$X \subset \P^{n + g - 2}$ containing~$S = X \cap \P^g$,
so that
\begin{equation*}
S \subset \cX_u \subset X
\end{equation*}
and each~$\cX_u$ satisfies the assumptions of Theorem~\ref{thm:main}.
For each of these Fano threefolds~$\cX_u$ we consider the corresponding Mukai bundle~$\cU_{\cX_u}$
and, using the exceptionality of~$\cU_{\cX_u}$ and isomorphisms~\mbox{$\cU_{\cX_u}\vert_S \cong \cU_S$},
we show in Lemma~\ref{lem:cu-cx-u} that there is a single bundle on the open subset of~$X$ containing all the~$\cX_u$
whose restriction to each~$\cX_u$ is isomorphic to~$\cU_{\cX_u}$.
Using this vector bundle,
we construct a rational map~$\gamma_X \colon X \dashrightarrow \Gr(r,r+s)$
that coincides with the Gushel morphism on each~$\cX_u$,
and conclude that~$\gamma_X$ factors through~$\rM_g$, see Corollary~\ref{cor:gushel}.

\subsubsection*{Step~\ref{step:extend-description}}

For this step we use a general result.
Its simplified version says that if~$Y \subset \P^N$ is a Cohen--Macaulay variety
and~$\P^{N_2} \hookrightarrow \P^{N_1} \hookrightarrow \P^N$ is a chain of linear projective subspaces 
such that
\begin{itemize}
\item 
$X \subset Y \cap \P^{N_1}$ is a Cohen--Macaulay variety 
with~$X \cap \P^{N_2} = Y \cap \P^{N_2}$
\item 
$\dim(Y \cap \P^{N_2}) = \dim(Y) - N + N_2$ and~$\dim(X) = \dim(Y) - N + N_1$,
\end{itemize}
then
\begin{equation*}
X = Y \cap \P^{N_1}
\end{equation*}
is a dimensionally transverse linear section of~$Y$.
The more general version, Proposition~\ref{prop:s-x}, 
requires~$Y$ to be arithmetically Cohen--Macaulay (see Section~\ref{ss:extension} for a reminder about this property), 
allows the map~$\P^{N_1} \dashrightarrow \P^N$ to be rational,
and proves that~$X$ is an iterated cone over a dimensionally transverse linear section of~$Y$.

To finish the proof of Theorem~\ref{thm:main}, 
we consider~$Y \coloneqq \rM_g \subset \P^{n_g + g - 2}$, $X \subset \P^{n + g - 2}$, and~\mbox{$S \subset \P^g$},
and note that the construction of the Gushel maps implies that the ambient projective spaces 
form a chain~\mbox{$\P^g \hookrightarrow \P^{n + g - 2} \dashrightarrow \P^{n_g + g - 2}$}
such that the assumptions of Proposition~\ref{prop:s-x} are satisfied.
Eventually, this gives us an identification of~$X$ with an iterated cone 
over a dimensionally transverse linear section of~$\rM_g$, as required.

To show that cones are not necessary for~$n \le 4$ we use the terminality assumption for~$X$,
see Lemma~\ref{lem:cone-singularity}.
As a consequence of this, we conclude that in these cases 
the map~$\P^{n + g - 2} \dashrightarrow \P^{n_g + g - 2}$ is injective, 
hence~$X = \rM_g \cap \P^{n + g - 2}$, 
so that the Gushel map~$\gamma_X \colon X \to \rM_g$ is regular and the sheaf~$\cU_X$ is locally free,
which is what was used for the extension of our results to higher dimensions.

\medskip 

The arguments outlined above work uniformly for~$g \in \{8,9,10,12\}$,
but for~$g = 7$ and~$g = 6$ they require a modification, for two different reasons.

If~$g = 7$, the problem appears already at Step~\ref{step:prime-k3}:
the Mukai variety~$\rM_7 = \OGr_+(5,10)$ lives in~$\Gr(5,10)$, 
the corresponding pair of integers~$(r,s) = (5,5)$ 
gives a factorization of the \emph{double polarization} of a prime K3 surface~$(S,H)$ of genus~7,
and it turns out that the construction of Mukai bundles 
used in Proposition~\ref{prop:muk-bundles} and~\cite{BKM} does not work for the double polarization.
Therefore, we act differently.
We use the technique of~\cite{BKM} to construct a pair of intermediate stable vector bundles of rank~2 and~3 on~$S$, 
and then define~$\cU_S$ as their extension, see Theorem~\ref{thm:g7-emb}.
This gives us an embedding~$S \hookrightarrow \rM_7 \cap \P^7$.

Furthermore, in contrast to Lemma~\ref{lem:no-schubert}, 
it is no longer true that~$S$ is not contained in Schubert divisors of the homogeneous variety~$\rM_7$;
in fact, the opposite is always true.
However, we show that~$S$ cannot be contained in Schubert cycles of codimension~2, see Lemma~\ref{lem:no-schubert-g7}.
We also prove Lemma~\ref{lem:yes-schubert-g7}, a variant of Lemma~\ref{lem:yes-schubert}, 
that allows us to bound the singularities of~$\rM_7 \cap \P^7$.
After that, the proof of the equality~$S = \rM_7 \cap \P^7$ goes in the same way as before, see Proposition~\ref{prop:s-g7}.

On the other hand, the case~$g = 7$ has the following major advantage.
Using the equality~\mbox{$S = \rM_7 \cap \P^7$}, one can identify the Mukai bundle~$\cU_S$
with a twisted normal bundle of~$S$ in~$\P^7$, see Corollary~\ref{cor:cw-cn}.
This simplifies Step~\ref{step:extend-gamma} considerably: 
we can define the extension~$\cU_{X}$ of~$\cU_S$ 
to a Fano variety~$X$ as the twisted normal sheaf of~$X \subset \P^{n + 5}$.
This immediately gives us an extension of the embedding~$S \hookrightarrow \rM_7$ 
to a rational map~$\gamma_X \colon X \dashrightarrow \rM_7$.

The last step, Step~\ref{step:extend-description}, for~$g = 7$, is the same as for other genera.

\medskip 

If~$g = 6$ and~$(r,s) = (2,3)$, the construction of the Mukai bundle on~$S$ in Proposition~\ref{prop:muk-bundles} works,
but the embedding~$S \hookrightarrow \rM_6 \cap \P^6 = \Gr(2,5) \cap \P^6$ is not an equality anymore.
We prove that it is a divisorial embedding, 
and that~$S$ is a complete intersection of type~$(1,1,1,2)$ in~$\Gr(2,5)$, see Proposition~\ref{prop:s-g6}.
Yet another special feature of this case is the appearance of cones much earlier---they show up already for smooth threefolds.

\medskip 

Finally, for~$g = 12$, there are some complications with the Mukai varieties~$\rM_{12} \subset \Gr(3,7)$.
They are also not unique, nor homogeneous, and it requires some effort to prove that they have the desired properties.
Some of the arguments for this case are in Appendix~\ref{sec:g-12}.

\subsection{History}\label{sec:History}

The whole story owes its existence to Mukai,
and his ideas and arguments are crucial in our proof.

The idea of realizing K3 surfaces and prime Fano threefolds 
as linear sections of special homogeneous manifolds appeared first (to our knowledge) in~\cite{Muk:CurvesK3}.
Theorems~\ref{thm:main} or~\ref{thm:prime-k3} are not stated there explicitly,
but it is evident that at that point the formulation was clear to Mukai, at least for~$g \in \{6,7,8,9,10\}$.
Instead, the emphasis of this paper is on the uniqueness of the realization.

In the next two papers, \cite{Mukai:PNAS} and~\cite{Mukai:Fano3folds}, 
the results about K3 surfaces and (smooth) Fano manifolds of arbitrary dimension 
appear in full generality (including the case of~$g = 12$), 
see~\cite[Theorems~3 and~4]{Mukai:PNAS} for K3 surfaces 
and~\cite[Theorem~2]{Mukai:PNAS} and~\cite[Section~2]{Mukai:Fano3folds} for Fano manifolds.
Also, it is made clear that the key idea is to construct an appropriate vector bundle (which we now call the Mukai bundle).
In particular, it is explained that to pass from K3 surfaces to Fano manifolds 
one needs to extend this bundle from a K3 surface to the ambient variety.
For this purpose Mukai suggests to use inductively Fujita's extension theorem~\cite{Fujita:extendampledivisor};
however, this argument does not work for the very first step, the extension to threefolds,
for the reasons explained in~\cite{BKM}.

These ideas were further developed in~\cite{Muk:New}.
Here Mukai introduced the Brill--Noether property for K3 surfaces and emphasized its importance.
In particular, here he announced:
\begin{itemize}[wide]
\item 
an extension of Theorem~\ref{thm:prime-k3} 
from prime K3 surfaces to Brill--Noether general K3 surfaces, see~\cite[Theorems~4.7 and~5.5]{Muk:New},
\item 
an extension of Theorem~\ref{thm:main} to Fano threefolds with terminal Gorenstein singularities 
and a Brill--Noether general anticanonical divisor, see~\cite[Theorem~6.5]{Muk:New}, and 
\item 
a possible further extension of Theorem~\ref{thm:main} to Fano threefolds with canonical Gorenstein singularities 
whose anticanonical divisor class does not admit a movable decomposition, see~\cite[Proposition~7.8]{Muk:New}.
\end{itemize}

In this paper Mukai also suggested a replacement for the Fujita's extension theorem (see a discussion following~\cite[Theorem~6.5]{Muk:New}),
and explained that this argument is \emph{not independent} of Theorem~\ref{thm:prime-k3}, but relies on it.
The last, but not the least idea of Mukai, that we want to mention,
is his argument at the very end of~\cite[Section~6]{Muk:New}, 
on which our extension principle from Proposition~\ref{prop:s-x} is modeled.

An approach to the Mukai theorem via degenerations has been developed in~\cite{CLM93,CLM98}.
Their arguments prove the irreducibility of the corresponding moduli spaces
and Theorems~\ref{thm:main} and~\ref{thm:prime-k3} for general Fano varieties and general K3 surfaces,
but not for all of them.

Taking into account the significance of Mukai's results, 
we give a complete argument for Theorems~\ref{thm:main} and~\ref{thm:prime-k3} in this paper and~\cite{BKM},
combining ideas of Mukai 
with some simplifications of our own. These simplifications include
Lemma~\ref{lem:no-schubert} (and its variant Lemma~\ref{lem:no-schubert-g7} for genus~7)
that makes the proof of non-degeneracy in Proposition~\ref{prop:sigma} quite simple, 
and to base the proof of Proposition~\ref{prop:m-cap-p} in part on the general connectness criterion in Proposition~\ref{lem:zl-deg}.

Finally, we want to mention the papers \cite{Muk:CSS, Muk:CurvesGrassmannians, Muk:CSS-I, Muk:CSS-II} 
where Mukai's interest shifts to lower dimension: from Fano varieties and K3 surfaces to canonical curves.
He described explicit conditions, in terms of special linear systems, determining
when a canonical curve~$C$ of genus~$g \in \{6, 7, 8, 9\}$ can be realized as a transverse linear section of~$\rM_g$;
these conditions are implied by the Brill--Noether generality property of~$C$, but are more precise.

\subsection{Open questions}

We do not know how to prove some of the results announced in~\cite{Muk:New}. 
The first is Mukai's extension of Theorem~\ref{thm:prime-k3} to Brill--Noether general K3 surfaces. 
One can show that any K3 surface appearing in the Mukai model contains a Brill--Noether general curve, 
and thus is Brill--Noether general itself; but we do not know how to prove the converse, 
that any Brill--Noether general K3 surface has a description as in Theorem~\ref{thm:prime-k3}.

We also do not know how to replace our factoriality assumption with the assumption 
that the Fano threefold contains a Brill--Noether general K3 as a hyperplane section,
nor how to replace terminal singularities by canonical singularities.

Despite recent progress by Tanaka on the classification of Fano threefolds in characteristic~$p$, 
see in particular~\cite{Tanaka:Fano3folds-II}, the analogue of Theorem~\ref{thm:main} is open. 
One difficulty is the same as in the non-factorial case, namely that the K3 hyperplane sections all have Picard rank at least two.

In~\cite[Theorem~5.4]{Muk:New} Mukai suggested a criterion for smoothness of the Mukai threefold~$\rM_{12}$ 
associated with a non-degenerate net of skew forms, the proof of which we also did not manage to recover.

\subsection{Conventions}\label{ss:conventions}

We work over an algebraically closed field~$\kk$ of characteristic zero.

As mentioned above, we call a smooth K3 surface~$S$ {\sf prime} if~$\Pic(S) \cong \ZZ$. 
When we write~$(S, H)$ for a polarized prime K3 surface, we always assume that~$H$ is the primitive polarization. 

If~$\cU$ is a vector bundle on a scheme~$X$ such that the dual bundle~$\cU^\vee$ 
is globally generated with~$V \coloneqq \rH^0(X,\cU^\vee)^\vee$, 
we usually denote by
\begin{equation*}
V/\cU \coloneqq (V \otimes \cO_X)/\cU \cong \Coker(\cU \xrightarrow{\ \coev\ } V \otimes \cO_X)
\quad\text{and}\quad 
\cU^\perp \coloneqq \Ker(V^\vee \otimes \cO_X \xrightarrow{\ \ev\ } \cU^\vee) \cong (V/\cU)^\vee
\end{equation*}
the quotient bundle and its dual.
Here~$\ev$ and~$\coev$ stand for the evaluation and coevaluation morphism, respectively.

\subsection*{Acknowledgements}
As in our previous paper, it is difficult to overestimate the influence
of Shigeru Mukai's work in our arguments.
We would also like to thank Francesco Denisi, Daniele Faenzi, Soheyla Feyzbakhsh,  
Chunyi Li, Shengxuan Liu, Laurent Manivel, and Yuri Prokhorov for useful discussions.


\section{Preliminaries}\label{sec:MukaiVarieties}

In this section we establish some preliminary results about zero loci of global sections of vector bundles
and about K3 surfaces isomorphic to quadratic sections of smooth projective threefolds.
These results are crucial for the proof of Theorem~\ref{thm:prime-k3}.

\subsection{Zero loci}\label{subsec:ZeroLoci}

For a projective variety~$M \subset \P(V)$ we write~$\cO_M(1) \coloneqq \cO_{\P(V)}(1)\vert_M$.
For a global section~$\sigma \in \rH^0(M,\cE)$ of a vector bundle~$\cE$ on~$M$ we denote by~$\sigma$
the induced morphism~$\cO_M \to \cE$ and, by abuse of notation, also the dual morphism~$\cE^\vee \to \cO_M$.

\begin{lemma}
\label{lem:zl-span}
Let~$M \subset \P(V)$ be a projective variety 
such that the natural restriction morphism~\mbox{$V^\vee \to \rH^0(M, \cO_M(1))$} is an isomorphism.
Let~$\cE$ be a vector bundle on~$M$, 
let~$\sigma \in \rH^0(M,\cE)$ be a global section,
and let~$Y \subset M$ be the zero locus of~$\sigma$.
Consider the morphism~$\sigma \colon \cE^\vee \to \cO_M$, 
its twist~$\sigma(1) \colon \cE^\vee(1) \to \cO_M(1)$,
the induced morphism on global sections
\begin{equation*}
\rH^0(M, \cE^\vee(1)) \xrightarrow{\ \rH^0(\sigma(1))\ } \rH^0(M, \cO_M(1)) = V^\vee,
\end{equation*}
and its dual morphism~$\rH^0(\sigma(1))^\vee \colon V \to \rH^0(M, \cE^\vee(1))^\vee$.
Then
\begin{equation}
\label{eq:zl-ls}
Y \subset M \cap \P(\Ker(V \xrightarrow{\ \rH^0(\sigma(1))^\vee\ } \rH^0(M, \cE^\vee(1))^\vee)).
\end{equation}
If, moreover, $\cE^\vee(1)$ is globally generated then the inclusion~\eqref{eq:zl-ls} is an equality of schemes.
\end{lemma}

\begin{proof}
Let~$W \subset V$ be the kernel of the map~$\rH^0(\sigma(1))^\vee$.
Then we have a right exact sequence
\begin{equation*}
\rH^0(M, \cE^\vee(1)) \to V^\vee \to W^\vee \to 0,
\end{equation*}
which induces a right exact sequence of sheaves
\begin{equation*}
\rH^0(M, \cE^\vee(1)) \otimes \cO_{\P(V)}(-1) \xrightarrow{\ \rH^0(\sigma(1))\ } \cO_{\P(V)} \to \cO_{\P(W)} \to 0
\end{equation*}
on~$\P(V)$.
Restricting it to~$M$, we obtain a right exact sequence
\begin{equation*}
\rH^0(M, \cE^\vee(1)) \otimes \cO_M(-1) \xrightarrow{\ \rH^0(\sigma(1))\ } \cO_M \to \cO_{M \cap \P(W)} \to 0.
\end{equation*}
Note, on the other hand, that the map~$\rH^0(\sigma(1))$ in this sequence factors as
\begin{equation*}
\rH^0(M, \cE^\vee(1)) \otimes \cO_M(-1) \xrightarrow{\ \ev\ } \cE^\vee  \xrightarrow{\ \sigma\ } \cO_M,
\end{equation*}
where the first arrow is the evaluation map.
Since the image of the second arrow is the ideal~$I_Y$, and the image of the composition is~$I_{M \cap \P(W)}$, it follows that~$I_{M \cap \P(W)} \subset I_Y$, i.e., $Y \subset M \cap \P(W)$.

If, moreover, the bundle~$\cE^\vee(1)$ is globally generated, the first arrow is surjective, hence the image of the composition~$I_{M \cap \P(W)}$ is equal to the image~$I_Y$ of the second arrow, i.e., $Y = M \cap \P(W)$.
\end{proof}

\begin{proposition}
\label{lem:zl-deg}
Let~$M$ be a connected Cohen--Macaulay projective variety of dimension~$n$
and let~$\cE$ be a globally generated vector bundle on~$M$.
If
\begin{equation}
\label{eq:vanishing-connectedness}
\rH^{i}(M, \wedge^{j}\cE^\vee) = 0 \qquad\text{for~$j \ge 1$ and~$i \in \{j-1,j\}$}
\end{equation}
then the zero locus of any global section of~$\cE$ is connected.
\end{proposition}

\begin{proof}
Consider the product~$M \times \P(\rH^0(M,\cE))$ and the universal section~$\tilde\sigma$ 
of the vector bundle~\mbox{$\pr_1^*\cE \otimes \pr_2^*\cO(1)$} on it.
Let~$m = \rank(\cE)$.
Since~$\cE$ is globally generated, the sheaf
\begin{equation*}
\cE' \coloneqq \Ker(\rH^0(M,\cE) \otimes \cO \to \cE)
\end{equation*}
is a subbundle in~$\rH^0(M,\cE) \otimes \cO_M$ of corank~$m$
and its projectivization~$\P_M(\cE') \subset M \times \P(\rH^0(M,\cE))$ coincides with the zero locus of~$\tilde\sigma$.
In particular, this zero locus has expected codimension~$m$, and since~$M$ is Cohen--Macaulay,
its structure sheaf has the Koszul resolution
\begin{equation*}
\dots \to 
\pr_1^*\wedge^2\cE^\vee \otimes \pr_2^*\cO(-2) \to 
\pr_1^*\cE^\vee \otimes \pr_2^*\cO(-1) \to 
\cO_{M \times \P(\rH^0(M,\cE))} \to \cO_{\P_M(\cE')} \to 0.
\end{equation*}
Consider the natural projection~$p \coloneqq \pr_2\vert_{\P_M(\cE')} \colon \P_M(\cE') \to \P(\rH^0(M,\cE))$.
Using the the Koszul resolution, the hypercohomology spectral sequence, connectedness of~$M$,
and the vanishing~\eqref{eq:vanishing-connectedness}, we compute
\begin{equation*}
p_*\cO_{\P_M(\cE')} \cong \cO_{\P(\rH^0(M,\cE))}.
\end{equation*}
By the Zariski's connectedness theorem every fiber of~$p$ is connected.
It remains to note that the zero loci of sections of~$\cE$ are the fibers of~$p$.
\end{proof}

\subsection{Brill--Noether property}
\label{subsec:BNproperty}

In this section, we show that the Brill--Noether property of curves and K3 surfaces
prevents them from being complete intersections of certain types.

Recall that a curve~$C$ is called {\sf Brill--Noether general} 
if the dimension of the subset of~$\Pic^d(C)$ parameterizing line bundles~$\cL$ of degree~$d$ with~$h^0(\cL) \ge r$ 
is less or equal than~\mbox{$g - r(g - 1 - d + r)$}; 
moreover, if this number is negative, this subset is empty, and otherwise it is nonempty.

For instance, for Brill--Noether general curves of genus~$g \ge 7$ 
there are neither line bundles~$\cL$ with~$\deg(\cL) = 6$ and~$h^0(\cL) \ge 3$,
nor line bundles~$\cL$ with~$\deg(\cL) = 4$ and~$h^0(\cL) \ge 2$.

\begin{lemma}\label{lem:r-c}
If~$R \subset \P^{g-1}$ is a smooth surface and~$Q \subset \P^{g-1}$ is a quadric hypersurface such that
\begin{equation*}
C = R \cap Q \subset \P^{g-1}
\end{equation*}
is a smooth canonically embedded curve of genus~$g$ then~$R$ is a del Pezzo surface or an elliptic scroll and~$\deg(R) = g - 1$.
Moreover, if~$g \ge 7$ then~$C$ is not Brill--Noether general.
\end{lemma}

A similar result was proved in~\cite[Theorem~1.1]{Cos}.

\begin{proof}
Let~$H \in \Pic(R)$ be the hyperplane class.
By assumption~$K_C = H\vert_C$.
It follows that
\begin{equation*}
2g - 2 = \deg(K_C) = H \cdot C = H \cdot 2H = 2 \deg(R),
\end{equation*}
hence~$\deg(R) = g - 1$.
Further, the adjunction formula gives~$H\vert_C = K_C = (K_R + 2H)\vert_C$, hence
\begin{equation*}
(K_R + H)\vert_C = 0.
\end{equation*}
In particular, since~$C \sim 2H$ is ample, the Hodge index theorem implies that either
\begin{aenumerate}
\item
\label{it:num-triv}
$K_R + H$ is numerically trivial, or
\item
\label{it:neg-square}
$(K_R + H)^2 \le -1$.
\end{aenumerate}

If~\ref{it:num-triv} holds then~$-K_R$ is numerically equivalent to the ample class~$H$; hence it is ample, i.e., $R$ is a del Pezzo surface.

If~\ref{it:neg-square} holds then~$K_R + H$ is not nef, i.e., there is an effective curve class~$\Gamma$ with~\mbox{$(K_R + H) \cdot \Gamma < 0$}.
Therefore, we have
\begin{equation*}
(-K_R) \cdot \Gamma \ge H \cdot \Gamma + 1 \ge 2.
\end{equation*}
In particular, $\Gamma$ is $K$-negative, and therefore we may assume it is primitive and extremal.
Consider the corresponding extremal contraction of~$R$.
Since~$(-K_R) \cdot \Gamma \ge 2$, this is a $\P^1$-bundle;
in particular, we have~\mbox{$K_R \cdot \Gamma = -2$}.
Moreover, the above inequalities imply that~$H \cdot \Gamma = 1$, 
hence~$R \subset \P^{g-1}$ is a (minimal) scroll and~$\Gamma$ is the class of its fiber.
Furthermore, if~$C' \subset R$ is a general curve in~$|H|$
then on the one hand,
\begin{equation*}
2\g(C') - 2 = \deg(K_{C'}) = (K_R + C') \cdot C' = (K_R + H) \cdot H = 0,
\end{equation*}
hence~$\g(C') = 1$,
and on the other hand, since~$C' \cdot \Gamma = H \cdot \Gamma = 1$, 
it follows that the base of the scroll is birational to~$C'$, hence it is also an elliptic curve.
Thus, $R$ is an elliptic scroll.

To prove the second statement note that, in case~\ref{it:num-triv}, where~$R$ is a smooth del Pezzo surface, 
any numerically trivial class is trivial, hence~$H = -K_R$ and so~$C$ is a bi-anticanonical curve on~$R$.
If~$R$ can be realized as the blowup of~$\P^2$, it follows that~$C$ has a 2-dimensional linear system of degree~$6$, and otherwise (if~$R \cong \P^1 \times \P^1$) it has a pencil of degree~$4$. 
Similarly, in case~\ref{it:neg-square}, it follows that the curve~$C$ is bielliptic,
hence it also has a pencil of degree~$4$.

In either case, we conclude that~$C$ is not Brill--Noether general when~$g \ge 7$.
\end{proof}

Note that a Brill--Noether general curve of genus~$g \le 6$ 
has a 2-dimensional linear system of degree~$6$ and a pencil of degree~$4$
(and for~$g \in \{4,5,6\}$ it can be realized as an intersection of a del Pezzo surface with a quadric),
so the assumption~$g \ge 7$ in the previous lemma is sharp.

\begin{corollary}\label{cor:y-cap-q}
Let~$Y \subset \P^g$ be an irreducible variety of dimension~$3$ 
such that~$S \coloneqq Y \cap Q$ is a smooth~$K3$ surface of genus~$g \ge 7$, 
where~$Q \subset \P^g$ is a quadric hypersurface.
Then a general hyperplane section of~$S$ is not Brill--Noether general.
\end{corollary}

\begin{proof}
Since~$S$ is a smooth surface, it follows that~$Y$ is smooth along~$S$, and since~$Q$ is an ample divisor, it follows that~$Y$ is a threefold with isolated singularities.
Let~$C \subset S$ be a general curve in~$|H|$ and let~$R \coloneqq Y \cap \P^{g-1} \subset Y$ be a general hyperplane section such that
\begin{equation*}
R \cap Q = Y \cap \P^{g-1} \cap Q = S \cap \P^{g-1} = C.
\end{equation*}
Note that~$R$ is a smooth surface by Bertini's Theorem and~$C$ is a canonically embedded curve.
Therefore, Lemma~\ref{lem:r-c} implies that~$C$ is not Brill--Noether general.
\end{proof}


\section{Gushel embeddings for prime K3 surfaces}\label{sec:GushelK3}

Recall that the genus~$g  =\g(S,H)$ of a polarized K3 surface~$(S,H)$ is defined by the equality
\begin{equation*}
H^2 = 2g - 2.
\end{equation*}
We say that a smooth polarized K3 surface~$(S,H)$ is {\sf prime} if~$\Pic(S) = \ZZ \cdot H$.

Given a sheaf~$\cF$ on~$S$ we write
\begin{equation*}
\rv(\cF) = (\rank(\cF), \rc_1(\cF), \rch_2(\cF) + \rank(\cF)) \in \ZZ \oplus \Pic(S) \oplus \ZZ
\end{equation*}
for its Mukai vector.
If~$\rv(\cF) = (r, D, s)$, then by Riemann--Roch we have
\begin{equation}
\label{eq:chi}
\upchi(\cF) = r + s
\qquad\text{and}\qquad 
\upchi(\cF,\cF) = 2rs - D^2.
\end{equation} 

\subsection{Construction of the Gushel map}\label{subsec:ConstructionGushel}

We start by recalling some results from~\cite{BKM}.

\begin{definition}[{\cite[Definition~3.1]{BKM}}]\label{def:mb}
If~$(S,H)$ is a polarized K3 surface of genus~\mbox{$g = r \cdot s$}, 
{\sf a Mukai bundle of type~$(r,s)$ on~$(S,H)$} is an $H$-stable vector bundle~$\cU_S$ with Mukai vector
\begin{equation}
\label{eq:muk-vec}
\rv(\cU_S) = (r,-H,s),
\end{equation}
i.e., $\rank(\cU_S) = r$, $\rc_1(\cU_S) = -H$, and~$\rch_2(\cU_S) = s - r$, or equivalently, $\upchi(\cU_S) = r + s$.
\end{definition}

The uniqueness of Mukai bundle is easy to verify.
The existence of an $H$-\emph{semistable} bundle with Mukai vector~\eqref{eq:muk-vec} 
follows from general existence results \cite{Kuleshov:spherical, Yoshioka:Abelian}, 
and for prime K3 surfaces this bundle is $H$-stable, hence it is a Mukai bundle.
However, to control its properties it is useful to construct it 
(as we did in~\cite{BKM} for Brill--Noether general K3 surfaces) 
as the Lazarsfeld bundle of a Brill--Noether-extremal line bundle on a curve in~$|H|$.
The following special case of~\cite[Theorem~3.4]{BKM} gives such a construction in the prime case. 

The notation~$\cU_S^\perp$ and~$V/\cU_S$ is introduced in Section~\ref{ss:conventions}.

\begin{proposition}
\label{prop:muk-bundles}
For any prime $K3$ surface~$(S,H)$ of genus~$g =r \cdot s$ there is a unique Mukai bundle~$\cU_S$ of type~$(r,s)$.
We have~$\rH^{\ge 1}(S, \cU_S^\vee) = 0$ and~$\cU_S^\vee$ is globally generated by the space
\begin{equation}
\label{eq:def-v}
V \coloneqq \rH^0(S, \cU_S^\vee)^\vee;
\end{equation}
of dimension~$r + s$, so that there are mutually dual exact sequences
\begin{equation}
\label{eq:taut}
0 \to \cU_S^\perp \xrightarrow{\quad} V^\vee \otimes \cO_S \xrightarrow{\ \ev\ } \cU_S^\vee \to 0
\qquad\text{and}\qquad 
0 \to \cU_S \xrightarrow{\quad} V \otimes \cO_S \xrightarrow{\ \ev\ } V/\cU_S \to 0,
\end{equation}
where~$\ev$ are the evaluation morphisms.
Moreover, $\cU_S^\perp$ is a Mukai bundle of type~$(s,r)$.
\end{proposition}

\begin{proof} 
As we mentioned above, the uniqueness is standard (see, e.g.,~\cite[Lemma~3.2]{BKM}).

Recall that by~\cite[Theorem and Corollary~1.4]{Laz} 
every general curve~$C \subset S$ in~$|H|$ is Brill--Noether--Petri general.
Therefore, as we checked in~\cite[Lemma~2.3]{BKM}, there is a pair of globally generated line bundles~$(\xi, \eta)$ on~$C$ 
of degree~$(r-1)(s+1)$ and~$(r+1)(s-1)$, respectively, with
\begin{equation*}
\xi \otimes \eta = \omega_C,
\qquad
h^0(\xi) = h^1(\eta) = r
\qquad\text{and}\qquad
h^0(\eta) = h^1(\xi) = s.
\end{equation*}
We define~$\cU_S$ by the exact sequence
\begin{equation} 
\label{eq:Lazbundle}
0 \to \cU_S \xrightarrow{\quad} \rH^0(C, \xi) \otimes \cO_S \xrightarrow{\ \ev\ } j_* \xi \to 0,
\end{equation}
where~$j \colon C \hookrightarrow S$ is the embedding.
The Auslander--Buchsbaum formula shows that~$\cU_S$ is a vector bundle with Mukai vector~\eqref{eq:muk-vec}.
It also follows that the dual bundle fits into the exact sequence
\begin{equation}
\label{eq:Lazbundle-dual}
0 \to \rH^0(C, \xi)^\vee \otimes \cO_S \xrightarrow{\quad} \cU_S^\vee \xrightarrow{\ \ev\ } j_* \eta \to 0.
\end{equation}
Since $\rH^2(S, \cU_S^\vee) = \rH^0(S, \cU_S)^\vee = 0$ by construction, the long exact cohomology sequence of \eqref{eq:Lazbundle-dual} shows that~$\cU_S^\vee$  is globally generated with~$\rH^1(S, \cU_S^\vee) = \rH^2(S, \cU_S^\vee) = 0$
and~$h^0(\cU_S^\vee) = r + s$.
Therefore, the space~\eqref{eq:def-v} has dimension~$r + s$,
and exact sequences~\eqref{eq:taut} also follow.

The slope of~$\cU_S$ is~$-1/r$; this is the maximal possible negative slope 
for bundles of rank~$r$ or less (because~$\Pic(S) = \ZZ \cdot H$), 
so if~$\cU_S$ is unstable, a destabilizing subsheaf~$\cF \subset \cU_S$ must have non-negative slope.
But~$\cU_S$ is a subsheaf of a trivial bundle, 
hence~\cite[Lemma~3.7(a)]{BKM} implies~$\cF \cong \cO_S^{\oplus m}$, 
hence~$\rH^0(S, \cU_S) \ne 0$ in contradiction with the defining sequence~\eqref{eq:Lazbundle}.
This proves stability of~$\cU_S$, and stability of~$\cU_S^\perp$ is proved similarly.
\end{proof}

Note that the construction of Proposition~\ref{prop:muk-bundles} depends on  choices (the curve~$C$ and the line bundle~$\xi$);
however, the resulting bundle~$\cU_S$ is unique, hence independent of them.

If a Mukai bundle~$\cU_S$ on a polarized K3 surface~$(S,H)$ exists 
and its dual~$\cU_S^\vee$ is globally generated 
(e.g., in the setup of Proposition~\ref{prop:muk-bundles} or in the more general setup of~\cite[Theorem~3.4]{BKM}), 
there is a morphism
\begin{equation}
\label{def:gushel}
\gamma_S \colon S \to \Gr(r, V) = \Gr(r, r + s)
\end{equation}
such that~$\gamma_S^*(\cU) \cong \cU_S$, where~$\cU$ is the tautological bundle on~$\Gr(r,V)$.
Taking the determinant of the above isomorphism, we see that~
\begin{equation*}
\gamma_S^*(\cO_{\Gr(r,V)}(1)) \cong \cO_S(H).
\end{equation*}
Following the terminology introduced in~\cite{DK}, we  call~$\gamma_S$ {\sf the Gushel morphism}.

The following result shows that the Gushel morphism~$\gamma_S$ is essentially unique.

\begin{lemma}
\label{lem:all-embeddings} 
Let~$(S,H)$ be a prime $K3$ surface of genus~$g \ge 4$.
Let~$g = r \cdot s$ be a factorization with~$r,s \ge 2$ 
and let~$f \colon S \to \Gr(r, r+s)$ be a morphism such that~$f^*(\cO_{\Gr(r,r+s)}(1)) \cong \cO_S(H)$
and~$f(S)$ does not lie in the sub-Grassmannians
\begin{equation*}
\Gr(r,r+s-1) \subset \Gr(r,r+s)
\qquad\text{or}\qquad
\Gr(r-1,r+s-1) \subset \Gr(r,r+s).
\end{equation*}
Then~$f^*(\cU)$ is the Mukai bundle of type~$(r,s)$ 
and there is~$\alpha \in \Aut(\Gr(r,r+s))$ such that~$f = \alpha \circ \gamma_S$.
\end{lemma}

\begin{proof}
Consider the pullback to~$S$ of the tautological sequence of~$\Gr(r,r+s)$: 
\begin{equation}
\label{eq:pb-taut}
0 \to f^*\cU \to \cO_S^{\oplus (r + s)} \to f^*\cQ \to 0.
\end{equation}
We claim that the bundles~$f^*\cU$ and~$f^*\cQ$ are stable.
Indeed, the argument of Proposition~\ref{prop:muk-bundles} shows that a destabilizing subsheaf~$\cF \subset f^*\cU$ must be trivial, i.e., $\cF \cong \cO_S^{\oplus m}$, but then the composition of an embedding~$\cO_S \hookrightarrow \cF \hookrightarrow f^*\cU_S \hookrightarrow \cO_S^{\oplus (r+s)}$ with the epimorphism~$\cO_S^{\oplus (r+s)} \twoheadrightarrow f^*\cQ$ vanishes, which means that~$f(S)$ is contained in a sub-Grassmannian~$\Gr(r-1,r+s-1) \subset \Gr(r,r+s)$, in contradiction to our assumptions.
Thus, the bundle~$f^*\cU$ is stable.
Considering analogously the dual sequence of~\eqref{eq:pb-taut}, we check that~$f^*\cQ^\vee$ is stable, hence so is the bundle~$f^*\cQ$.

By using~$f^*\cO_{{\Gr(r,r+s)}}(1) \cong \cO_S(H)$ and~\eqref{eq:pb-taut}, we write the Mukai vectors of~$f^*\cU$ and~$f^*\cQ$ as
\begin{equation*}
\rv(f^*\cU) = (r,-H,s'),
\qquad 
\rv(f^*\cQ) = (s,H,r'),
\end{equation*}
with~$r' + s' = r + s$.
Since~$f^*\cU$ is stable, we have~$\upchi(f^*\cU, f^*\cU) \le 2$, which by~\eqref{eq:chi} implies that~\mbox{$s' {}\le{} s$}.
Similarly, the stability of~$f^*\cQ$ implies that~$r' {}\le{} r$.
Combining these inequalities with the equality~\mbox{$r' + s' = r + s$}, 
we conclude that~$r' = r$ and~$s' = s$; in particular, $\rv(f^*\cU) = \rv(\cU_S)$, 
hence~$f^*\cU$ is a Mukai bundle, and by the uniqueness of the latter, we conclude that~$f^*\cU \cong \cU_S$.

Finally, the composition
\begin{equation*}
\kk^{\oplus (r + s)} \to \rH^0(S, f^*\cU^\vee) \cong \rH^0(S, \cU_S^\vee) = V^\vee,
\end{equation*}
where the first arrow is induced by the dual of the first arrow in~\eqref{eq:pb-taut}, 
is injective because~$f(S)$ is not contained in a sub-Grassmannian.
Hence, this composition is an isomorphism, 
and its dual~$\alpha \colon V \to \kk^{r + s}$ satisfies~$f = \alpha \circ \gamma_S$.
\end{proof}

\subsection{Injectivity and Schubert divisors}

In this section we show that the Gushel morphism is a closed embedding and its image is not contained in Schubert divisors.
Recall the notation~\eqref{eq:def-v}.

\begin{lemma}
\label{lem:gamma-x-s}
Let~$(S,H)$ be a prime $K3$ surface of genus~$g \ge 4$.
Let~$g = r \cdot s$ be a factorization with~$r,s \ge 2$ 
and let~$\cU_S$ be the corresponding Mukai bundle on~$S$.
The Gushel morphism~\mbox{$\gamma_S \colon S \to \Gr(r,V)$} is a closed embedding 
and~$\gamma_S(S) \subset \Gr(r,V) \cap \P^g$ in~$\P(\wedge^rV)$.
\end{lemma}

\begin{proof}
Taking the $r$-th exterior power of the evaluation epimorphism~$V^\vee \otimes \cO_S \twoheadrightarrow \cU_S^\vee$ 
from~\eqref{eq:taut}, we obtain an epimorphism
\begin{equation*}
\wedge^rV^\vee \otimes \cO_S \twoheadrightarrow \cO_S(H).
\end{equation*}
We claim that the induced linear map~$\phi \colon \wedge^rV^\vee \to \rH^0(S, \cO_S(H))$ of global sections is surjective.

For this we use the construction of~$\cU_S$ in \eqref{eq:Lazbundle}, for a smooth curve~$C \in |H|$. 
The sequence~\eqref{eq:Lazbundle-dual} induces an embedding~$\rH^0(C, \xi)^\vee \subset V^\vee$ 
and shows that the image of the line
\begin{equation*}
	\wedge^r(\rH^0(C, \xi))^\vee \subset \wedge^rV^\vee
\end{equation*}
in~$\rH^0(S, \cO_S(H))$ under the map~$\phi$ is generated by an equation of the curve~$C$.
Since~$C \in |H|$ can be any general curve, 
we conclude that the map~$\phi \colon \wedge^rV^\vee \to \rH^0(S, \cO_S(H))$ is dominant, 
and since~$\phi$ is linear, it is surjective.

Now consider the natural commutative diagram
\begin{equation*}
\vcenter{\xymatrix@C=3em{
S \ar[d]_{|H|} \ar[r]^-{\gamma_S} &
\Gr(r,V) \ar@{^{(}->}[d]^{\text{Pl\"ucker}}
\\
\P^g \ar@{^{(}->}[r]^-{\phi^\vee} &
\P(\wedge^rV)
}}
\end{equation*}
The bottom arrow is regular and injective, because~$\phi$ is surjective, and the left vertical arrow is a closed embedding as $H$ is very ample on $S$ for~$g \ge 3$, by~\cite[Theorem~3.1 and Theorem~5.2]{SD}. 
Therefore, the Gushel morphism~$\gamma_S$ has to be a closed embedding, 
and the diagram shows that~$\gamma_S(S) \subset \Gr(r,V) \cap \P^g$.
\end{proof}

In the next lemma, which is crucial for our proof of Theorem~\ref{thm:prime-k3} in the next section, 
we show that the property~$\Pic(S) = \ZZ \cdot H$ prevents~$\gamma_S(S) \subset \Gr(r,V)$ from being contained in a Schubert divisor. 
Recall that a {\sf Schubert divisor} on~$\Gr(r,V) = \Gr(r,r+s)$ is defined as
\begin{equation*}
\Sigma_1(V_s) \coloneqq \{ U_r \subset V \mid \dim(U_r \cap V_s) \ge 1\}
= \Gr(r,V) \cap \P(V_s \wedge (\wedge^{r-1} V)),
\end{equation*}
where~$V_s \subset V$ is a subspace of dimension~$s$. Alternatively,
$\Sigma_1(V_s)$ can be defined as the intersection of~$\Gr(r, V) \subset \P(\wedge^rV)$ 
with the hyperplane defined by the decomposable $r$-form 
\begin{equation*}
f_1 \wedge \dots \wedge f_r \in \wedge^rV^\vee, 
\end{equation*}
where~$f_1,\dots,f_r$ is a basis of the annihilator~$V_s^\perp \subset V^\vee$ of~$V_s$.

\begin{remark}
\label{rem:schubert-sing}
The singular locus of the Schubert divisor~$\Sigma_1(V_s) \subset \Gr(r, r+s)$ can be described as
\begin{equation}
\label{eq:schubert-sing}
\Sing(\Sigma_1(V_s)) = \{ U_r \subset V \mid \dim(U_r \cap V_s) \ge 2\}
= \Gr(r,V) \cap \P(\wedge^2V_s \wedge (\wedge^{r-2} V)).
\end{equation}
If~$r,s \ge 2$ it is nonempty and has codimension~$4$ in~$\Gr(r, r+s)$.
\end{remark}

\begin{lemma}\label{lem:no-schubert}
Let~$(S,H)$ be a prime $K3$ surface of genus~$g = r \cdot s \ge 4$ with~$r,s \ge 2$.
Then the image~$\gamma_S(S) \subset \Gr(r,V)$ of~$S$ under the Gushel embedding is not contained in any Schubert divisor.
\end{lemma}

\begin{proof}
If~$S$ is contained in a Schubert divisor~$\Sigma_1(V_s)$ 
then the morphism~$\phi \colon V_s \otimes \cO_S \to V/\cU_S$ of vector bundles of rank~$s$ is everywhere degenerate.
In other words, if~$\cF \coloneqq \Ima(\phi)$ denotes its image, then~$\rank(\cF) \le s - 1$.
If~$\rc_1(\cF) \ge H$, then we have the following inequalities for the slopes
\begin{equation*}
\upmu(\cF) \ge \frac1{s-1} > \frac1s = \upmu(V/\cU_S),
\end{equation*}
in contradiction to the stability of~$V/\cU_S \cong (\cU_S^\perp)^\vee$ proved in Proposition~\ref{prop:muk-bundles}.
Thus, $\rc_1(\cF) \le 0$, and since~$\cF$ is torsion free, it follows that~$\cF \cong \cO_S^{\oplus m}$ with~$m \le s - 1$, 
see~\cite[Lemma~3.7(b)]{BKM}. This is in contradiction to the injectivity of the morphism
\begin{equation*}
\rH^0(S, V_s \otimes \cO_S) = V_s \hookrightarrow V = \rH^0(S, V/\cU_S)
\end{equation*}
induced by~\eqref{eq:taut}.
This proves the lemma.
\end{proof}

We showed in Lemma~\ref{lem:gamma-x-s} that~$\gamma_S(S) \subset \Gr(r,V) \cap \P^g$.
Next, we prove a general observation about linear sections of Grassmannians, which is particularly useful in combination with Lemma~\ref{lem:no-schubert}.

\begin{lemma}\label{lem:yes-schubert}
Let~$Y \subset \Gr(r,V) \cap \P^g$ be a subscheme of a linear section of the Grassmannian,
where~$\dim(V) = r + s$ and~$g = r \cdot s$.
If~$r,s \ge 2$ and~$y \in Y$ is a point such that~$\dim(\rT_{y,Y}) \ge 4$ then~$Y$ is contained in a Schubert divisor, 
i.e., there are inclusions~$Y \subset \Sigma_1(V_s) \subset \Gr(r,V)$ for some~$V_s \subset V$.
\end{lemma}

\begin{proof}
Let~$U_r \subset V$ be the subspace corresponding to the point~$y \in Y$.
Consider the filtration
\begin{equation*}
\wedge^rU_r \subset \wedge^{r-1}U_r \wedge V \subset \wedge^rV
\end{equation*}
that corresponds to the point~$y$ and the embedded tangent space to~$\Gr(r,V) \subset \P(\wedge^rV)$ at~$y$.
Furthermore, let~$W \subset \wedge^rV$ be a vector subspace of dimension~$g + 1$ 
such that~$Y \subset \Gr(r,V) \cap \P(W)$ and consider the induced filtration~$W_0 \subset W_1 \subset W$, where
\begin{equation*}
W_0 = \wedge^rU_r,
\qquad\text{and}\qquad 
W_1 = W \cap (\wedge^{r-1}U_r \wedge V).
\end{equation*}
Then~$\rT_{y,Y} = W_1/W_0$; hence~$\dim(W_1) \ge 5$ by assumption, and therefore~$\dim(W/W_1) \le g - 4$.
It follows that
\begin{equation}
\label{eq:dim-sum}
\dim\Big((\wedge^{r-1}U_r \wedge V) + W\Big) \le \dim (\wedge^{r-1}U_r \wedge V) + g - 4.
\end{equation}

Consider the scheme
\begin{equation*}
Z_y \coloneqq \Sing(\Sigma_1(U_r)) \subset \Gr(s, V) = \Gr(r, V^\vee) \subset \P(\wedge^rV^\vee).
\end{equation*}
By~Remark~\ref{rem:schubert-sing}, it has codimension~4 in~$\Gr(r, V^\vee)$, hence~$\dim(Z_y) = r \cdot s - 4 = g - 4$.

On the other hand, we have~$Z_y \subset \P(\wedge^2 U_r^\perp \wedge (\wedge^{r-2}V^\vee))$ by~\eqref{eq:schubert-sing}, and since~$\wedge^2 U_r^\perp \wedge (\wedge^{r-2}V^\vee)$ is the annihilator of~$\wedge^{r-1}U_r \wedge V$, it follows from~\eqref{eq:dim-sum} that~$Z_y \cap \P(W^\perp)$ is a linear section of~$Z_y$ of codimension at most~$g - 4$.
Therefore, $Z_y \cap \P(W^\perp) \ne \varnothing$.
It remains to note that this intersection parameterizes Schubert divisors 
which are singular at~$y$ and whose linear span contains~$\P(W)$ (and which 
therefore contain~$Y$).
\end{proof}

\begin{remark}
The same argument proves that if~$Y \subset \Gr(r,V) \cap \P^{g + k}$ for some~$k \ge 0$
and~$y \in Y$ is a point such that~$\dim(\rT_{y,Y}) \ge 4 + k$ 
then~$Y$ is contained in a Schubert divisor.
\end{remark}

In Section~\ref{sec:Genus6-12} we will also need the following observation 
about the pullback under the Gushel morphism of the conormal bundle of the Grassmannian.
In contrast with the other statements of this section, this one is only true for special values of~$g$ and~$(r,s)$.

\begin{lemma}\label{lem:mukai-conormal} 
Let~$S$ be a prime $K3$ surface~$S$ of genus~$g = r \cdot s$ such that
\begin{equation}
\label{eq:rs-inequality}
\frac1r + \frac1s > \frac12.
\end{equation} 
If~$\gamma_S \colon S \hookrightarrow \Gr(r,V)$ is the Gushel embedding 
then~$\rH^0(S, \gamma_S^*(\cN^\vee_{\Gr(r,V)/\P(\wedge^rV)}(1))) = 0$.
\end{lemma}

\begin{proof}
The pullback~\eqref{eq:taut} of the tautological exact sequence of~$\Gr(r,V)$ to~$S$ 
induces a filtration of the trivial bundle~$\wedge^rV \otimes \cO_S$ with factors
\begin{equation*}
\wedge^r\cU_S,\qquad 
\wedge^{r-1}\cU_S \otimes V/\cU_S,\qquad 
\wedge^{r-2}\cU_S \otimes \wedge^2(V/\cU_S),\qquad 
\dots,\qquad 
\wedge^r(V/\cU_S).
\end{equation*}
The first factor corresponds to the pullback of~$\cO_{\P(\wedge^rV)}(-1)$ under the Pl\"ucker embedding,
hence the remaining factors provide a filtration of~$\gamma_S^*(\cT_{\P(\wedge^rV)}(-1))$.
Similarly, the second factor is the twisted tangent bundle~$\cT_{\Gr(r,V)}(-1)$ 
and its embedding is given by the differential of~$\gamma_S$,
hence the remaining factors provide a filtration of~$\gamma_S^*(\cN_{\Gr(r,V)/\P(\wedge^rV)}(-1))$.
Dualizing, we conclude that the bundle~$\gamma_S^*(\cN^\vee_{\Gr(r,V)/\P(\wedge^rV)}(1))$ has a filtration with factors
\begin{equation*}
\wedge^r\cU_S^\perp, \qquad \dots, \qquad \wedge^{r-2}\cU^\vee_S \otimes \wedge^2\cU_S^\perp.
\end{equation*}

Since~$\cU_S^\vee$ and~$\cU_S^\perp$ are stable of slope~$\tfrac1r$ and~$-\tfrac1s$, respectively,
all these sheaves are semistable with slopes~$(r - i) \cdot \tfrac1r - i \cdot \tfrac1s = 1 - i \cdot (\tfrac1r + \tfrac1s)$, 
where~$r \ge i \ge 2$; 
in particular, the assumption~\eqref{eq:rs-inequality} implies that all the slopes are negative.
We conclude that the bundles have no global sections, 
and hence the same is true for the bundle~$\gamma_S^*(\cN^\vee_{\Gr(r,V)/\P(\wedge^rV)}(1))$.
\end{proof}

\begin{remark}
For~$g \in \{8,9,10,12\}$ and~$(r,s)$ as in~\eqref{eq:rs} the inequality~\eqref{eq:rs-inequality} obviously holds.
It also holds 
for~$r = 2$ and any~$s$, or for~$(r, s) = (3, 5)$. 
\end{remark}


\section{Prime K3 surfaces of genus~6, 8, 9, 10, or 12}
\label{sec:Genus6-12}

In this section we prove Theorem~\ref{thm:prime-k3} for~$g \in \{6,8,9,10,12\}$.
In Section~\ref{subsec:MukaiVarieties8-12} we recall the realization of Mukai varieties~$\rM_g$ inside~$\Gr(r,r+s)$,
in Section~\ref{subsec:ContainmentMukai} we show that the image of~$S$ under the Gushel morphism is contained in~$\rM_g$,
and in Section~\ref{subsec:K3g=6-10} we prove that~$S = \rM_g \cap \P^g$.

Throughout the section we fix the pair of integers~$(r,s)$ as in~\eqref{eq:rs}, so that~$g = r \cdot s$.

\subsection{Mukai varieties of genus 9, 10, or 12}
\label{subsec:MukaiVarieties8-12}

Mukai varieties~$\rM_g$ for~$g \in \{9,10,12\}$ were defined in the introduction.
In this section we describe~$\rM_g \subset \Gr(r,r+s)$ 
as zero loci of global sections of vector bundles.
Note that an embedding of~$\rM_g$ into~$\Gr(r,r+s)$ is not unique,
but when it is fixed, we call its image a {\sf Mukai subvariety} of~$\Gr(r, r+s)$.

We denote by~$V = V_{r+s}$ the vector space of dimension~$r + s$, 
by~$\cU$ the tautological bundle on the Grassmannian~$\Gr(r,V) = \Gr(r, r+s)$, 
and by~$V/\cU$ the quotient bundle.
We also recall the notation~$\cU^\perp \cong (V/\cU)^\vee$
and consider the vector bundle
\begin{equation}
\label{eq:def-ce-zero}
\cE_0 \coloneqq 
\begin{cases}
\wedge^2\cU^\vee, & 
\text{on~$\Gr(3,V_6)$ for~$g = 9$},\\
\cU^\perp(1), & 
\text{on~$\Gr(2,V_7)$ for~$g = 10$},\\
\wedge^2\cU^\vee \oplus \wedge^2\cU^\vee \oplus \wedge^2\cU^\vee, & 
\text{on~$\Gr(3,V_7)$ for~$g = 12$}.
\end{cases}
\end{equation}
Note that
\begin{equation*}
\rH^0(\Gr(r,V_{r+s}), \cE_0) =
\begin{cases}
\wedge^2V_6^\vee, & \text{for~$g = 9$},\\
\wedge^3V_7^\vee, & \text{for~$g = 10$},\\
\wedge^2V_7^\vee \oplus \wedge^2V_7^\vee \oplus \wedge^2V_7^\vee, & \text{for~$g = 12$}.
\end{cases}
\end{equation*}
Recall that the~$\GL(V_7)$-action on the space~$\wedge^3V_7^\vee$ 
has an open orbit (\cite[Theorem~24(8)]{KS}).
We will usually impose the following non-degeneracy conditions on a global section of~$\cE_0$:

\begin{definition}
\label{def:nondegenerate}
A section~$\sigma_0 \in \rH^0(\Gr(r,r+s), \cE_0)$ 
of one of the bundles~\eqref{eq:def-ce-zero} is {\sf nondegenerate}, if
\begin{itemize}
\item
for~$g = 9\hphantom{0}$,
$\sigma_0 \in \wedge^2V_6^\vee$ is a $2$-form of rank~$6$;
\item 
for~$g = 10$,
$\sigma_0 \in \wedge^3V_7^\vee$ is a $3$-form in the open~$\GL(V_7)$-orbit in~$\wedge^3V_7^\vee$;
\item 
for~$g = 12$,
$\sigma_0 = (\sigma_{0,1},\sigma_{0,2},\sigma_{0,3}) \in \wedge^2V_7^\vee \oplus \wedge^2V_7^\vee \oplus \wedge^2V_7^\vee$
satisfies the property
\begin{equation}
\label{eq:rank-6}
\rank(a_1 \sigma_{0,1} + a_2 \sigma_{0,2} + a_3 \sigma_{0,3}) = 6
\qquad\text{for every~$(a_1,a_2,a_3) \in \kk^3 \setminus 0$}.
\end{equation}
\end{itemize}
\end{definition}

Recall the integers~$n_g$ listed in~\eqref{eq:ng}.

\begin{lemma}
\label{lem:rm-9-10}
For~$g \in \{9,10,12\}$ let~$\sigma_0 \in \rH^0(\Gr(r,V), \cE_0)$ be a nondegenerate global section.
Then the zero locus of~$\sigma_0$ is a Mukai subvariety~$\rM_g \subset \Gr(r,V)$.
In particular, 
\begin{itemize}[wide]
\item 
for~$g = 9$ and~$g = 10$
these are smooth homogeneous Fano varieties,
\item 
for~$g = 12$ these are local complete intersection Fano threefolds.
\end{itemize}
In all cases, $\dim(\rM_g) = n_g$, $\deg(\rM_g) = 2g - 2$, $K_{\rM_g} = -(n_g - 2)H$, $\rH^\bullet(\rM_g, \cO_{\rM_g}) = \kk$,
and the restriction morphism~$\rH^0(\Gr(r,V), \cO_{\Gr(r,V)}(1)) \to \rH^0(\rM_g, \cO_{\rM_g}(1))$ is surjective.
\end{lemma}

\begin{proof}
For~$g = 9$ the identification of the Mukai variety~$\rM_9 = \LGr(3,6) \subset \Gr(3,6)$ 
as a zero locus is classical, so we omit the proof.
The computation of the dimension, canonical class, and cohomology of the structure sheaf are immediate,
and the degree is computed in Lemma~\ref{lem:ctop}.
Finally, since~$\rH^0(\rM_9,\cO_{\rM_9}(1))$ is an irreducible representation of the group~$\SP_6$
and the restriction morphism is equivariant and non-zero, it is surjective.

For~$g = 10$ the description of~$\rM_{10} = \rG_2/\rP_2$ as a zero locus was observed by Mukai, \cite[Remark~1]{Mukai:PNAS}; 
we give here a sketch of the proof for the reader's convenience.

Assume~$V$ is a vector space of dimension~7 and let~$\sigma_0 \in \wedge^3V^\vee$ be a 3-form from the open~$\GL(V)$-orbit.
The stabilizer of~$\sigma_0$ in~$\GL(V)$ is a simple algebraic group of type~$\rG_2$, see~\cite[Example~30]{KS}, and~$V$ is its irreducible representation.
Since~$(\cU^\perp(1))^{\vee} \otimes \cO(1) \cong V/\cU$ is globally generated by the space~$\rH^0(\Gr(2,V), V/\cU) = V$, Lemma~\ref{lem:zl-span} shows that the zero locus~$\rZ(\sigma_0) \subset \Gr(2,V)$ of~$\sigma_0$ is a linear section of~$\Gr(2,V)$; more precisely,
\begin{equation*}
\rZ(\sigma_0) = \Gr(2,V) \cap \P(\Ker(\wedge^2V \xrightarrow{\ \rH^0(\sigma_0(1))^\vee\ } V^\vee)).
\end{equation*}
From the representation theory of the group~$\rG_2$ it is easy to compute 
that~$\Ker(\wedge^2V \xrightarrow{\ \rH^0(\sigma_0(1))^\vee\ } V^\vee)$ 
is the adjoint representation of~$\rG_2$, in particular, it is irreducible.
Hence it contains a unique closed $\rG_2$-orbit---the adjoint Grassmannian~$\rG_2/\rP_2$.
On the other hand, the zero locus~$\rZ(\sigma_0)$ is obviously closed and $\rG_2$-invariant, 
hence~$\rG_2/\rP_2 \subset \rZ(\sigma_0)$.
Moreover, any connected component of~$\rZ(\sigma_0)$ must contain a closed $\rG_2$-orbit, hence~$\rZ(\sigma_0)$ is connected.
Finally, we have~$\dim(\rG_2/\rP_2) = 5$,
and at the same time by Bertini's Theorem for zero loci of sections of globally generated vector bundles, 
the scheme~$\rZ(\sigma_0)$ is smooth of dimension~$\dim(\rZ(\sigma_0)) = 5$.
Therefore, $\rG_2/\rP_2 = \rZ(\sigma_0)$.

Other statements in this case are proved in the same way as for~$g = 9$.

For~$g = 12$ all results are proved in Proposition~\ref{prop:v22}.
\end{proof}

Next, we show that the~$\rM_g$ are (non-transverse) linear sections of the Grassmannians.

\begin{lemma}
\label{lem:xi0}
For~$g \in \{9,10,12\}$ let~$\rM_g \subset \Gr(r,V)$ be a Mukai subvariety of genus~$g$,
i.e., the zero locus of a nondegenerate global section~$\sigma_0$ of the bundle~$\cE_0$ from~\eqref{eq:def-ce-zero}.
Then the sequence
\begin{equation}
\label{eq:h0-sequence}
0 \to \rH^0(\Gr(r,V), \cE_0^\vee(1)) \xrightarrow{\ \rH^0(\sigma_0(1))\ } 
\rH^0(\Gr(r,V), \cO(1)) \xrightarrow{\quad\qquad}
\rH^0(\rM_g, \cO_{\rM_g}(1)) \to 0,
\end{equation}
with the first map defined in Lemma~\textup{\ref{lem:zl-span}} is exact.
Explicitly, the first map has the following form:
\begin{aenumerate}
\item 
\label{it:rh0-s0-g9}
if~$g = 9$, so that~$\sigma_0 \in \wedge^2V^\vee$, this map is
\begin{equation*}
V^\vee \to \wedge^3V^\vee,
\qquad 
f \mapsto f \wedge \sigma_0,
\end{equation*}
\item 
\label{it:rh0-s0-g10}
if~$g = 10$, so that~$\sigma_0 \in \wedge^3V^\vee$, this map is
\begin{equation*}
V \to \wedge^2V^\vee,
\qquad
v \mapsto \sigma_0(v,-,-),
\end{equation*}
\item 
\label{it:rh0-s0-g12}
if~$g = 12$, so that~$\sigma_0 = (\sigma_{0,1},\sigma_{0,2},\sigma_{0,3}) \in \wedge^2V^\vee \oplus \wedge^2V^\vee \oplus \wedge^2V^\vee$, 
this map is
\begin{equation*}
(V^\vee)^{\oplus 3} \to \wedge^3V^\vee,
\qquad 
(f_1,f_2,f_3) \mapsto f_1 \wedge \sigma_{0,1} + f_2 \wedge \sigma_{0,2} + f_3 \wedge \sigma_{0,3}.
\end{equation*}
\end{aenumerate}
Moreover, $h^0(\cO_{\rM_g}(1)) = n_g + g - 1$
and~$\rM_g = \Gr(r,V) \cap \P^{n_g + g - 2}$.
\end{lemma}

\begin{proof}
The bundle~$\cE_0^\vee(1)$ has the following explicit form:
\begin{align}
\label{eq:ce0-vee}
\cE_0^\vee(1) &\cong 
\begin{cases}
\cU^\vee, & \text{for~$g = 9$},\\
V/\cU, & \text{for~$g = 10$},\\
\cU^\vee \oplus \cU^\vee \oplus \cU^\vee, & \text{for~$g = 12$},
\end{cases}
\intertext{in particular, it is globally generated.
Moreover, we have}
\label{eq:h0-ce0-vee}
\rH^0(\Gr(r,V), \cE_0^\vee(1)) &\cong 
\begin{cases}
V^\vee, & \text{for~$g = 9$},\\
V, & \text{for~$g = 10$},\\
V^\vee \oplus V^\vee \oplus V^\vee, & \text{for~$g = 12$}.
\end{cases}
\end{align}
We prove injectivity of~$\rH^0(\sigma_0(1))$ by a case-by-case analysis.

\ref{it:rh0-s0-g9}
Let~$g = 9$.
The formula for the map~$\rH^0(\sigma_0(1))$ is obvious and its injectivity 
can be easily deduced from non-degeneracy of~$\sigma_0$.
For instance, it follows from~$\SP_6$-equivariance of the morphism, 
because~$V^\vee$ is an irreducible representation of~$\SP_6$.

\ref{it:rh0-s0-g10}
Let~$g = 10$.
The formula for~$\rH^0(\sigma_0(1))$ is obvious and its injectivity can be easily deduced from non-degeneracy of~$\sigma_0$.
For instance, it follows from~$\rG_2$-equivariance of the morphism, because~$V$ is an irreducible representation of~$\rG_2$.

\ref{it:rh0-s0-g12}
Let~$g = 12$.
The formula  for~$\rH^0(\sigma_0(1))$ is analogous to part~\ref{it:rh0-s0-g9} 
and its injectivity is proved in Proposition~\ref{prop:v22}.

It remains to note that for all~$g$ the sequence~\eqref{eq:h0-sequence} is exact in the middle by Lemma~\ref{lem:zl-span},
on the right by Lemma~\ref{lem:rm-9-10}, and on the left by the above analysis,
so a straightforward computation gives the required formula for~$h^0(\cO_{\rM_g}(1))$.
The equality~$\rM_g = \Gr(r,V) \cap \P^{n_g+g-2}$ follows from Lemma~\ref{lem:zl-span},
because~$\cE_0^\vee(1)$ is globally generated.
\end{proof}

\subsection{Containment in the Mukai subvarieties}
\label{subsec:ContainmentMukai}

The goal of this section is to show that the Gushel embedding~$\gamma_S \colon S \to \Gr(r,r+s)$ 
of a prime K3 surface~$S$ of genus~\mbox{$g \in \{9,10,12\}$} factors through a Mukai subvariety~$\rM_g$.
Our proof relies on the following cohomology vanishing.

\begin{lemma} 
\label{lem:g10H1vanishing}
Let~$(S,H)$ be a prime $K3$ surface of genus~$g \ge 2$ 
and let~$\cF$ be a stable vector bundle on~$S$ with Mukai vector~$\rv(\cF) = (r_{\cF},-H,s_{\cF})$, where~$r_{\cF} \ge 2$.
Assume also
\begin{equation} \label{ineq:s+g}
s_{\cF} + g > 
\begin{cases}
\tfrac43g - 1, & \text{if~$r_{\cF} = 2$},\\
\tfrac{r_{\cF}+1}{2(r_{\cF}-1)}g, & \text{if~$r_{\cF} \ge 3$}.
\end{cases}
\end{equation}
Then~$\rH^1(S, \cF(H)) = \rH^2(S, \cF(H)) = 0$.
In particular, for~$r_{\cF} \ge 3$ this holds for any positive~$s_{\cF}$.
\end{lemma}

\begin{proof}
The vanishing~$h^2(\cF(H)) = 0$ follows immediately from Serre duality and stability of~$\cF$.
Assume~$h^1(\cF(H)) \ne 0$.
Then by Serre duality there is a non-split extension
\begin{equation}
\label{eq:tcf}
0 \to \cO_S(-H) \to \tcF \to \cF \to 0.
\end{equation}
We have~$\rv(\tcF) = (r_{\cF}+1,-2H,s_{\cF}+g)$ and Riemann--Roch implies that
\begin{equation*}
\upchi(\tcF,\tcF) = 2(r_{\cF}+1)(s_{\cF}+g) - 4(2g - 2) > 2,
\end{equation*}
where the inequality is a consequence of~\eqref{ineq:s+g}: 
for~$r_{\cF} = 2$ this is obvious, and for~$r_{\cF} \ge 3$ this follows from~$(r_{\cF}+1)^2 \ge 8(r_{\cF}-1)$. 
Therefore, $\tcF$ cannot be stable.

By stability of~$\cF$, every subsheaf of~$\tcF$ has negative slope, 
and thus the same holds for all Jordan--H\"older (JH) factors  
in a refinement of the Harder--Narasimhan filtration of~$\tcF$ with respect to slope stability. 
But the first Chern classes of the JH factors have to sum up to~$\rc_1(\tcF) = -2H$, 
hence there can only be two JH factors. 
In other words, there is a short exact sequence
\begin{equation*}
0 \to \cF_1 \to \tcF \to \cF_2 \to 0
\end{equation*}
with~$\cF_i$ torsion-free and slope-stable, $\rv(\cF_i) = (r_i, -H, s_i)$ 
and~$-\frac 1{r_1} \ge -\frac 2{r_{\cF}+1} \ge -\frac 1{r_2}$.

In particular, we have~$r_1 \ge 2$. 
If~$r_2 = 1$, then either the composition~$\cO_{S}(-H) \hookrightarrow \tcF \twoheadrightarrow \cF_2$ is injective, 
therefore an isomorphism (because~$\rc_1(\cF_2) = -H$ and~$\cF_2$ is torsion-free), 
in contradiction to~\eqref{eq:tcf} being non-split; 
or it is zero, in which case~$\cF_1/\cO_S(-H)$ is a subsheaf of~$\cF$ of slope zero, 
in contradiction to stability of~$\cF$. Therefore, $r_2 \ge 2$.

This immediately implies~$r_{\cF} = r_1 + r_2 -1 \ge 3$. 
By stability of~$\cF_i$ we have~$r_i s_i \le g$. 
This leads to a contradiction to the assumption~\eqref{ineq:s+g} because
\begin{equation*}
s_{\cF} + g = s_1 + s_2 \le 
g \cdot \left(\tfrac1{r_1} + \tfrac1{r_2}\right) \le
g \cdot \left(\tfrac1{r_{\cF}-1} + \tfrac12\right) = 
\tfrac{r_{\cF}+1}{2(r_{\cF}-1)}g
\end{equation*}
where the second inequality follows from convexity of the function~\mbox{$\frac 1x + \frac 1{r_{\cF}+1 - x}$}
since~$r_1 + r_2 = r_{\cF} + 1$ and~$r_i \ge 2$.
\end{proof}

Recall that the Gushel morphism~$\gamma_S \colon S \to \Gr(r,r+s)$ was defined in~\eqref{def:gushel}.
It is characterized by the property~$\gamma_S^*\cU \cong \cU_S$, 
where~$\cU$ is the tautological bundle on~$\Gr(r,r+s)$ and~$\cU_S$ is the Mukai bundle~$\cU_S$ on~$S$.
If~$S$ is prime, we proved in Lemma~\ref{lem:gamma-x-s} that~$\gamma_S$ is a closed embedding.

\begin{proposition}
\label{prop:sigma}
Let~$S$ be a prime $K3$ surface of genus~$g  \in \{9,10,12\}$
and let~$(r,s)$ be as in~\eqref{eq:rs}.
Then the Gushel embedding~$\gamma_S$ 
factors through a unique Mukai subvariety~$\rM_g \subset \Gr(r, r+s)$.
\end{proposition}

\begin{proof}
Recall from Lemma~\ref{lem:rm-9-10} that each Mukai subvariety~$\rM_g \subset \Gr(r,r+s) = \Gr(r,V)$
is the zero locus of a nondegenerate (in the sense of Definition~\ref{def:nondegenerate}) global section~$\sigma_0$
of the vector bundle~$\cE_0$ defined in~\eqref{eq:def-ce-zero}. 
So, we have to check that there is a unique 
(in case of $g = 12$, unique up to the obvious $\GL_3$-action) section~$\sigma_0$ of~$\cE_0$ whose pullback to~$S$ vanishes, 
and that~$\sigma_0$ is nondegenerate.

Assume~$g = 9$ or $g = 12$, so that~$(r,s) = (3,g/3)$.
Then~$\wedge^2\cU_S^\vee \cong \cU_S(H)$ and 
\begin{equation*}
\rH^1(S, \cU_S(H)) = \rH^2(S, \cU_S(H)) = 0
\end{equation*}
by Lemma~\ref{lem:g10H1vanishing}.
On the other hand, \eqref{eq:muk-vec} implies that~$\rv(\cU_S(H)) = (3, 2H, 4g/3 - 1)$, 
and a combination of~\eqref{eq:chi} with the above vanishings proves that
\begin{equation*}
h^0(\wedge^2\cU_S^\vee) = h^0(\cU_S(H)) = \upchi(\cU_S(H)) = 4g/3 + 2.
\end{equation*}

For~$g = 9$, we obtain~$h^0(\wedge^2\cU_S^\vee) = 14$, whereas $\rH^0(\Gr(3,V), \wedge^2\cU^\vee) = \wedge^2V^\vee$ has dimension~$15$; therefore a nontrivial global section~$\sigma_0$ of~$\cE_0 = \wedge^2\cU^\vee$ vanishes on~$\gamma_S(S)$.

If~$\sigma_0$ is degenerate, it can be written as~$\sigma_0 = f_1 \wedge f_2$ or~$\sigma_0 = f_1 \wedge f_2 + f_3 \wedge f_4$,
where~$f_i$ are appropriate linearly independent elements of~$V^\vee$.
In either case, there is~$f_3 \in V^\vee$ 
such that~$\sigma_0 \wedge f_3 = f_1 \wedge f_2 \wedge f_3\neq 0$ is nonzero and decomposable.
By Lemma~\ref{lem:xi0}\ref{it:rh0-s0-g9} the zero locus of~$\sigma_0$ 
is contained in the Schubert divisor associated to $f_1 \wedge f_2 \wedge f_3$, 
in contradiction to Lemma~\ref{lem:no-schubert}.

For~$g = 12$, we obtain~$h^0(\wedge^2\cU_S^\vee) = 18$, whereas $\rH^0(\Gr(3,V), \wedge^2\cU^\vee) = \wedge^2V^\vee$ has dimension~$21$; therefore three linearly independent global sections~$\sigma_{0,1},\sigma_{0,2},\sigma_{0,3}$ of~$\wedge^2\cU^\vee$
vanish on~$\gamma_S(S)$.

If~$\sigma_0 = (\sigma_{0,1},\sigma_{0,2},\sigma_{0,3})$ is degenerate (in the sense of Definition~\ref{def:nondegenerate})
then a non-zero linear combination of~$\sigma_{0,i}$ has rank~2 or~4.
Then the above computation shows that its wedge product with an appropriate linear function on~$V$
is a decomposable 3-form, hence by Lemma~\ref{lem:xi0}\ref{it:rh0-s0-g12} 
the zero locus of~$\sigma_0$ is contained in a Schubert divisor, 
in contradiction to Lemma~\ref{lem:no-schubert}.

Finally, for $g= 10$, Lemma~\ref{lem:g10H1vanishing} implies $h^0(\cU_S^\perp(H)) = \upchi(\cU_S^\perp(H)) = 34$.
On the other hand, the space~$\rH^0(\Gr(2,V), \cU^\perp(1)) = \wedge^3V^\vee$ has dimension~$35$; 
therefore a nontrivial section~$\sigma_0$ of~$\cU^\perp(1)$ vanishes on~$\gamma_S(S)$.

If~$\sigma_0$ is degenerate, i.e., not in the open $\GL(V)$-orbit, 
then there is~$v \in V$ such that~$\sigma_0(v) \in \wedge^2V^\vee$ is non-zero and decomposable.
Indeed, this can be checked by an inspection of representatives of all other~$\GL(V)$-orbits of 3-forms 
that are listed in~\cite[Section~3]{KW}: one can take~$v$ to be the second base vector for orbits~$1$, $2$, $3$, $5$, $6$, and~$7$, 
and the third base vector for the orbits~$4$ and~$8$.
Therefore, by Lemma~\ref{lem:xi0}\ref{it:rh0-s0-g10} the zero locus of any~$\sigma_0$ not from the open orbit 
is contained in the Schubert divisor associated to $\sigma_0(v)$, 
in contradiction to Lemma~\ref{lem:no-schubert}.

To see that~$\sigma_0$ is unique for~$g \in \{9,10\}$ it is enough to note 
that any nontrivial pencil of sections of~$\cE_0$ contains a nonzero degenerate section.
Similarly, the linear span of~$\sigma_{0,i}$ is unique for~$g = 12$
because any 4-dimensional subspace in~$\wedge^2V^\vee$ contains a form of rank~$\le 4$.
\end{proof}

\subsection{K3 surfaces of genus 6, 8, 9, 10, or 12}
\label{subsec:K3g=6-10}

We are ready to start proving Theorem~\ref{thm:prime-k3}.
The case~$g = 6$ is easy.
Recall that~$(r,s) = (2,3)$, so the target of the Gushel morphism is~$\Gr(2,5)$.

\begin{proposition}\label{prop:s-g6}
If~$g = 6$ then~$\gamma_S(S) = Y \cap Q$ is a complete intersection, 
where~$Y = \Gr(2,5) \cap \P^6$ is a smooth quintic del Pezzo threefold and~$Q$ is a quadric.
\end{proposition}

\begin{proof}
Since~$g = 6$ and the Pl\"ucker space is~$\P(\wedge^2V_5) = \P^9$, 
Lemma~\ref{lem:gamma-x-s} implies that~$S$ is contained in the intersection of~$\Gr(2,5)$ 
with a linear subspace~$\P^6 \subset \P^9$ of codimension~$3$.
Moreover, Lemma~\ref{lem:no-schubert} shows that every hyperplane in~$\P(\wedge^2V_5)$ through~$S$ 
corresponds to a 2-form of rank~4, hence (see~\cite[Proposition~2.24]{DK}) the intersection
\begin{equation*}
Y \coloneqq \Gr(2,5) \cap \P^6
\end{equation*}
is a smooth quintic del Pezzo threefold, and~$S \subset Y$ is a divisor of degree~$2g - 2 = 10$.
Since~$\Pic(Y)$ is generated by the hyperplane class, it follows that~$S = Y \cap Q$ is an intersection of~$Y$ with a quadric~$Q \subset \P^6$.
\end{proof}

The argument for~$g \in \{8,9,10,12\}$ is more complicated.
To treat all these cases unifromly,
it is convenient to consider the sum of the bundle~$\cE_0$ with a few copies of~$\cO(1)$ and define
\begin{equation}
\label{eq:def-ce}
\cE \coloneqq \cO_{\Gr(r,V)}(1)^{\oplus (n_g - 2)} \oplus \cE_0 \cong
\begin{cases}
\cO_{\Gr(2,V)}(1)^{\oplus 6}, & \text{for~$g = 8$},\\
\cO_{\Gr(3,V)}(1)^{\oplus 4} \oplus \wedge^2\cU^\vee, & \text{for~$g = 9$},\\
\cO_{\Gr(2,V)}(1)^{\oplus 3} \oplus \cU^\perp(1), & \text{for~$g = 10$},\\
\cO_{\Gr(3,V)}(1)^{\hphantom{\oplus 4}} \oplus \wedge^2\cU^\vee \oplus \wedge^2\cU^\vee \oplus \wedge^2\cU^\vee, & \text{for~$g = 12$},
\end{cases}
\end{equation}
where~$n_g = \dim(\rM_g)$ is listed in~\eqref{eq:ng}.
We will need the following elementary observation:
\begin{equation}
\label{eq:det-ce}
\rank(\cE) = g - 2,
\qquad 
\rc_1(\cE) = r + s,
\qquad 
\rc_{g-2}(\cE)\cdot c_1(\cO(1))^2 = 2g - 2.
\end{equation} 
Indeed, the first two equalities are obvious, and the third is checked in Lemma~\ref{lem:ctop}.

We can now extend the result of Lemma~\ref{lem:xi0} from Mukai varieties to K3 surfaces.

\begin{corollary}
\label{cor:s-mukai}
Let~$S$ be a prime $K3$ surface of genus~$g  \in \{8,9,10,12\}$
and let~$(r,s)$ be as in~\eqref{eq:rs}.
There is a global section~$\sigma = (\lambda,\sigma_0) \in \rH^0(\Gr(r,V), \cE)$ vanishing on~$S$
and such that the zero locus of~$\sigma_0 \in \rH^0(\Gr(r,V), \cE_0)$ is a Mukai subvariety~$\rM_g \subset \Gr(r,V)$ 
and the sequence
\begin{equation}
\label{eq:xi}
0 \to 
\rH^0(\Gr(r,V), \cE^\vee(1)) \xrightarrow{\ \rH^0(\sigma(1))\ } 
\rH^0(\Gr(r,V), \cO(1)) \xrightarrow{\qquad}
\rH^0(S, \cO_S(H)) \to 
0
\end{equation}
is exact.
In particular, $S \subset \rM_g \cap \P^g$.
\end{corollary}

\begin{proof}
We have~$h^0(\cO_{\rM_g}(1)) = n_g + g - 1$ by Lemma~\ref{lem:xi0} and~$h^0(\cO_S(H)) = g + 1$ by Riemann--Roch.
Therefore, there are linearly independent global sections
\begin{equation*}
\lambda_1,\dots,\lambda_{n_g-2} \in \rH^0(\rM_g, \cO_{\rM_g}(1))
\end{equation*}
that vanish on~$S$.
Lifting them to global sections of~$\cO_{\Gr(r,V)}(1)$ (for~$g \ne 8$ this is possible by Lemma~\ref{lem:rm-9-10}), 
we obtain the required extension~\mbox{$\sigma \coloneqq (\lambda,\sigma_0)$} of~$\sigma_0$.
The exact sequence~\eqref{eq:xi} follows for~$g \ne 8$ from the exact sequence~\eqref{eq:h0-sequence}, 
Lemma~\ref{lem:gamma-x-s}, and the construction of~$\lambda = (\lambda_1,\dots,\lambda_{n_g-2})$.
The inclusion~$S \subset \rM_g \cap \P^g$ is obvious from the vanishing of~$\sigma$ on~$S$ and its definition.
\end{proof}

\begin{proposition}
\label{prop:m-cap-p}
Let~$(S,H)$ be a prime $K3$ surface of genus~$g \in \{8,9,10,12\}$
and let~$\sigma$ be the global section of~$\cE$ constructed in Corollary~\textup{\ref{cor:s-mukai}}.
Then~$\gamma_S(S)$ is equal to the zero locus of~$\sigma$.
Moreover, $\gamma_S(S) = \rM_g \cap \P^g$ is a transverse intersection,
where~$\rM_g \subset \Gr(r,r+s)$ is the Mukai subvariety constructed in Proposition~\textup{\ref{prop:sigma}}.
\end{proposition}

\begin{proof}
Let~$Y \subset \Gr(r,V)$ denote the zero locus of~$\sigma$.
Since~$\cE^\vee(1)$ is globally generated, 
it follows from Lemma~\ref{lem:zl-span} and exact sequence~\eqref{eq:xi} that
\begin{equation*}
Y = \Gr(r,V) \cap \P(\Ker(\rH^0(\sigma(1))^\vee)) = \Gr(r,V) \cap \P^{g}
\end{equation*}
Furthermore, by construction of~$\sigma$ and Lemma~\ref{lem:rm-9-10}, we have~$Y = \rM_g \cap \P^g$,
and since~$S \subset Y$ by Corollary~\ref{cor:s-mukai}, while~$\rM_g, \P^g \subset \P^{n_g + g - 2}$ with
\begin{equation*}
\dim(\P^{n_g + g - 2}) - \dim(\P^g) = 
(n_g + g - 2) - g = 
n_g - 2 = 
\dim(\rM_g) - \dim(S), 
\end{equation*}
it remains to show that~$S = Y$.

On the one hand, since~$S \subset Y$, it follows from Lemma~\ref{lem:no-schubert} that~$Y$ 
is not contained in a Schubert divisor of~$\Gr(r,V)$, hence Lemma~\ref{lem:yes-schubert} implies 
\begin{equation*}
\dim(\rT_{y,Y}) \le 3
\end{equation*}
for any point~$y \in Y$.
Therefore, the union~$Y'$ of components of~$Y$ of dimension greater than~$2$ is a smooth threefold 
and the other components of~$Y$ do not intersect~$Y'$.
On the other hand, a combination of Corollary~\ref{cor:bbw} and Proposition~\ref{lem:zl-deg} shows that~$Y$ is connected. 
Thus, either
\begin{aenumerate}
\item 
\label{it:dim2}
$\dim(Y) \le 2$, or
\item 
\label{it:dim3}
$Y$ is a smooth connected threefold.
\end{aenumerate}
We consider these cases separately.

In case~\ref{it:dim2},
since~$\rank(\cE) = g - 2$ (by~\eqref{eq:det-ce}), while~$\dim(\Gr(r,V)) = r \cdot s = g$,  
the zero locus~\mbox{$Y \subset \Gr(r,V)$} is a Cohen--Macaulay surface of degree~$2g - 2$ (again by~\eqref{eq:det-ce}).
As it contains the surface~$S$, also of degree~$2g - 2$, we conclude that~$S = Y = \Gr(2,V) \cap \P^g = \rM_g \cap \P^g$, as required.

So, it remains to show that the case~\ref{it:dim3}, where~$Y$ is a smooth threefold, is impossible.
In this case the codifferential morphism~$\rd\sigma \colon \gamma_S^*(\cE^\vee) \to \cN^\vee_{S/\Gr(r,V)}$
everywhere has corank~$1$, hence its kernel and cokernel are line bundles, 
and since~$\rc_1(\gamma_S^*\cE) = (r + s)H = \rc_1(\cN_{S/\Gr(r,V)})$ by~\eqref{eq:det-ce},
it follows that there is an integer~$t \in \ZZ$ such that
\begin{equation*}
\Ker(\rd\sigma) \cong \cO_S(-tH) \cong \Coker(\rd\sigma).
\end{equation*}
Since~$\gamma_S^*(\cE^\vee)$ is the direct sum of semistable sheaves of negative slope 
(by~\eqref{eq:def-ce} and Proposition~\ref{prop:muk-bundles}), we have~$t \ge 1$.
On the other hand, since~$\Pic(S) = \ZZ \cdot H$ and~$\g(S,H) \ge 5$, 
the surface~$S \subset \P^g$ is an intersection of quadrics (this follows immediately from~\cite[Theorem~7.2]{SD}), 
the sheaf~$\cN^\vee_{S/\Gr(r,V)}(2H)$ is globally generated, hence~$t \le 2$.
Thus, we have~\mbox{$t \in \{1,2\}$}, and it remains to show that both cases are impossible.

First, assume~$t = 1$, i.e., $\Ker(\rd\sigma) \cong \cO_S(-H) \subset \gamma_S^*(\cE^\vee)$.
Note that
\begin{equation*}
\rH^0(S, \gamma_S^*(\cE^\vee)(H)) = \rH^0(\Gr(r,V), \cE^\vee(1)).
\end{equation*}
Indeed, it is enough to check the equality for the direct summands~$\cE_0$ and~$\cO(1)$ of~$\cE$; 
for~$\cO(1)$ the equality is obvious
and for~$\cE_0$ it follows from the description~\eqref{eq:ce0-vee}, 
computation~\eqref{eq:h0-ce0-vee} of the right side, and Proposition~\ref{prop:muk-bundles} computing the left side.
Thus, the embedding~$\cO_S(-H) \hookrightarrow \gamma_S^*(\cE^\vee)$ 
corresponds to a global section of the bundle~$\cE^\vee(1)$ on~$\Gr(r,V)$
and by~\eqref{eq:xi} it corresponds to a hyperplane in~$\P(\wedge^rV)$ containing~$S$;
in particular it gives a nonzero global section of the bundle~$\cN^\vee_{S / \P(\wedge^rV)}(H)$.
Now, consider the natural exact sequence
\begin{equation*}
0 \to
\gamma_S^*(\cN_{\Gr(r,V)/\P(\wedge^rV)}^\vee(1)) \to
\cN_{S/\P(\wedge^rV)}^\vee(H) \to
\cN_{S/\Gr(r,V)}^\vee(H) \to
0.
\end{equation*}
Since the composition~$\cO_S(-H) = \Ker(\rd\sigma) \hookrightarrow \gamma_S^*(\cE^\vee) \xrightarrow{\ \rd\sigma\ } \cN^\vee_{S/\Gr(r,V)}$ vanishes,
the corresponding hyperplane is tangent to~$\Gr(r,V)$ along~$S$, 
hence the image of the corresponding section of~$\cN_{S/\P(\wedge^rV)}^\vee(H)$ in~$\rH^0(S,\cN_{S/\Gr(r,V)}^\vee(H))$ vanishes,
hence this section comes from a nonzero global section of the bundle~$\gamma_S^*(\cN_{\Gr(r,V)/\P(\wedge^rV)}^\vee(1))$,
in contradiction to Lemma~\ref{lem:mukai-conormal}.

Now assume~$t = 2$, i.e., $\Coker(\rd\sigma) \cong \cO_S(-2H)$.
Consider the composition of epimorphisms
\begin{equation*}
\cO_S(-2H)^{\oplus N} \twoheadrightarrow \cN^\vee_{S/\Gr(r,V)} \twoheadrightarrow \cO_S(-2H),
\end{equation*}
where the first arrow is given by the space of quadrics cutting out~$S$ in~$\Gr(r,V)$, 
and the second is the cokernel of~$\rd\sigma$.
Choose a summand in the source that maps isomorphically onto the target 
and let~$Q \subset \P(\wedge^rV)$ be the corresponding quadric.
By Nakayama lemma the morphism
\begin{equation*}
\cE^\vee \oplus \cO_{\Gr(r,V)}(-2) \xrightarrow{\ (\sigma,Q)\ } I_S
\end{equation*}
is surjective along~$S$, hence~$Y \cap Q$ contains~$S$ as a connected component.
Since~$Q$ is ample and~$Y$ is connected, we see that~$S = Y \cap Q$.
But then Corollary~\ref{cor:y-cap-q} implies that a general hyperplane section of~$S$ is not Brill--Noether general, in contradiction to the assumption~$\Pic(S) = \ZZ \cdot H$ and Lazarsfeld's result \cite[Corollary 1.4]{Laz}.
This completes the proof of the proposition.
\end{proof}

We combine the above results to prove Theorem~\ref{thm:prime-k3} for~$g \ne 7$.

\begin{proof}[Proof of Theorem~\textup{\ref{thm:prime-k3}} for~$g \in \{6,8,9,10,12\}$]
For~$g \in \{8,9,10,12\}$ an embedding~$S \hookrightarrow \rM_g$ 
is constructed in~\eqref{def:gushel} and Proposition~\ref{prop:sigma}
and an isomorphism of~$S$ with a linear section of~$\rM_g$ is proved in Proposition~\ref{prop:m-cap-p}.
The uniqueness follows from the uniqueness of the Gushel morphism proved in Lemma~\ref{lem:all-embeddings} 
and the uniqueness of the section~$\sigma_0$ proved in Proposition~\ref{prop:sigma}.

For~$g = 6$ we apply Proposition~\ref{prop:s-g6} and use the uniqueness of the Gushel morphism.
\end{proof}


\section{Prime K3 surfaces of genus 7}
\label{sec:genus-7}

In this section we treat the case of prime K3 surfaces~$S$ of genus~7.
As we explained in the introduction, in this case the approach of~\cite{BKM} 
does not allow us to construct a Mukai vector bundle, so we use a different approach.
In Section~\ref{ss:lp-g7} we construct and study a pair of Lazarsfeld bundles,
in Section~\ref{ss:mb-g7} we construct the Mukai bundle as an extension of these 
and check that it induces a morphism of~$S$ into the corresponding Mukai variety~$\rM_7$.
Finally, in Section~\ref{ss:ls-g7} we prove that the morphism~$S \to \rM_7$ 
identifies~$S$ with~$\rM_7 \cap \P^7$.

\subsection{Lazarsfeld pair}
\label{ss:lp-g7}

Let~$(S,H)$ be a prime K3 surface of genus~$\g(S,H) = 7$. 
Let~$C \in |H|$ be a general curve; then~$C$ is Brill--Noether--Petri general by~\cite[Theorem]{Laz}.
Therefore, since~\mbox{$\g(C) = 7 > 2 \cdot 3$}, there exists a line bundle~$\xi$ on~$C$ such that
\begin{equation}
\label{eq:cl5}
\deg(\xi) = 5,
\qquad 
h^0(\xi) = 2,
\qquad 
h^1(\xi) = 3.
\end{equation}
On the other hand, since~$\g(C)$ is less than~$2 \cdot 4$ and~$3 \cdot 3$, 
the curve~$C$ has no line bundles of degree~$4$ with~$h^0 = 2$, nor line bundles of degree~$6$ with~\mbox{$h^0 = 3$}, 
hence~$\xi$ and its adjoint line bundle
\begin{equation}
\label{eq:cl7}
\eta \coloneqq \xi^{-1}(K_C)
\end{equation}
are both globally generated.
Note that~$\deg(\eta) = 7$, $h^0(\eta) = 3$, and~$h^1(\eta) = 2$, by Serre duality.

Denoting the embedding~$C \hookrightarrow S$ by~$j$, we see that 
the Lazarsfeld bundles~$\cR_2$ and~$\cR_3$ defined by the following exact sequences
\begin{align}
\label{eq:bls-l5}
0 \to \cR_2 \xrightarrow{\quad} \rH^0(C,\xi) \otimes \cO_S \xrightarrow{\ \ev\ } j_*\xi \to 0,
\\ 
\label{eq:bls-l7}
0 \to \cR_3 \xrightarrow{\quad} \rH^0(C,\eta) \otimes \cO_S \xrightarrow{\ \ev\ } j_*\eta \to 0,
\end{align}
have Mukai vectors~$\rv(\cR_2) = (2,-H,3)$ and~$\rv(\cR_3) = (3,-H,2)$.

\begin{lemma}
\label{lem:laz-cl57}
Let~$S$ be a prime $K3$ surface of genus~$7$, let~$C \subset S$ 
be a smooth Brill--Noether--Petri general curve of genus~$7$ on~$S$, 
let~$\xi$ be a line bundle on~$C$ satisfying~\eqref{eq:cl5}, 
and let~$\eta$ be its adjoint line bundle as in~\eqref{eq:cl7}.
Then the Lazarsfeld bundles~$\cR_2$ and~$\cR_3$ are stable and there are exact sequences
\begin{align}
\label{eq:R2-sub}
0 \to \cR_2 \to \mathmakebox[0pt][l]{V_5}\hphantom{V_5^\vee} \otimes \cO_S \to \cR_3^\vee \to 0,
\\ \label{eq:R3-sub}
0 \to \cR_3 \to V_5^\vee \otimes \cO_S \to \cR_2^\vee \to 0,
\end{align}
where~$V_5 \coloneqq \rH^0(S, \cR_3^\vee) \cong \rH^0(S, \cR_2^\vee)^\vee$, 
while~$\rH^{>0}(S, \cR_3^\vee) = \rH^{>0}(S, \cR_2^\vee) = 0$.
Moreover, we have
\begin{equation}
\label{eq:ext-cr-cr}
h^0(\cR_i^\vee \otimes \cR_i) = 1,
\quad 
h^1(\cR_i^\vee \otimes \cR_i) = 2,
\quad 
h^2(\cR_i^\vee \otimes \cR_i) = 1,
\qquad\text{for~$i = 2,3$.}
\end{equation} 
\end{lemma}

\begin{proof}
It is immediate from~\eqref{eq:cl5} and exact sequences~\eqref{eq:bls-l5} and~\eqref{eq:bls-l7} that
\begin{equation*}
h^i(\cR_2) = h^i(\cR_3) = 0
\qquad\text{for~$i \le 1$}
\qquad\text{and}\qquad 
h^2(\cR_2) = h^2(\cR_3) = 5.
\end{equation*}
Using Serre duality, we compute the cohomology of~$\cR_2^\vee$ and~$\cR_3^\vee$.
The proof of~\eqref{eq:R2-sub} and~\eqref{eq:R3-sub} is analogous to the proof of~\cite[Lemma~3.3]{BKM}.

Furthermore, since~$\Pic(S) = \ZZ \cdot H$, the slopes~$\upmu(\cR_2) = -1/2$ and~$\upmu(\cR_3) = -1/3$ 
are the maximal negative slopes for sheaves of rank at most~2 and~3, respectively.
Therefore, the argument of Proposition~\ref{prop:muk-bundles} proves stability of~$\cR_2$ and~$\cR_3$.

Finally, for~$i = 2,3$ we have~$h^0(\cR_i^\vee \otimes \cR_i) = 1$ by stability, 
hence~$h^2(\cR_i^\vee \otimes \cR_i) = 1$ by Serre duality,
and since~$\upchi(\cR_i,\cR_i) = 0$ by~\eqref{eq:chi}, we conclude that~$h^1(\cR_i^\vee \otimes \cR_i) = 2$.
\end{proof}

\begin{remark}
\label{rem:bl57-hi}
In contrast to the case of Mukai bundles, the bundles~$\cR_2$ and~$\cR_3$ 
(and consequently the morphism $\beta_S$ defined below)
depend on the choice of the curve~$C$ and the line bundles~$\xi$ and~$\eta$.
In fact, the construction of Lemma~\ref{lem:laz-cl57} 
produces a 2-dimensional moduli space of vector bundles, which is isomorphic to another K3 surface of genus~7, a Fourier--Mukai partner of~$S$.
\end{remark}

\begin{proposition}
\label{prop:g7-gr25}
Let~$S$ be a prime $K3$ surface of genus~$7$.
For any curve~$C$ and line bundles~$\xi$ and~$\eta$ on it as in Lemma~\textup{\ref{lem:laz-cl57}}
there is a closed embedding~$\beta_S \colon S \hookrightarrow \Gr(2,V_5)$ 
such that~$\cR_2 \cong \beta_S^*(\cU)$ and~$\cR_3 \cong \beta_S^*(\cU^\perp)$.
\end{proposition}

\begin{proof}
The existence of the morphism~$\beta_S$ follows from~\eqref{eq:R2-sub} and~\eqref{eq:R3-sub},
so we only need to show it is a closed embedding.
For this we check that the morphism~$\wedge^2V_5^\vee \otimes \cO_S \to \wedge^2\cR_2^\vee \cong \cO_S(H)$
induced by the second arrow in~\eqref{eq:R3-sub}
is surjective on global sections.
As~$\cR_3^\vee \cong \beta_S^*(V_5/\cU)$ is stable, 
the argument of Lemma~\ref{lem:no-schubert} shows that 
any skew-form in the kernel of~$\wedge^2V_5^\vee \to \rH^0(S, \cO_S(H))$ has rank~4, 
so if the morphism is not surjective then the composition
\begin{equation*}
S \xrightarrow{\ \beta_S\ } \Gr(2,V_5) \hookrightarrow \P(\wedge^2V_5) = \P^9
\end{equation*}
factors through~$\P^6 \subset \P^9$ such that~$Y \coloneqq \Gr(2,V_5) \cap \P^6$ is a smooth quintic del Pezzo threefold.
Since~$H$ is ample, the morphism~$S \to Y$ must be finite onto a surface in~$Y$
whose degree divides~\mbox{$\deg(S) = 12$}.
But~$\Pic(Y)$ is generated by the restriction of the Pl\"ucker class, 
hence the degree of any surface in~$Y$ is divisible by~5;
this contradiction shows the required surjectivity.
Since~$H$ is very ample by~\cite[Theorem~3.1 and Theorem~5.2]{SD}, 
it follows that the composition of the above arrows is a closed embedding, hence so is~$\beta_S$.
\end{proof}

\begin{remark}
We do not need this, but it is not hard to prove that~$\beta_S(S) \subset \Gr(2,V_5)$ 
is the zero locus of a regular global section of the vector bundle~$\cO(1) \oplus \cO(1) \oplus \cU^\vee(1)$.
\end{remark}

We will need the following results about the cohomology of various natural bundles on $S$:

\begin{lemma}
\label{lem:coh-tps}
The following table contains the dimensions of cohomology of some bundles on~$S$:
\begin{equation*}
\begin{array}{|c|c|c|c|c|c|c|}
\hline
{} & 
\wedge^2\cR_2 &
\wedge^2\cR_3 &
\cR_2 \otimes \cR_3^\vee & 
\cR_2 \otimes \cR_3 &
\Sym^2\cR_2 &
\Sym^2\cR_3 
\\
\hline 
h^2 & 
8 & 11 & 
2 & 26 & 
16 & 15
\\
h^1 & 
0 & 0 & 
1 & 1 & 
0 & 2
\\
h^0 & 
0 & 0 & 
0 & 0 & 
0 & 0
\\
\hline 
\end{array}
\end{equation*}
Moreover, the natural morphism
\begin{equation}
\label{eq:h1h1-h2}
\rH^1(S, \cR_2 \otimes \cR_3) \otimes \rH^1(S, \cR_2 \otimes \cR_3^\vee) \to
\rH^2(S, \cR_2 \otimes \cR_2) 
\end{equation}
is injective and its image is contained in the symmetric part~$\rH^2(S, \Sym^2\cR_2)$ of~$\rH^2(S, \cR_2 \otimes \cR_2)$.
\end{lemma}

\begin{proof}
Since $\wedge^2\cR_2 \cong \cO_S(-H)$ the first column is immediate from Serre duality. Similarly, the second column follows from~$\wedge^2\cR_3 \cong \cR_3^\vee(-H)$,
Serre duality, Lemma~\ref{lem:g10H1vanishing}, and Riemann--Roch. 

Using the surjectivity of the morphism~$\wedge^2V_5^\vee \to \rH^0(S, \wedge^2\cR_2^\vee) = \rH^0(S, \cO_S(H))$
proved in Proposition~\ref{prop:g7-gr25}, 
we see that the vector bundle~$\Ker(\wedge^2V_5^\vee \otimes \cO_S \to \wedge^2\cR_2^\vee)$
has~$h^0 = 2$, $h^1 = 0$, and~$h^2 = 10$. 
On the other hand, ~\eqref{eq:R3-sub} implies an exact sequence
\begin{equation*}
0 \to \wedge^2\cR_3 \to \Ker\left(\wedge^2V_5^\vee \otimes \cO_S \to \wedge^2\cR_2^\vee\right) \to \cR_3 \otimes \cR_2^\vee \to 0.
\end{equation*}
Taking into account that~$\rH^2(S, \cR_3 \otimes \cR_2^\vee) = \Hom(\cR_3, \cR_2)^\vee = 0$
by Serre duality and stability of~$\cR_3$ and~$\cR_2$, 
we compute the cohomology of~$\cR_3 \otimes \cR_2^\vee$, hence the third column of the table.

To compute the fourth column we tensor~\eqref{eq:R2-sub} with~$\cR_3$ and use~\eqref{eq:ext-cr-cr}.
Note that this computation also shows that the extension class of~\eqref{eq:R2-sub}
generates the space
\begin{equation*}
\Ext^1(\cR_3^\vee, \cR_2) = \rH^1(S, \cR_3 \otimes \cR_2).
\end{equation*}

Next, tensoring~\eqref{eq:R2-sub} with~$\cR_2$ and using the cohomology of~$\cR_2 \otimes \cR_3^\vee$
computed above, we compute the cohomology of~$\cR_2 \otimes \cR_2$,
hence also the cohomology of~$\Sym^2\cR_2$.

Furthermore, since the extension class of~\eqref{eq:R2-sub} generates~$\rH^1(S, \cR_3 \otimes \cR_2)$,
it follows that the morphism~\eqref{eq:h1h1-h2} coincides with the connecting morphism 
in the tensor product of this sequence with~$\cR_2$,
and since~$\rH^1(S, V_5 \otimes \cR_2) = 0$, the connecting morphism is injective
and its image is the kernel of the natural morphism
\begin{equation*}
\rH^2(S, \cR_2 \otimes \cR_2) \to V_5 \otimes \rH^2(S, \cR_2) = V_5 \otimes V_5.
\end{equation*}
Since this kernel is 1-dimensional and invariant under the action of transposition, it is symmetric or skew-symmetric.
But the argument of Proposition~\ref{prop:g7-gr25} combined with Serre duality shows that
the skew-symmetric part of this morphism is injective, hence the kernel is symmetric.

Finally, tensoring~\eqref{eq:R2-sub} with~$\cR_3^\vee$, 
we compute the cohomology of~$\cR_3^\vee \otimes \cR_3^\vee$; 
using Serre duality and the second column we obtain the last column of the table.
\end{proof}

\subsection{Mukai bundle}
\label{ss:mb-g7}

We know from Lemma~\ref{lem:coh-tps} that~$\Ext^1(\cR_3,\cR_2) \cong \rH^1(S, \cR_2 \otimes \cR_3^\vee) \cong \kk$.
We now consider the vector bundle~$\cU_S$ on~$S$ defined by the unique non-split exact sequence
\begin{equation}
\label{eq:g7-cw}
0 \to \cR_2 \to \cU_S \to \cR_3 \to 0.
\end{equation}
Then~$\rv(\cU_S) = (5,-2H,5)$.
Note that~$\cU_S$ has the Mukai vector of a Mukai bundle of type~$(5,5)$ 
with respect to the double polarization~$2H$ of~$S$,
see Definition~\ref{def:mb}.

\begin{lemma}\label{lem:coh-w2cw}
We have~$h^0(S, \wedge^2\cU_S) = 0$, $h^1(S, \wedge^2\cU_S) = 1$, $h^2(S, \wedge^2\cU_S) = 45$. 
\end{lemma}

\begin{proof}
Consider the filtration of~$\wedge^2\cU_S$ with factors~$\wedge^2\cR_2$, $\cR_2 \otimes \cR_3$, and~$\wedge^2\cR_3$.
The lemma follows easily from the standard spectral sequence and Lemma~\ref{lem:coh-tps}.
\end{proof}

In the next theorem we show that~$\cU_S$ is a Mukai bundle of type~$(5,5)$ and induces a morphism from~$S$
to~$\rM_7 = \OGr_+(5,10)$.
Here we denote by~$\OGr(5,10) \subset \Gr(5,10)$ the subvariety parameterizing subspaces 
that are isotropic for a fixed nondegenerate symmetric bilinear form;
it is smooth of dimension~10 and has two (mutually isomorphic) connected components 
that we denote by~$\OGr_+(5,10)$ and~$\OGr_-(5,10)$.
The restriction of the Pl\"ucker line bundle to either of these components is a square,
and its square root is the ample generator of the Picard group of~$\OGr_\pm(5,10)$, 
denoted by~$\cO_{\OGr_+(5,10)}(1)$ and called the {\sf spinor line bundle}.
It induces an embedding
\begin{equation*}
\OGr_+(5,10) \hookrightarrow \P^{15}
\end{equation*}
into the projectivization of the spinor representation of~$\Spin(V_{10})$, called {\sf the spinor embedding}.
The degree of~$\OGr_+(5,10)$ in the spinor embedding is~12, see~\cite{K18} for more information.

\begin{theorem}
\label{thm:g7-emb}
Let~$(S,H)$ be a prime $K3$ surface of genus~$g = 7$.
The vector bundle~$\cU_S$ defined in~\eqref{eq:g7-cw} is stable; 
in particular, it is the Mukai bundle on~$(S,2H)$ of type~$(5,5)$.
The dual bundle~$\cU_S^\vee$ is globally generated with~$h^0(\cU_S^\vee) = 10$ and~$h^1(\cU_S^\vee) = h^2(\cU_S^\vee) = 0$; it defines a closed embedding
\begin{equation*}
\gamma_S \colon S \to \OGr_+(5,10) \subset \Gr(5,10)
\end{equation*}
such that~$\cU_S$ is the pullback of the tautological subbundle 
and~$\cO_S(H) \cong \gamma_S^*(\cO_{\OGr_+(5,10)}(1))$.
\end{theorem}

\begin{proof}
Since~$\cR_2$ and~$\cR_3$ are stable of slope $-1/2$ and~$-1/3$, and the extension~\eqref{eq:g7-cw} is non-split, each factor of the Harder--Narasimhan filtration of~$\cU_S$ has slope in the interval~$(-1/2,-1/3)$.
But this interval has no rational numbers with denominator less than~5, hence~$\cU_S$ is stable.

Now, dualizing~\eqref{eq:g7-cw} and using Lemma~\ref{lem:laz-cl57}, we obtain a commutative diagram
\begin{equation}
\label{eq:cw-diagram}
\vcenter{\xymatrix{
0 \ar[r] &
V_5 \otimes \cO_S \ar[r] \ar[d]^\ev &
\rH^0(S, \cU_S^\vee) \otimes \cO_S \ar[r] \ar[d]^\ev &
V_5^\vee \otimes \cO_S \ar[r] \ar[d]^\ev &
0
\\
0 \ar[r] &
\cR_3^\vee \ar[r] &
\cU_S^\vee \ar[r] &
\cR_2^\vee \ar[r] &
0.
}}
\end{equation}
It also follows that~$h^0(\cU_S^\vee) = 10$, $h^{1}(\cU_S^\vee) = h^{2}(\cU_S^\vee) = 0$, and that~$\cU_S^\vee$ is globally generated.
Further, if~$\cU_S'$ denotes the kernel of the middle vertical arrow,
using~\eqref{eq:R2-sub} and~\eqref{eq:R3-sub} we obtain an exact sequence
\begin{equation*}
0 \to \cR_2 \to \cU_S' \to \cR_3 \to 0.
\end{equation*}
If it splits, dualizing the above construction we deduce that~\eqref{eq:g7-cw} also splits, contradicting our assumption.
Thus, $\cU_S'$ is a non-split extension of the same form as~$\cU_S$, and since such an extension 
is unique by Lemma~\ref{lem:coh-tps}, there is an isomorphism
\begin{equation*}
\cU_S' \cong \cU_S,
\end{equation*}
unique up to constant (because~$\cU_S$ is stable), and there is an exact sequence
\begin{equation}
\label{eq:cw-cwd}
0 \to \cU_S \xrightarrow{\ \varphi\ } V_{10} \otimes \cO_S \xrightarrow{\ \ev\ } \cU_S^\vee \to 0,
\end{equation}
where~$V_{10} \coloneqq \rH^0(S, \cU_S^\vee)$ and~$\varphi$ is unique up to constant.
Dualizing this sequence and using the stability of~$\cU_S$, we obtain a unique isomorphism~$q \colon V_{10} \xrightarrow{\ \sim\ } V_{10}^\vee$ such that the following diagram 
\begin{equation*}
\xymatrix@C=3em{
0 \ar[r] & 
\cU_S \ar[r]^-{\ev^\vee} \ar[d]_{c\id} & 
V_{10}^\vee \otimes \cO_S \ar[r]^-{\varphi^\vee} \ar[d]^q & 
\cU_S^\vee \ar[r] \ar[d]^{\id} &
0
\\
0 \ar[r] & 
\cU_S \ar[r]^-{\varphi} & 
V_{10} \otimes \cO_S \ar[r]^-{\ev} & 
\cU_S^\vee \ar[r] &
0
}
\end{equation*}
commutes for some non-zero scalar~$c$.
The uniqueness of~$q$ then implies that~$q^\vee = c q$, hence~\mbox{$c = \pm1$}; 
in other words, $q$ is a non-degenerate symmetric or skew-symmetric form.

Consider the morphism~$\gamma_S \colon S \to \Gr(5,V_{10})$ given by the bundle~$\cU_S$.
It follows from the diagram that~$\ev \circ\, q \circ \ev^\vee = 0$, 
which means that~$\gamma_S$ factors through the locus of $q$-isotropic subspaces.
We show below that~$q$ is symmetric, hence~$\gamma_S$ factors through~$\OGr(5,V_{10})$.

Assume to the contrary that~$q$ is skew-symmetric.
Then the morphism 
\begin{equation*}
\wedge^2V_{10} \to \rH^0(S, \wedge^2\cU_S^\vee)
\end{equation*}
induced by the second arrow in~\eqref{eq:cw-cwd} contains~$q$ in the kernel.
To arrive at a contradiction, we show that this kernel space is zero.
For this we consider the morphism~$\psi \colon \wedge^2V_{10} \otimes \cO_S \to \wedge^2\cU_S^\vee$ induced by the above map.
On the one hand, $\psi$ is surjective by~\eqref{eq:cw-cwd}, so it is enough to check that~$\rH^0(S, \Ker(\psi)) = 0$.
On the other hand, \eqref{eq:cw-cwd} implies that there is an exact sequence
\begin{equation*}
0 \to \wedge^2\cU_S \to \Ker(\psi) \to \cU_S \otimes \cU_S^\vee \to 0.
\end{equation*}
Since~$h^0(\wedge^2\cU_S) = 0$ by Lemma~\ref{lem:coh-w2cw} and~$h^0(\cU_S \otimes \cU_S^\vee) = 1$ by stability of~$\cU_S$, it is enough to check that the connecting morphism~$\rH^0(S, \cU_S \otimes \cU_S^\vee) \to \rH^1(S, \wedge^2\cU_S)$ of this exact sequence is non-trivial. 
Clearly, this morphism factors as the composition of two maps
\begin{equation}
\label{eq:composition}
\rH^0(S, \cU_S \otimes \cU_S^\vee) \to 
\rH^1(S, \cU_S \otimes \cU_S) \to
\rH^1(S, \wedge^2\cU_S)
\end{equation}
where the first is the connecting morphism of sequence~\eqref{eq:cw-cwd} tensored with~$\cU_S$ and the second is induced by the projection~$\cU_S \otimes \cU_S \to \wedge^2\cU_S$.
We check that both maps in~\eqref{eq:composition} are isomorphisms.

Indeed, tensoring~\eqref{eq:cw-cwd} by~$\cU_S$ we obtain an exact sequence
\begin{equation*}
0 \to 
\cU_S \otimes \cU_S \xrightarrow{\ \varphi\ } 
V_{10} \otimes \cU_S \xrightarrow{\ \ev\ } 
\cU_S \otimes \cU_S^\vee \to 
0,
\end{equation*}
and since~$h^0(\cU_S) = h^1(\cU_S) = 0$ (by the vanishing~$h^{2}(\cU_S^\vee) = h^{1}(\cU_S^\vee) = 0$ proved above and Serre duality), 
it follows that the first arrow in~\eqref{eq:composition} is an isomorphism, hence we have~$h^1(\cU_S \otimes \cU_S) = 1$.
Now, combining this with Lemma~\ref{lem:coh-w2cw} we conclude that~$h^1(\Sym^2\cU_S) = 0$, and therefore the second arrow in~\eqref{eq:composition} is also an isomorphism. 

We see that the composition~\eqref{eq:composition} is an isomorphism, hence~$\rH^0(S, \Ker(\psi)) = 0$, hence~$q = 0$, which is absurd.
Thus, $q$ must be symmetric, and we see that the morphism~\mbox{$\gamma_S \colon S \to \Gr(5,V_{10})$} factors through~$\OGr(5,V_{10})$.
Since~$S$ is connected, $\gamma_S$ factors through a connected component of~$\OGr(5,V_{10})$, 
and we may and will assume that this component is~$\OGr_+(5,V_{10})$.

Next, we note that~\eqref{eq:cw-diagram} implies that there is a commutative diagram
\begin{equation*}
\xymatrix@C=3em{
& S \ar@{_{(}->}[dl]_{\beta_S} \ar[d] \ar[dr]^{\gamma_S}
\\
\Gr(2,V_5) & 
Z \ar[l] \ar@{^{(}->}[r] & 
\OGr_+(5,V_{10}),
}
\end{equation*}
where~$Z = \{ [U_5] \in \OGr_+(5,V_{10}) \mid \dim(U_5 \cap V_5) = 2 \}$, 
the lower right arrow is the natural embedding, the lower left arrow is defined by~$U_5 \mapsto U_5 \cap V_5$, 
and~$\beta_S$ is the map defined in Proposition~\ref{prop:g7-gr25}.
Since~$\beta_S$ is a closed embedding, the same is true for the vertical arrow, 
and hence~$\gamma_S$ is a closed embedding as well.

Finally, since the pullback along~$\gamma_S$ of the Pl\"ucker class of~$\Gr(5,V_{10})$ is~$\rc_1(\cU_S^\vee)$, 
which is equal to~$2H$ by~\eqref{eq:g7-cw}, it follows that~$\cO_S(H) \cong \gamma_S^*(\cO_{\OGr_+(5,10)}(1))$.
\end{proof}

In what follows we call~$\gamma_S$ {\sf the Gushel map} of~$S$.

\subsection{K3 surfaces of genus~7}
\label{ss:ls-g7}

In this section we show that~$\gamma_S(S) \subset \OGr_+(5,V_{10})$ is a transverse linear section 
and identify the Mukai bundle~$\cU_S$ with a twisted normal bundle of~$S$ in~$\P^7$.%

We need to recall some extra facts about the geometry of~$\OGr_+(5,V_{10})$; 
for a more detailed treatment we refer to~\cite{K18}.
First, recall that for any pair of isotropic subspaces~$U_5,U'_5 \subset V_{10}$ 
the dimension~$\dim(U_5 \cap U'_5)$ is odd or even if~$[U_5]$ and~$[U'_5]$ 
belong to the same or to different connected components of~$\OGr(5,V_{10})$.
Moreover, if~$[V_5] \in \OGr_-(5,V_{10})$ then
\begin{align}
\label{eq:schubert1-ogr}
\Sigma_1^{\OGr}(V_5) &\coloneqq \{ [U_5] \in \OGr_+(5,V_{10}) \mid \dim(U_5 \cap V_5) \ge 2 \}
\intertext{is an irreducible (Schubert) divisor with}
\label{eq:sing-schubert1-ogr}
\Sing(\Sigma_1^{\OGr}(V_5)) &\,= \{ [U_5] \in \OGr_+(5,V_{10}) \mid \dim(U_5 \cap V_5) = 4 \}.
\end{align}
Finally, if~$V_3 \subset V_{10}$ is a 3-dimensional isotropic subspace then the Schubert variety
\begin{equation*}
\Sigma_{2,1}^{\OGr}(V_3) \coloneqq \{ [U_5] \in \OGr_+(5,V_{10}) \mid \dim(U_5 \cap V_3) \ge 1 \}
\end{equation*}
is an irreducible subvariety in~$\OGr_+(5,V_{10})$ of codimension~2.

The following result plays the role of Lemma~\ref{lem:no-schubert} (and the argument is analogous).

\begin{lemma}\label{lem:no-schubert-g7}
Let~$S$ be a prime $K3$ surface of genus~$7$.
If~$\gamma_S \colon S \hookrightarrow \OGr_+(5,V_{10})$ is the Gushel map then~$\gamma_S(S)$ 
is not contained in any Schubert variety~$\Sigma_{2,1}^{\OGr}(V_3)$.
\end{lemma}

\begin{proof}
If~$\gamma_S(S)$ is contained in~$\Sigma_{2,1}^{\OGr}(V_3)$ then the morphism~$\phi \colon V_3 \otimes \cO_S \to \cU_S^\vee$ is everywhere degenerate.
In other words, if~$\cF \coloneqq \Ima(\phi)$ is its image, then~$\rank(\cF) \le 2$.
Now, if~$\rc_1(\cF) \ge H$, then
\begin{equation*}
\upmu(\cF) \ge \frac12 > \frac25 = \upmu(\cU_S^\vee),
\end{equation*}
in contradiction to the stability of~$\cU_S^\vee$ proved in Theorem~\textup{\ref{thm:g7-emb}}.
Thus, $\rc_1(\cF) \le 0$, and since~$\cF$ is torsion free, it follows from~\cite[Lemma~3.7(b)]{BKM} that~$\cF \cong \cO_S^{\oplus m}$ with~$m \le 2$, in contradiction to the injectivity of the morphism
\begin{equation*}
\rH^0(S, V_3 \otimes \cO_S) = V_3 \hookrightarrow V_{10} = \rH^0(S, \cU_S^\vee).
\end{equation*}
This contradiction proves the lemma.
\end{proof}

And the following result plays the role of Lemma~\ref{lem:yes-schubert} (but now the argument is different).

\begin{lemma}\label{lem:yes-schubert-g7}
Let~$Y = \OGr_+(5,V_{10}) \cap \P^7$.
If~$[U_5] \in Y$ is a point such that~$\dim(\rT_{[U_5],Y}) \ge 4$ then~$Y$ 
is contained in a Schubert cycle~$\Sigma_{2,1}^{\OGr}(V_3) \subset \OGr_+(5,V_{10})$ 
for some isotropic subspace~$V_3 \subset V$.
\end{lemma}

\begin{proof}
Since~$\OGr_+(5,V_{10}) \subset \P^{15}$ is a smooth variety of dimension~10
and~$Y$ is an intersection of~$\OGr_+(5,V_{10})$ with a linear subspace of codimension~8, 
the tangent space~$\rT_{[U_5],Y}$ is the kernel of a linear map~$\kk^{10} = \rT_{[U_5],\OGr_+(5,V_{10})} \to \kk^8$,
hence the assumption~$\dim(\rT_{[U_5],Y}) \ge 4$ 
implies that~$Y$ is contained in two hyperplane sections of~$\OGr_+(5,V_{10})$ singular at~$[U_5]$. 
But any singular hyperplane section of~$\OGr_+(5,V_{10})$ is a Schubert divisor~\eqref{eq:schubert1-ogr} 
(see, e.g., \cite[Corollary~4.2]{K18}) and its singular locus is described by~\eqref{eq:sing-schubert1-ogr}.
Therefore, there are distinct isotropic subspace~$V'_5, V''_5 \subset V_{10}$ 
from the other component~$\OGr_-(5,V_{10})$ of the isotropic Grassmannian)
such that
\begin{equation*}
Y \subset \Sigma_1^{\OGr}(V'_5) \cap \Sigma_1^{\OGr}(V''_5)
\qquad\text{and}\qquad
\dim(V'_5 \cap U_5) = \dim(V''_5 \cap U_5) = 4.
\end{equation*}
Let $V_3 \coloneqq V'_5 \cap V''_5$; then the second property implies that~$\dim V_3 = 3$.
We will show that
\begin{equation}
\label{eq:schubert-inclusion}
\Sigma_1^{\OGr}(V'_5) \cap \Sigma_1^{\OGr}(V''_5) = \Sigma_{2,1}^{\OGr}(V_3);
\end{equation}
obviously, this will complete the proof of the lemma.

To prove~\eqref{eq:schubert-inclusion}, we first note that the right side is contained in the left.
Indeed, if~$[W_5] \in \Sigma_{2,1}^{\OGr}(V_3)$
then~$\dim(V'_5 \cap W_5)$ and~$\dim(V''_5 \cap W_5)$ are positive, and since they must be even, 
they are greater or equal than~2, hence~$[W_5]$ is in the left side. Moreover, the right-hand side is reduced by definition; hence this set-theoretic inclusion is also a scheme-theoretic embedding.\
On the other hand, by Schubert calculus the second Chow group of~$\OGr_+(5,V_{10})$ 
is generated by the square of the hyperplane class
(indeed, Schubert cells in~$\rG/\rP$ of codimension~2 are parameterized by the vertices of the Dynkin diagram of~$\rG$
adjacent to the vertex corresponding to~$\rP$, and in the case where~$\rG = \SO_{10}$ and~$\rP = \rP_5$
there is only one adjacent vertex),
hence the class of the right side in the second Chow group must be a multiple of the class of the left side.
This implies that the left side is irreducible and generically reduced 
and the embedding of the right side is an isomorphism at the general point.
Finally, the left side is a complete intersection, hence it is Cohen--Macaulay, 
hence the embedding is an isomorphism.
\end{proof}

The last few things we need to know about the spinor variety~$\OGr_+(5,V_{10})$ 
is that it is an intersection of~$10$ quadrics in~$\P^{15}$, its degree is~$12$ (as we already mentioned), and
\begin{equation}
\label{eq:cn-ogr}
\cN_{\OGr_+(5,V_{10})/\P^{15}} \cong \cU(2),
\end{equation}
where~$\cU$ is the tautological bundle of~$\OGr_+(5,V_{10})$,
see, e.g., \cite[Corollaries~4.3 and~4.6]{K18}.

\begin{proposition}\label{prop:s-g7}
If~$S$ is a prime $K3$ surface of genus~$7$ then
\begin{equation*}
\gamma_S(S) = \OGr_+(5,V_{10}) \cap \P^7
\end{equation*}
is a transverse intersection.
\end{proposition}

\begin{proof}
The argument is analogous to that of Proposition~\ref{prop:m-cap-p}.
First, we check that
\begin{equation*}
S \subset Y \coloneqq \OGr_+(5,V_{10}) \cap \P^7
\end{equation*}
by using the argument of the second part of Lemma~\ref{lem:gamma-x-s}. 
Next, using Lemmas~\ref{lem:no-schubert-g7} and~\ref{lem:yes-schubert-g7} 
(instead of Lemmas~\ref{lem:no-schubert} and~\ref{lem:yes-schubert}) 
we show that either~$Y$ is a Cohen--Macaulay surface or a smooth connected threefold.
In the first case, it follows that~$Y = S$, because in this case
\begin{equation*}
\deg(Y) = \deg(\OGr_+(5,V_{10})) = 12 = \deg(S),
\end{equation*}
and in the second case we use the argument of Proposition~\ref{prop:m-cap-p}
taking into account that
\begin{equation*}
\rH^0(S, \gamma_S^*(\cN_{\OGr_+(5,V_{10})/\P^{15}}^\vee(1))) \cong
\rH^0(S, \gamma_S^*(\cU^\vee(-1))) \cong
\rH^0(S,\cU_S^\vee(-H)) = 0,
\end{equation*}
where the first is~\eqref{eq:cn-ogr}, 
the second follows from the definition of~$\gamma_S$,
and the third is implied by the stability of the Mukai bundle~$\cU_S$.
\end{proof}

As we will see, the following alternative description of the Mukai bundle~$\cU_S$ is very useful.

\begin{corollary}
\label{cor:cw-cn}
If~$(S,H)$ is a prime $K3$ surface of genus~$7$ 
and~$\cU_S$ is the Mukai bundle of type~$(5,5)$ with respect to the polarization~$2H$
constructed in Theorem~\textup{\ref{thm:g7-emb}}, 
there are isomorphisms
\begin{equation*}
\cU_S \cong \cN_{S/\P^7}(-2H),
\qquad\text{and}\qquad
\rH^0(S, \cU_S^\vee) \cong \rH^0(\P^7, I_S(2)) \cong \kk^{10},
\end{equation*}
where~$S \hookrightarrow \P^7$ is the embedding given by the polarization~$|H|$.
\end{corollary}

\begin{proof}
Since~$S = \OGr_+(5,V_{10}) \cap \P^7$ is a transverse intersection, it follows that
\begin{equation*}
\cN_{S/\P^7} \cong \cN_{\OGr_+(5,V_{10})/\P^{15}}\vert_S \cong \cU(2)\vert_S \cong \cU_S(2H),
\end{equation*}
where in the second isomorphism we used~\eqref{eq:cn-ogr}
and the last holds by definition of~$\gamma_S$.

Similarly, we have isomorphisms
\begin{equation*}
\rH^0(\P^7, I_S(2)) \cong
\rH^0(\P^{15}, I_{\OGr_+(5,V_{10})}(2)) \cong
V_{10} \cong
\rH^0(\OGr_+(5,V_{10}), \cU^\vee) \cong
\rH^0(S, \cU_S^\vee),
\end{equation*}
where the first isomorphism is proved in~\cite[Lemma~A.5]{DK},
the second is proved in~\cite[Corollary~4.3]{K18}
(thus, $V_{10}$ coincides with the 10-dimensional space of quadrics in~$\P^{15}$ passing through~$\OGr_+(5,V_{10})$),
the third follows from~\cite[Corollary~4.4]{K18},
and the last follows from the argument of Theorem~\ref{thm:g7-emb}.
In conclusion, we obtain the second isomorphism of the lemma.
\end{proof}

We can now finish proving Theorem~\ref{thm:main} 
(recall that the cases~$g \in \{6,8,9,10,12\}$ are established in Section~\ref{subsec:K3g=6-10}).

\begin{proof}[Proof of Theorem~\textup{\ref{thm:prime-k3}} for~$g = 7$]
The Gushel embedding~$\gamma_S \colon S \hookrightarrow \rM_7 = \OGr_+(5,V_{10})$ 
is constructed in Theorem~\ref{thm:g7-emb}
and an isomorphism of~$S$ with a linear section of~$\rM_7$ is proved in Proposition~\ref{prop:s-g7}.
It remains to prove the uniqueness.

So, let~$f \colon S \hookrightarrow \OGr_+(5,V_{10})$ be a closed embedding 
such that~$S = \OGr_+(5,V_{10}) \cap \P^7$ is a transverse intersection.
If~$\cU$ is the tautological bundle of~$\OGr_+(5,V_{10})$ then
\begin{equation*}
f^*(\cU^\vee) \cong 
f^*(\cN^\vee_{\OGr_+(5,V_{10})/\P^{15}}(2)) \cong
\cN^\vee_{S/\P^7}(2H) \cong
\cU_S^\vee,
\end{equation*}
where the first is~\eqref{eq:cn-ogr}, 
the second follows from the transversality of intersection,
and the third is Corollary~\ref{cor:cw-cn}.
Moreover, the argument of Corollary~\ref{cor:cw-cn} also shows that this isomorphism
induces an isomorphism~$\rH^0(\OGr_+(5,V_{10}), \cU^\vee) \xrightiso{} \rH^0(S, \cU_S^\vee)$
such that its dual map gives an automorphisms~$\alpha \in \Aut(\OGr_+(5,V_{10}))$ such that~$f = \alpha \circ \gamma_S$, 
as required.
\end{proof}


\section{Fano varieties}\label{sec:fano}

In this section we deduce Theorem~\ref{thm:main} from Theorem~\ref{thm:prime-k3}.
In Section~\ref{ss:hyperplanes-cones} we show that Fano varieties as in Theorem~\ref{thm:main}
contain prime K3 surfaces as transverse linear sections;
in Section~\ref{ss:extension} we explain the extension argument used in the proof of Theorem~\ref{thm:main};
in Section~\ref{ss:fano-3} we prove the theorem for threefolds, where
the argument relies on the main result of~\cite{BKM};
and in Section~\ref{ss:fano-all} we explain how the Mukai bundle and the Gushel map extend to higher dimensions, concluding the proof of Theorem~\ref{thm:main}.

\subsection{Hyperplanes and positivity properties}
\label{ss:hyperplanes-cones}

Our reduction of Theorem~\ref{thm:main} to Theorem~\ref{thm:prime-k3} 
is based on the following result, which is well known to experts.
Recall that any factorial variety is Gorenstein.

\begin{lemma}
\label{lem:hva}
Let~$X$ be a Fano variety of dimension~$n \ge 3$ 
with at most factorial terminal singularities such that~\eqref{eq:prime-fano} holds
\textup(in particular, let~$g = \g(X)$ be the genus of~$X$\textup).
\begin{enumerate}[label={\textup{(\alph*)}}, wide]
\item 
\label{item:Hpositive}	
If~\mbox{$g \ge 3$} then~$|H|$ is base point free,
if~\mbox{$g \ge 4$} then~$|H|$ is very ample and induces a projectively normal embedding~$X \hookrightarrow \P^{n + g - 2}$, and 
if~\mbox{$g \ge 5$} then~$X \subset \P^{n + g - 2}$ is an intersection of quadrics.
\item 
\label{item:gendivisor} 
If~$g \ge 3$ and~$n \ge 4$ a general divisor~$X' \subset X$ in~$|H|$ 
has at most factorial terminal singularities and satisfies~\eqref{eq:prime-fano}.
\item 
\label{item:genK3} 
If~$g  \ge 4$ a very general intersection~$S = X \cap \P^g \subset X$ of~$n - 2$ divisors in~$|H|$ 
is a smooth prime $K3$ surface of genus~$g$.
\end{enumerate}
\end{lemma}

\begin{proof}
We first note that claim~\ref{item:gendivisor} is a combination of~\cite[Theorem~2.5 and Remark~2.6]{Mella}, 
which shows that~$X'$ is terminal, and~\cite[Theorem~1]{RS06}, which implies~$X'$ is factorial of Picard rank one. 
Consequently, we can prove claim~\ref{item:Hpositive} by induction. 

The base of induction for claim~\ref{item:Hpositive}, $n=3$, is already established in the literature, 
see~\cite[Section~4]{KP23} and references therein, 
in particular~\cite[Proposition~4.1]{KP23} and~\cite[Theorem~1.1]{Jahnke-Radloff} for the base point freeness for~$g \ge 3$, 
\cite[Proposition~4.4]{KP23} and~\cite[Theorem~4.2]{Prokhorov:rationality} for the very ampleness, 
and its proof for the projective normality, 
and~\cite[Proposition~4.8]{KP23} and~\cite[Theorem~4.5]{Prokhorov:rationality} for the intersection by quadric property. 
(These statements are based on reduction to K3 surfaces and the corresponding results by Saint-Donat in~\cite{SD}, 
with a similar proof as the induction step below.)

We now proceed by induction for~$n \ge 4$, and let~$X' \subset X$ be a divisor satisfying~\ref{item:gendivisor}. 
It is a standard argument to show that the positivity properties of~$H$ asserted in~\ref{item:Hpositive} 
follow from the corresponding properties for~$H|_{X'}$. 
Indeed, for the base point free property this is obvious, 
and for the other two we consider the restriction morphism
\begin{equation*} 
\bigoplus_{m \ge 0} \rH^0(X, \cO_{X}(mH)) \to 
\bigoplus_{m \ge 0} \rH^0(X', \cO_{X'}(mH\vert_{X'}))
\end{equation*}
between the section rings. 
It is surjective by Kawamata--Viehweg vanishing theorem, and the kernel is generated by the equation of~$X'$, an element of degree~$1$.

The section ring of an ample divisor is generated in degree one if and only if 
the divisor is very ample and induces a projectively normal embedding.
In this case, the variety is an intersection of quadrics if and only if the section ring has relations generated in degree two. 
Both properties hold for a non-negatively graded ring if and only if they hold for its quotient by a degree~$1$ element.
	
Finally, claim~\ref{item:genK3} follows from~\cite[Theorem~1]{RS09}.
Indeed, by claim~\ref{item:gendivisor} we may assume that~\mbox{$n=3$}, 
and by claim~\ref{item:Hpositive} we know that~$X \subset \P^{g+1}$ in this case. 
Thus the pushforward to~$\P^{g + 1}$ of the canonical bundle is ~$\cO_X(K_X) \cong \cO_X(-H)$, 
and after twist by~$\cO_{\P^{g+1}}(1)$ it becomes trivial, hence globally generated, as required in~\cite[Theorem~1]{RS09}.
\end{proof}

\begin{remark}
Lemma~\ref{lem:hva}\ref{item:genK3} does not hold for~$g = 3$.
Indeed, if~$X \to Q^3$ is a smooth double covering of a smooth 3-dimensional quadric 
branched at the intersection of the quadric and a quartic, 
every smooth hyperplane section of~$X$ is a double covering of~$\P^1 \times \P^1$, 
hence its Picard group has rank at least~2.
Thus, the assumption~$g \ge 4$ is sharp.
\end{remark}

One consequence of this lemma is the genus bound in higher dimensions.

\begin{corollary}
\label{cor:genus-bound}
If~$X$ is a Fano variety of dimension~$n \ge 3$ and genus~$g \ge 6$ 
with at most factorial terminal singularities such that~\eqref{eq:prime-fano} holds
then~$g \in \{6,7,8,9,10,12\}$.
\end{corollary}

\begin{proof}
By Lemma~\ref{lem:hva} a general intersection~$X \cap \P^{g + 1}$ is a prime Fano threefold of genus~$g$
with factorial terminal (hence Gorenstein) singularities.
By~\cite[Theorem~11]{Namikawa} it has a smoothing, 
the general fiber of which is a smooth prime Fano threefold, which by~\cite[Proposition~2.5]{KP23} has the same genus.
Therefore, $g \in \{6,7,8,9,10,12\}$ by the classification of Iskovskikh, see~\cite[Theorem~4.6.7]{Fano-book}.
\end{proof}

The following simple observation gives us a control of singularities for some cones
(recall the notation~$\Cone_{\P(K)}(Y)$ defined in the introduction).

\begin{lemma}
\label{lem:cone-singularity}
Let~$Y_0 \subset \P^N$ be a normal Gorenstein projective variety with~$\omega_{Y_0} \cong \cO_{Y_0}(-m_0)$.
Let~$Y = \Cone_{\P(K)}(Y_0) \subset \P^{N + k + 1}$ be an iterated cone over~$Y_0$ and let~$Q \subset \P^{N + k + 1}$ be a quadric.
\begin{enumerate}[label={\textup{(\alph*)}}, wide]
\item
\label{it:cone}
$Y$ has terminal singularities if and only if~$Y_0$ has terminal singularities and~$m_0 \ge 2$.
\item 
\label{it:cone-quadric-general}
If~$Q$ is general then~$Y \cap Q$ has terminal singularities if and only if~$Y$ has.
\item 
\label{it:cone-quadric-m1}
If~$\P(K) \subset Q$ and~$Q$ is general with this property then~$Y \cap Q$ has terminal singularities
if and only if~$Y_0$ has terminal singularities and~$m_0 \ge 3$.
\item 
\label{it:cone-quadric-m2}
If~$Q$ has multiplicity~$2$ along~$\P(K)$ and~$Q$ is general with this property then~$Y \cap Q$ has terminal singularities
if and only if~$Y_0$ has terminal singularities and~$m_0 \ge 4$.
\end{enumerate}
\end{lemma}

\begin{proof}
\ref{it:cone}
Let~$\tY \to Y$ be the blowup of the vertex~$\P(K) \subset Y$ of the cone 
and let~$E \subset \tY$ be the exceptional divisor of the blowup, 
Then
\begin{equation*}
\tY \cong \P_{Y_0}\Big((K \otimes \cO_{Y_0}) \oplus \cO_{Y_0}(-1)\Big)
\qquad\text{and}\qquad 
E = \P_{Y_0}(K \otimes \cO_{Y_0}) \cong \P(K) \times Y_0.
\end{equation*}
The morphism~$\tY \to Y_0$ is smooth, hence~$\tY$ has terminal singularties if and only if~$Y_0$ does,
and hence~$Y$ is terminal if and only if~$Y_0$ is terminal
and the discrepancy of~$E$ is positive.

Furthermore, if~$H$ denotes the pullback to~$\tY$ of the hyperplane class of~$Y$ 
then the pullback of the hyperplane class of~$Y_0$ is equal to~$H - E$, 
hence the projective bundle formula gives
\begin{equation*}
K_\tY = 
-m_0(H - E) + (H-E) - (k+1)H =
-(m_0 + k)H + (m_0 - 1)E.
\end{equation*}
This proves that~$Y$ is Gorenstein and the discrepancy of~$E$ is~$m_0 - 1$;
in particular, $Y$ is terminal if and only if~$Y_0$ is terminal and~$m_0 - 1 \ge 1$.

\ref{it:cone-quadric-general}
This follows from~\ref{it:cone} by Bertini's Theorem (see~\cite[Lemma~5.17]{KM}).

\ref{it:cone-quadric-m1}
Again, $Y \cap Q$ is terminal away from~$\P(K)$ by Bertini's theorem.
On the other hand, the strict transform of~$Y \cap Q$ in~$\tY$ is a divisor of class~$2H - E$, hence its canonical class is
\begin{equation*}
-(m_0 + k)H + (m_0 - 1)E + (2H - E) = 
-(m_0 + k - 2)H + (m_0 - 2)E,
\end{equation*}
and we conclude that~$Y \cap Q$ is terminal if and only if~$m_0 - 2 \ge 1$.

\ref{it:cone-quadric-m2}
Now the class of the strict transform of~$Y \cap Q$ in~$\tY$ is~$2H - 2E$, hence its canonical class is
\begin{equation*}
-(m_0 + k)H + (m_0 - 1)E + (2H - 2E) = 
-(m_0 + k - 2)H + (m_0 - 3)E,
\end{equation*}
and the same argument as before shows that the terminality criterion is~$m_0 - 3 \ge 1$.
\end{proof}

\subsection{Cones and linear sections} 
\label{ss:extension}

In this section we make a simple observation which underlies our extension argument. 
We start with some preparations.

Recall that a projective variety~$Y \subset \P^N$ is called {\sf arithmetically Cohen--Macaulay} ({\sf ACM})
if its homogeneous coordinate ring (the quotient of the coordinate ring of~$\P^N$ by the homogeneous ideal of~$Y$) is Cohen--Macaulay.
Note that if~$Y$ is ACM, it is a fortiori Cohen--Macaulay.

\begin{lemma}
\label{lem:acm-cone}
If~$Y \subset \P^N$ is an ACM scheme, the cone~$\Cone(Y) \subset \P^{N+1}$ is ACM.
\end{lemma}

\begin{proof}
The coordinate ring of the cone~$\Cone(Y)$ is the polynomial algebra over the coordinate ring of~$Y$,
hence the former is Cohen--Macaulay if the latter is.
\end{proof}

\begin{lemma}
\label{lem:mg-acm}
The Mukai varieties~$\rM_g \subset \P^{N_g}$ for~$g \in \{6,7,8,9,10,12\}$ are ACM.
\end{lemma}

\begin{proof}
By~\cite[Theorem~5]{Rama} any homogeneous variety~$\rG/\rP$ is ACM;
this proves the result in all cases except for~$\rM_{12}$.
In the latter case the ACM property is proved in Corollary~\ref{cor:m12-acm}.
\end{proof}

We will also use the following standard fact about linear sections of projective varieties and regular sequences in Cohen--Macaulay rings, 
see, e.g., \cite[Lemma~02JN]{stacks-project}.

\begin{lemma} 
\label{lem:reg-sequence-CM} 
Let~$Y \subset \P(W)$ be an ACM variety, let~$W' \subset W$ be a vector subspace of codimension~\mbox{$k \le \dim(Y)$}, 
and let~$x_1, \dots, x_k \in W^\vee$ be an independent set of linear equations for~$W'$.
Then~$Y \cap \P(W')$ is a dimensionally transverse intersection if and only if~$(x_1, \dots, x_k)$ 
is a regular sequence in the homogeneous coordinate ring of~$Y$.
\end{lemma}

We are now ready to prove the main result.

\begin{proposition} 
\label{prop:s-x}
Let~$Y \subset \P(W)$ be an ACM variety 
and let~$W_2 \hookrightarrow W_1 \twoheadrightarrow W_0 \hookrightarrow W$ be a chain of linear maps 
such that the composition~$W_2 \to W$ is injective 
and~$Y_2 \coloneqq Y \cap \P(W_2)$ is nonempty and dimensionally transverse.
If~$X \subset \P(W_1)$ is a Cohen--Macaulay subvariety such that
\begin{renumerate}
\item 
\label{it:x-cap-p2}
$X \cap \P(W_2) = Y_2$ is a dimensionally transverse intersection in~$\P(W_1)$ and
\item 
\label{it:x-in-y}
the image of~$X$ under the rational map~$\P(W_1) \dashrightarrow \P(W)$ is contained in~$Y$.
\end{renumerate}
Then~$Y_0 \coloneqq Y \cap \P(W_0)$ is a dimensionally transverse intersection and~$X = \Cone_{\P(\Ker(W_1 \to W_0))}(Y_0)$.
\end{proposition}

\begin{proof}
Consider the chain of embeddings~$W_2 \hookrightarrow W_0 \hookrightarrow W$.
Let~$x_1, \dots, x_k$ be an independent set of linear equations of~$W_0  \subset W$, 
and extend it to an independent set~$x_1, \dots, x_m$ of linear equations of~$W_2 \subset W$. 
Applying Lemma~\ref{lem:reg-sequence-CM} to the intersections~$Y \cap \P(W_2)$ and~$Y \cap \P(W_0)$ inside~$\P(W)$, 
we first see that~$(x_1, \dots, x_m)$ is a regular sequence in the homogeneous coordinate ring of~$Y$, 
hence so is the subsequence~$(x_1, \dots, x_k)$, 
and therefore~$Y_0 \coloneqq Y \cap \P(W_0)$ is dimensionally transverse.
 
To prove the second claim, we denote~$K \coloneqq \Ker(W_1 \to W_0)$ 
and rewrite the composition of linear maps~$W_1 \twoheadrightarrow W_0 \hookrightarrow W$ 
as~$W_1 \hookrightarrow \tW \twoheadrightarrow W$ 
in such a way that~$\Ker(\tW \to W) = K$,
replace~$Y \subset \P(W)$ by~$\Cone_{\P(K)}(Y) \subset \P(\tW)$, 
which is also ACM by Lemma~\ref{lem:acm-cone}
and consider the chain of maps~$W_2 \hookrightarrow W_1 \hookrightarrow \tW$.
Since~$W_2 \to W$ is injective, we have~$W_2 \cap K = 0$, hence
\begin{equation*}
\Cone_{\P(K)}(Y) \cap \P(W_2) = Y \cap \P(W_2) = Y_2
\end{equation*}
is a dimensionally transverse intersection.
Therefore, the argument of the first paragraph shows that~$\Cone_{\P(K)}(Y) \cap \P(W_1)$ 
is a dimensionally transverse intersection, 
hence it is Cohen--Macaulay of dimension~$\dim(Y_2) + \dim(W_1/W_2)$ and degree~$\deg(Y_2)$.
Since~$X$ is Cohen--Macaulay and~$X \cap \P(W_2)$ is dimensionally transverse in~$\P(W_1)$ by~\ref{it:x-cap-p2}, 
it follows that~$X$ has the same dimension and degree.
Finally, we have~\mbox{$X \subset \Cone_{\P(K)}(Y) \cap \P(W_1)$} by condition~\ref{it:x-in-y},
hence
\begin{equation*}
X = \Cone_{\P(K)}(Y) \cap \P(W_1),
\end{equation*}
and since~$K \subset W_1$, we have~$\Cone_{\P(K)}(Y) \cap \P(W_1) = \Cone_{\P(K)}(Y \cap \P(W_0)) = \Cone_{\P(K)}(Y_0)$.
\end{proof}

\subsection{Fano threefolds}
\label{ss:fano-3}

In this section we prove Theorem~\ref{thm:main} for Fano threefolds.

So, let~$X$ be a Fano threefold of genus~$g \ge 6$ satisfying assumptions of Theorem~\ref{thm:main}.
Let~$S \subset X$ be a very general anticanonical divisor, which is a prime K3 surface of genus~$g$ by Lemma~\ref{lem:hva}
and let~$\gamma_S \colon S \hookrightarrow \Gr(r, r + s)$ be its Gushel embedding constructed in Lemma~\ref{lem:gamma-x-s}.
Recall that~$\gamma_S$ factors through a unique Mukai subvariety~$\rM_g \subset \Gr(r, r+s)$ 
by Proposition~\ref{prop:sigma} and Theorem~\ref{thm:g7-emb}.
We first show that the morphism~$\gamma_S$ extends to a rational map from~$X$.

We denote by~$\cO_{\rM_g}(1)$ the restriction of the Pl\"ucker line bundle from~$\Gr(r,r + s)$ for~$g \ne 7$
and the spinor line bundle of~$\rM_7 = \OGr_+(5,10)$, respectively.

\begin{lemma}
\label{lem:s-x}
Let~$X$ be a Fano threefold with at most factorial terminal singularities 
such that~\eqref{eq:prime-fano} holds with~$g \in \{6,7,8,9,10,12\}$,
so that~$X \subset \P^{g+1}$ and a very general hyperplane section~\mbox{$S \subset X$}
is a smooth prime $K3$ surface of genus~$g$. 
Let~$X_\sm \subset X$ be the smooth locus of~$X$, so that~$S \subset X_\sm$.

There is a stable maximal Cohen--Macaulay sheaf~$\cU_X$ on~$X$ extending the Mukai bundle~$\cU_S$ on~$S$
and inducing a regular morphism~$\gamma_{X_\sm} \colon X_\sm \to \rM_g$
extending the Gushel map~$\gamma_S$ of~$S$, i.e.,
\begin{equation*}
\cU_X\vert_S \cong \cU_S
\qquad\text{and}\qquad
\gamma_{X_\sm}\vert_S = \gamma_S.
\end{equation*}
In particular, 
$\gamma_{X_\sm}^*(\cO_{\rM_g}(1)) \cong \cO_X(H)\vert_{X_\sm}$.
\end{lemma}

\begin{proof}
The fact that~$(S,H\vert_S)$ is a prime K3 surface of genus~$g$ is proved in Lemma~\ref{lem:hva}\ref{item:genK3}.

First, consider the case~$g \ne 7$.
We proved in~\cite[Theorem~5.3]{BKM} that the Mukai bundle~$\cU_S$ of~$S$ 
extends to a stable maximal Cohen--Macaulay (hence locally free on the smooth locus) sheaf~$\cU_X$ on~$X$ 
whose dual sheaf~$\cU_X^\vee$ is globally generated with~$\rH^0(X, \cU_X^\vee) = \rH^0(S, \cU_S^\vee)$. 
Therefore, the Gushel morphism~$\gamma_S \colon S \to \Gr(r,r+s)$ 
extends to the smooth locus~$X_\sm \subset X$ as a regular morphism~$\gamma_{X_\sm} \colon X_\sm \to \Gr(r,r+s)$.
Moreover, 
\begin{equation*}
\gamma_{X_\sm}^*(\cO_{\Gr(r,r+s)}(1))\vert_S \cong 
\gamma_S^*(\cO_{\Gr(r,r+s)}(1)) \cong 
\cO_S(H\vert_S) \cong
\cO_X(H)\vert_S,
\end{equation*}
where the first isomorphism holds because~$\gamma_{X_\sm}$ extends~$\gamma_S$,
the second holds by definition of the Gushel morphism~$\gamma_S$, and the third is obvious.
Therefore, $\gamma_{X_\sm}^*(\cO_{\Gr(r,r+s)}(1)) \cong \cO_X(H)\vert_{X_\sm}$ 
(because~$\Pic(X) = \ZZ \cdot H$ and~$X \setminus X_\sm$ has codimension greater than~1 in~$X$).

So, for~$g \ne 7$, it only remains to check that~$\gamma_{X_\sm}(X_\sm) \subset \rM_g$, and we may assume~$g \in \{9,10,12\}$.
Since~$\rM_g \subset \Gr(r,r+s)$ is defined as the zero locus of a global section
of the vector bundle~$\cE_0$ on~$\Gr(r,r+s)$ (see~\eqref{eq:def-ce-zero}) 
that vanishes on~$S$ (see Lemma~\ref{lem:rm-9-10}),
it is enough to check that the same section vanishes on~$X_\sm$.
Thus, it is enough to check that the restriction morphism
\begin{equation*}
\rH^0(X_\sm, \gamma_{X_\sm}^*\cE_0) \to \rH^0(S, \gamma_{X_\sm}^*\cE_0\vert_S)
\end{equation*}
is injective, or equivalently that~$\rH^0(X_\sm, \gamma_{X_\sm}^*(\cE_0(-1))) = 0$.
Now, by~\eqref{eq:ce0-vee} and the definition of the morphism~$\gamma_{X_\sm}$, we have
\begin{itemize}
\item 
if~$g = \hphantom{0}9$ 
then~$\gamma_{X_\sm}^*(\cE_0(-1)) \cong \cU_{X_\sm}$,
\item 
if~$g = 10$ 
then~$\gamma_{X_\sm}^*(\cE_0(-1)) \cong \cU_{X_\sm}^\perp$, and
\item 
if~$g = 12$ 
then~$\gamma_{X_\sm}^*(\cE_0(-1)) \cong \cU_{X_\sm}^{\oplus 3}$.
\end{itemize}
Note that~$\cU_{X_\sm}$ and~$\cU_{X_\sm}^\perp$ are stable of negative slope (because their restrictions to~$S$ are),
hence they have no global sections on~$X_\sm$, and we conclude that~$\gamma_{X_\sm}(X_\sm) \subset \rM_g$.

Now, consider the case~$g = 7$. Then~$X \subset \P^8$ is an intersection of quadrics (by Lemma~\ref{lem:hva}),
\mbox{$S = X \cap \P^7$}, hence~$\cN_{S/\P^7} \cong \cN_{X_\sm/\P^8}\vert_S$.
Therefore, if we define the bundle~$\cW_{X_\sm}$ on~$X_\sm$ as
\begin{equation*}
\cU_{X_\sm} \coloneqq \cN_{X_\sm/\P^8}(-2),
\end{equation*}
then~$\cU_{X_\sm}^\vee$ is globally generated 
and Corollary~\ref{cor:cw-cn} and the argument of~\cite[Lemma~A.5]{DK} 
imply that~$\cU_{X_\sm}\vert_S \cong \cU_S$
and~$\rH^0(X_\sm, \cU_{X_\sm}^\vee) = \rH^0(S, \cU_S^\vee)$.
It follows that the Gushel morphism~\mbox{$\gamma_S \colon S \hookrightarrow \OGr_+(5,10) \to \Gr(5,10)$}
extends to a regular morphism~$\gamma_{X_\sm} \colon X_\sm \to \Gr(5,10)$.
As before, to prove the inclusion~$\gamma_{X_\sm}(X_\sm) \subset \OGr_+(5,10)$
it is enough to check that the morphism
\begin{equation*}
\rH^0(X_\sm, \Sym^2\cU_{X_\sm}^\vee) \to \rH^0(S, \Sym^2\cU_S^\vee) 
\end{equation*}
is injective,
i.e., that~$\rH^0(X_\sm, \Sym^2\cU_{X_\sm}^\vee \otimes \cO_{X_\sm}(-H)) = 0$.
For this just note that stability of~$\cU_S$ (see Theorem~\ref{thm:g7-emb}) implies stability of~$\cU_{X_\sm}$,
hence also semistability of~\mbox{$\Sym^2\cU_{X_\sm}^\vee \otimes \cO_{X_\sm}(-H)$},
and since the slope of this sheaf is equal to~$4/5 - 1 = -1/5$, its space of global sections vanishes. 

The isomorphism~$\gamma_{X_\sm}^*(\cO_{\OGr_+(5,10)}(1)) \cong \cO_X(H)\vert_{X_\sm}$ is verified as in the previous case.
\end{proof}

We can now prove Theorem~\ref{thm:main} for Fano threefolds.
Recall that~$\rM_6 = \Gr(2,5) \subset \P^9$.
Recall also the definition of a Mukai bundle on a Fano threefold (\cite[Definition~5.1]{BKM}).

\begin{corollary}
\label{cor:threefolds}
If~$X$ is a prime Fano threefold with at most factorial terminal singularities 
over an algebraically closed field of characteristic zero of genus~$g \in \{6,7,8,9,10,12\}$
then
\begin{equation*}
X = 
\begin{cases}
\rM_g \cap \P^{g+1}, & \text{if~$g \in \{7,8,9,10,12\}$},\\
\Cone(\Gr(2,5)) \cap \P^7 \cap Q, & \text{if~$g = 6$},
\end{cases}
\end{equation*}
is a dimensionally transverse intersection,
where in the case~$g = 6$ we denote by~$Q$ a quadric not containing the vertex of~$\Cone(\Gr(2,5)) \subset \P^{10}$.

Moreover, the sheaf~$\cU_X$ constructed in Lemma~\textup{\ref{lem:s-x}} 
is the unique Mukai bundle of type~$(r,s)$ on~$X$;
in particular~$\cU_X$ is locally free, acyclic, and exceptional and~$\cU_X^\vee$ is globally generated.
\end{corollary}

\begin{proof}
By Lemma~\ref{lem:s-x} there is a regular morphism~$\gamma_{X_\sm} \colon X_\sm \to \rM_g$ 
from the smooth locus~$X_\sm$ of~$X$
extending the Gushel embedding~$\gamma_S \colon S \hookrightarrow \rM_g$ of a very general K3 surface~\mbox{$S \subset X_\sm \subset X$}. 
We consider~$\gamma_{X_\sm}$ as a rational map~$X \dashrightarrow \rM_g$ and denote it by~$\gamma_X$.
We also denote
\begin{equation}
\label{eq:ws-wx-w}
W_S \coloneqq \rH^0(S, \cO_S(H\vert_S))^\vee,
\qquad
W_X \coloneqq \rH^0(X, \cO_X(H))^\vee,
\qquad 
W \coloneqq \rH^0(\rM_g, \cO_{\rM_g}(1))^\vee.
\end{equation}
Since~$\gamma_{X_\sm}^*(\cO_{\rM_g}(1)) \cong \cO_X(H)\vert_{X_\sm}$ by Lemma~\ref{lem:s-x},
the restriction morphisms induce a chain of linear maps~$W_S \hookrightarrow W_X \to W$,
where the first arrow and the composition of the arrows are both injective
(this is obvious for the first arrow, because~$S$ is a transverse linear section of~$X$, 
and for the composition this follows from Theorem~\ref{thm:prime-k3}).
Therefore, we obtain a commutative diagram
\begin{equation}
\label{eq:sxv-diagram}
\vcenter{\xymatrix@C=3em{
S \ar@/^3ex/[rr]^{\gamma_S} \ar[r] \ar@{^{(}->}[d] &
X \ar@{-->}[r]_{\gamma_X} \ar@{^{(}->}[d] &
\rM_g \ar@{^{(}->}[d]
\\
\P(W_S) \ar@{^{(}->}[r] 
&
\P(W_X) \ar@{-->}[r] &
\P(W),
}}
\end{equation}
where the maps in the bottom row are induced by the linear maps~$W_S \hookrightarrow W_X \to W$.
Let 
\begin{equation}
\label{eq:k}
K \coloneqq \Ker(W_X \to W).
\end{equation}
The injectivity of~$W_S \to W$ implies~$W_S \cap K = 0$, 
and since~$\dim(W_X/W_S) = 1$, we have~$\dim(K) \le 1$.
Moreover, if~$\dim(K) = 1$ then~$W_X = W_S \oplus K$, hence the image~$W_0$ of~$W_X$ in~$W$ is equal to~$W_S$.

First, assume~$g \ge 7$.
Applying Proposition~\ref{prop:s-x} with~$W_2 = W_S$, $W_1 = W_X$, $W_0 = W_X/K$, 
and~$Y = \rM_g$ (which is ACM by Lemma~\ref{lem:mg-acm}), so that~$Y_2 = S$, 
we conclude that either
\begin{itemize}[wide]
\item
$K = 0$, $W_0 = W_X \subset W$, and~$X = \rM_g \cap \P(W_X)$ is a dimensionally transverse intersection, or
\item 
$\dim(K) = 1$, $W_0 = W_X/K = W_S$ and~$X = \Cone(\rM_g \cap \P(W_0)) = \Cone(S)$.
\end{itemize}
It remains to note that 
the singularity of the cone~$\Cone(S)$ at the vertex is worse than terminal
(even worse than canonical) by Lemma~\ref{lem:cone-singularity}, 
hence this case is impossible.

Next, assume~$g = 6$ and~$K = 0$.
Then the composition of arrows~$X \to \P(W_X) \to \P(W)$ in~\eqref{eq:sxv-diagram} is a closed embedding,
hence the map~$\gamma_X$ is a regular closed embedding.
Moreover, in this case~$M_X \coloneqq \Gr(2,5) \cap \P(W_X)$ is a quintic del Pezzo fourfold, 
and since its hyperplane section~$M_S \coloneqq \Gr(2,5) \cap \P(W_S)$ is smooth (see Proposition~\ref{prop:s-g6}), 
$M_X$ is also smooth (by~\cite[Proposition~2.24]{DK}),
hence its Picard group is generated by the hyperplane class.
Since~$X \subset M_X$ is a divisor of degree~$10$, we conclude that~$X = M_X \cap Q$, where~$Q$ is a quadric.

Similarly, assume~$g = 6$ and~$\dim(K) = 1$.
As we explained above, in this case~$W_0 = W_S$, hence~$\gamma_{X}(X) \subset M_S \coloneqq \Gr(2,5) \cap \P(W_S)$
and it follows from~\eqref{eq:sxv-diagram} that~$X \subset M_X \coloneqq \Cone_{\P(K)}(M_S)$.
Since~$M_S$ is a smooth quintic del Pezzo threefold (by Proposition~\ref{prop:s-g6}),
$M_X$ is a factorial fourfold of degree~5 with the Picard group generated by the hyperplane class,
and~$X$ is a divisor in~$M_X$ of degree~10, hence~$X = M_X \cap Q$, where~$Q$ is a quadric,
which does not contain the vertex of the cone by Lemma~\ref{lem:cone-singularity}.

Note that both the smooth quintic del Pezzo fourfold and the cone over a smooth quintic del Pezzo threefold
can be written uniformly as linear sections of~$\Cone(\Gr(2,5))$ by~$\P^7$, 
hence we obtain the required description for these two cases.

As a consequence of the above analysis we see that for all~$g \in \{6,7,8,9,10,12\}$ the Gushel map~$\gamma_X$ is regular,
hence the maximal Cohen--Macaulay sheaf~$\cU_X$ constructed in Lemma~\ref{lem:s-x} 
is isomorphic on~$X_\sm$ to the restriction of the vector bundle~$\gamma_X^*\cU$.
Since~$\codim(X \setminus X_\sm) > 1$, we conclude that~$\cU_X \cong \gamma_X^*(\cU)$; in particular, $\cU_X$ is locally free.
Finally, the sheaf~$\cU_X$ is acyclic and its dual is globally generated by~\cite[Theorem~5.3]{BKM},
and~$\cU_X$ is exceptional and unique by~\cite[Corollary~5.9]{BKM}.
\end{proof}

\subsection{Higher dimensions}
\label{ss:fano-all}

Let~$X$ be a Fano variety of dimension~$n \ge 4$ and genus~$g \ge 6$ 
satisfying assumptions of Theorem~\ref{thm:main}.
Recall notation~\eqref{eq:ws-wx-w}. 
By Lemma~\ref{lem:hva} we have~$X \subset \P(W_X) = \P^{n + g - 2}$
and a very general linear section~$S = X \cap \P(W_S) = X \cap \P^g$ 
is a prime K3 surface of genus~$g$.
On the other hand, by Theorem~\ref{thm:prime-k3} we have
\begin{equation*}
S = \rM_g \cap \P(W_S)
\qquad\text{or}\qquad 
S = \Gr(2,5) \cap \P(W_S) \cap Q.
\end{equation*}

Consider the blowup
\begin{equation*}
\cX \coloneqq \Bl_S(X) \xrightarrow{\ p\ } \P(W_X/W_S) = \P^{n-3}.
\end{equation*}
The argument of Lemma~\ref{lem:hva} shows that a general intersection~$X \cap \P^{g+1}$ containing~$S$ 
is a Fano threefold satisfying the assumptions of Theorem~\ref{thm:main}.
Let~$U \subset \P(W_{X}/W_S)$ be the open subset that parameterizes such linear sections,
let~$\cX_U \coloneqq p^{-1}(U) \subset \cX$, 
and for~$u \in U$ let~$\cX_u \coloneqq p^{-1}(u)$,
so that~$\cX_u = X \cap \P^{g + 1}$ is a Fano threefold as in Theorem~\ref{thm:main}.
Applying Corollary~\ref{cor:threefolds} we conclude that each~$\cX_u$  is endowed with a unique Mukai bundle~$\cU_{\cX_u}$.
Note also that the exceptional divisor~$E \subset \cX$ of the blowup~$\cX \to X$ is isomorphic to~$S \times \P(W_{X}/W_S)$
and~$E_U \coloneqq E \cap \cX_U \cong S \times U$.

\begin{lemma}
\label{lem:cu-cx-u}
Let~$X$ be a Fano variety of dimension~$n \ge 4$ as in Theorem~\textup{\ref{thm:main}}.
There is a unique vector bundle~$\cU_{\cX_U}$ on~$\cX_U$ such that
\begin{renumerate}
\item 
\label{it:cu-xh}
$\cU_{\cX_U}\vert_{\cX_u} \cong \cU_{\cX_u}$ for any~$u \in U$, and 
\item 
\label{it:cu-su}
$\cU_{\cX_U}\vert_{E_U} \cong \cU_S \boxtimes \cO_U$.
\end{renumerate}
Moreover, $p_*(\cU_{\cX_U}^\vee) \cong \cO_U^{\oplus(r+s)}$
and the natural map~$\cO_{\cX_U}^{\oplus(r+s)} \cong p^*p_*(\cU_{\cX_U}^\vee) \to \cU_{\cX_U}^\vee$ is surjective.
\end{lemma}

\begin{proof}
Consider the relative moduli space~$\cM$ over~$U$ of stable vector bundles on fibers of~$\cX_U \to U$
with the same Hilbert polynomials as those of Mukai bundles of the fibers.
Since the Mukai bundles are exceptional (by Corollary~\ref{cor:threefolds}), 
$\cM$ is \'etale over~$U$ at every point corresponding to a Mukai bundle~$\cU_{\cX_u}$ (see~\cite[Theorem~3.7]{K22})
and since the Mukai bundles are unique, the moduli space has a connected component $\cM_{\mathrm{Muk}} \subset \cM$ parametrizing Mukai bundles such that the morphism~$\cM_{\mathrm{Muk}} \to U$ is bijective.
Therefore, $\cM_{\mathrm{Muk}} \cong U$, see~\cite[Corollary~3.8]{K22}
and by~\cite[Proposition~3.11]{K22} there is a Brauer class~$\upbeta \in \Br(U)$ 
and a $p^*(\upbeta)$-twisted universal bundle~$\cU_{\cX_U}$ on~$\cX_U$ such that~\ref{it:cu-xh} holds,
unique up to a twist by a line bundle pulled back from~$U$.

Furthermore, it follows that~$\cU_{\cX_U}\vert_{S \times \{u\}} \cong \cU_{\cX_u}\vert_S \cong \cU_S$,
hence the sheaf
\begin{equation*}
\cL \coloneqq p_*((\cU_S^\vee \boxtimes \cO_U) \otimes \cU_{\cX_U}\vert_{E_U}) 
\end{equation*}
is a $\upbeta$-twisted line bundle on~$U$.
Therefore, $\upbeta = 1$ by~\cite[Corollary~2.8]{K22}, hence~$\cL$ and~$\cU_{\cX_U}$ are untwisted.
Moreover, the canonical morphism~$\cU_S \boxtimes \cL \to \cU_{\cX_U}\vert_{E_U}$ is an isomorphism.
So, tensoring~$\cU_{\cX_U}$ with~$p^*\cL^{-1}$ we obtain a vector bundle on~$\cX_U$
for which both~\ref{it:cu-xh} and~\ref{it:cu-su} hold.
Note also that condition~\ref{it:cu-su} ensures the uniqueness of~$\cU_{\cX_U}$.

To compute~$p_*(\cU_{\cX_U}^\vee)$ we consider the natural exact sequence 
\begin{equation*}
0 \to \cU_{\cX_U}^\vee(-E_U) \to \cU_{\cX_U}^\vee \to \cU_{\cX_U}^\vee\vert_{E_U} \to 0.
\end{equation*}
For each~\mbox{$u \in U$} the restriction of the first term to~$\cX_u$ is isomorphic to~$\cU_{\cX_u}^\vee(-H)$.
Since~$\cU_{\cX_u}$ is acyclic (by Corollary~\ref{cor:threefolds}),
Serre duality implies that~$\cU_{\cX_u}^\vee(-H)$ has no cohomology, hence 
\begin{equation*}
p_*(\cU_{\cX_U}^\vee) \cong 
p_*(\cU_{\cX_U}^\vee\vert_{E_U}) \cong 
p_*(\cU_S^\vee \boxtimes \cO_U) \cong
\rH^0(S, \cU_S^\vee) \otimes \cO_U \cong
\cO_U^{\oplus(r+s)}.
\end{equation*}
Since, moreover, $\cU_{\cX_u}^\vee$ is globally generated for each~$u \in U$ (by Corollary~\ref{cor:threefolds}), 
the bundle~$\cU_{\cX_U}^\vee$ is relatively globally generated over~$U$.
\end{proof}

Note that the restriction of the blowup morphism~$\pi \colon \cX = \Bl_S(X) \to X$
to any fiber~$\cX_u \subset \cX$ of~$p \colon \cX \to \P^{n-3}$ 
is an isomorphism~$\cX_u \xrightiso{} \pi(\cX_u) = X \cap \P^{g+1}$ onto a linear section of~$X$.

\begin{corollary}
\label{cor:gushel}
Let~$X$ be a Fano variety of dimension~$n \ge 4$ and genus~$g \ge 6$ as in Theorem~\textup{\ref{thm:main}}.
There is a unique rational map 
\begin{equation*}
\gamma_X \colon X \dashrightarrow \rM_g
\end{equation*}
such that~$\gamma_X\vert_{\pi(\cX_U)}$ is regular and~$\gamma_X\vert_{\cX_u}$
coincides with~$\gamma_{\cX_u}$ on~$\cX_u$ for each~$u \in U$.
\end{corollary}

\begin{proof}
We use notation introduced in the proof of Lemma~\ref{lem:cu-cx-u}.
Consider the morphism
\begin{equation*}
\gamma_{\cX_U} \colon \cX_U \to \Gr(r, r + s)
\end{equation*}
given by the globally generated vector bundle~$\cU_{\cX_U}$.
We denote by~$\gamma_\cX \colon \cX \dashrightarrow \Gr(r, r + s)$  
the rational extension of~$\gamma_{\cX_U}$ from the dense open subset~$\cX_U \subset \cX$.
By Lemma~\ref{lem:cu-cx-u} its restriction to each fiber~$\cX_u$ of~$\cX_U$ over~$u \in U$ 
coincides with the Gushel morphism of~$\cX_u$
and its restriction to the exceptional divisor~$E_U = S \times U \subset \cX_U$ of~$\cX_U$ over~$\cX$ factors as the composition
\begin{equation*}
S \times U \xrightarrow{\ \pr_1\ } S \xrightarrow{\ \gamma_S\ } \Gr(r, r + s)
\end{equation*}
of the projection onto the first factor and the Gushel morphism of~$S$.
Therefore, the map~$\gamma_\cX$
contracts the exceptional divisor~$E = S \times \P(W/W_S) \subset \cX$ along the first projection,
hence~$\gamma_\cX$ factors as the composition
\begin{equation*}
\cX \xrightarrow\quad X \dashrightarrow \Gr(r, r + s)
\end{equation*}
of the blowup morphism and a rational map~$\gamma_X \colon X \dashrightarrow \Gr(r, r + s)$.
The restriction of the map~$\gamma_X$ to~$\cX_u$ for~$u \in U$ coincides with the restriction of~$\gamma_\cX$,
hence it is regular and coincides with~$\gamma_{\cX_u}$.

Furthermore, Corollary~\ref{cor:threefolds} shows
that the image of~$\gamma_{\cX_u}$ is contained in an appropriate Mukai subvariety~$\rM_g \subset \Gr(r, r + s)$
which a priori may depend on~$u$.
Since, on the other hand, $S \subset \cX_u$ for each~$u$ and by Proposition~\ref{prop:sigma} 
there is only one Mukai subvariety that contains~$\gamma_{\cX_u}(S) = \gamma_S(S)$,
it follows that~$\rM_g \subset \Gr(r, r + s)$ is the same for all points~$u \in U$.
It follows that~$\gamma_{\cX_U}(\cX_U) \subset \rM_g$,
and therefore~$\gamma_X(X) \subset \rM_g$.

The uniqueness of~$\gamma_X$ is obvious because~$\pi(\cX_U)$ is dense in~$X$.
\end{proof}

Finally, we can prove our main result. Note that in the case~$g = 6$, this extends~\cite[Theorem~2.16]{DK} from the case where~$\codim(\Sing(X)) \ge 4$  to~$\codim(\Sing(X)) \ge 3$.

\begin{proof}[Proof of Theorem~\textup{\ref{thm:main}}]
The case where~$n \coloneqq \dim(X) = 3$ is covered by Corollary~\ref{cor:threefolds}, so we assume~$n \ge 4$.
Arguing as in the proof of Corollary~\ref{cor:threefolds} (with Lemma~\ref{lem:s-x} replaced by Corollary~\ref{cor:gushel}),
we construct the diagram~\eqref{eq:sxv-diagram} (where we use notation~\eqref{eq:ws-wx-w}),
that has similar properties, and define the subspace~$K \subset W_X$ by~\eqref{eq:k}.
We still have~$W_S \cap K = 0$, but now we have~$\dim(W_X/W_S) = n - 2$, hence~$\dim(K) \le n - 2$.
Note also that the space~$W_0 \coloneqq W_X/K$ fits into a chain of linear maps~$W_S \hookrightarrow W_X \twoheadrightarrow W_0 \hookrightarrow W$.

If~$g \in \{7,8,9,10,12\}$, we apply Proposition~\ref{prop:s-x}
with~$W_2 = W_S$, $W_1 = W_X$, $W_0 = W_X/K$, and~$Y = \rM_g$ so that~$Y_2 = S$,
and since~$\rM_g$ is ACM by Lemma~\ref{lem:mg-acm}, conclude that
\begin{equation*}
X_0 \coloneqq \rM_g \cap \P(W_0)
\end{equation*}
is a dimensionally transverse intersection and 
\begin{equation*}
X = \Cone_{\P(K)}(X_0).
\end{equation*}
In particular, $X_0$ is normal and Gorenstein with~$\omega_{X_0} \cong \cO_{X_0}(2 - \dim(X_0))$.
Finally, applying Lemma~\ref{lem:cone-singularity}, we conclude that~$X$ has terminal singularities only if~$\dim(X_0) \ge 4$.

Similarly, if~$g = 6$, we conclude that
\begin{equation*}
X \subset \Cone_{\P(K)}(Y_0),
\qquad\text{where}\qquad
Y_0 = \Gr(2,5) \cap \P(W_0).
\end{equation*}
Since, moreover,~$Y_0 \cap \P(W_S) = \Gr(2,5) \cap \P(W_S)$ is a smooth quintic del Pezzo threefold (see Proposition~\ref{prop:s-g6}), 
we deduce from~\cite[Proposition~2.24]{DK} that~$Y_0$ is a smooth quintic del Pezzo variety;
in particular, $\omega_{Y_0} \cong \cO_{Y_0}(1 - \dim(Y_0))$;
and the intersection~$\Gr(2,5) \cap \P(W_0)$ is dimensionally transverse.
Furthermore, $\Cone_{\P(K)}(Y_0)$ is factorial and its class group is generated by the hyperplane class.
Since~$\deg(X) = 10$, we conclude that
\begin{equation*}
X = \Cone_{\P(K)}(Y_0) \cap Q,
\end{equation*}
where~$Q$ is a quadric. 
Finally, applying Lemma~\ref{lem:cone-singularity}, we conclude that
\begin{itemize}
\item 
if~$\dim(Y_0) = 3$, we have~$\P(K) \not\subset Q$, and
\item 
if~$\dim(Y_0) \le 4$, the multiplicity of~$Q$ along~$\P(K)$ does not exceed~1.
\end{itemize}
This completes the proof of the existence part of the theorem.

The uniqueness in Theorem~\ref{thm:main} follows from the uniqueness in Corollary~\ref{cor:gushel}.
\end{proof}

\begin{remark}
Corollary~\ref{cor:threefolds} shows that for all Fano threefolds~$X$ as in Theorem~\ref{thm:main} (including singular threefolds) 
the bundles~$(\cU_X,\cO_X)$ form an exceptional pair.
Similarly, for~$n = \dim(X) \ge 4$, if the cones are not involved, we obtain an exceptional collection of vector bundles
\begin{equation*}
(\cU_X(3-n), \cO_X(3-n), \dots, \cU_X, \cO_X)
\end{equation*}
of length~$2(n-2)$ in the derived category of~$X$; 
this can be proved by induction on~$n$ using the argument of~\cite[Lemma~5.8]{BKM},
or can be deduced from Theorem~\ref{thm:main} and Borel--Bott--Weil Theorem.
Finally, if~$X = \Cone_{\P^{n-n_0-1}}(X_0)$ with~$n > n_0$,
a similar exceptional collection can be constructed in the \emph{categorical cone} over~$X_0$, see~\cite{KP:cones}.
\end{remark}


\appendix

\section{Schubert and Borel--Bott--Weil computations}
\label{sec:bbw}

In this appendix we perform some Schubert and Borel--Bott--Weil computations on Grassmannians~$\Gr(r,V) = \Gr(r,r+s)$
that are used in the body of the paper for the cases~$g \in \{8,9,10,12\}$.

\begin{lemma}
\label{lem:ctop}
For~$g \in \{8,9,10,12\}$
if~$\cE$ is the vector bundle of rank~$g - 2$ on~$\Gr(r,V)$ defined by~\eqref{eq:def-ce} 
then the degree of $\rc_{g - 2}(\cE)$ with respect to the Pl\"ucker polarization is~\mbox{$2g - 2$}.
\end{lemma}

\begin{proof}
The definition of~$\cE$ implies that~$\rc_{g - 2}(\cE)$ and $\rc_{g - n_g}(\cE_0)$ have the same degree.

If~$g = 8$, the required number is the degree of~$\Gr(2,6)$, which is equal to~$14 = 2g - 2$.

If~$g = 9$, the splitting principle gives
\begin{equation*}
\rc_{g - n_g}(\cE_0) = 
\rc_3(\wedge^2\cU^\vee) = 
\rc_1(\cU^\vee)\rc_2(\cU^\vee) - \rc_3(\cU^\vee).
\end{equation*}
By Schubert calculus on~$\Gr(3,6)$ the terms are equal 
to the numbers of standard Young tableaux of shapes~$(3,2,2)$ and~$(2,2,2)$, respectively,
and by the hook length formula these are~21 and~5, hence finally the required number is~$21 - 5 = 16 = 2g - 2$.

If~$g = 10$, the splitting principle gives
\begin{equation*}
\rc_{g - n_g}(\cE_0) = 
\rc_5(\cU^\perp(1)) = 
\rc_1(V/\cU)^3 \rc_2(V/\cU) - \rc_1(V/\cU)^2 \rc_3(V/\cU) + \rc_1(V/\cU) \rc_4(V/\cU) - \rc_5(V/\cU).
\end{equation*}
By Schubert calculus on~$\Gr(2,7)$ the terms are equal 
to the numbers of standard Young tableaux of shapes~$(5,3)$, $(5,2)$, $(5,1)$, and~$(5,0)$, respectively,
and by the hook length formula these are~28, 14, 5, and~1, hence finally the required number is~$28 - 14 + 5 - 1 = 18 = 2g - 2$.

If~$g = 12$, the splitting principle gives
\begin{align*}
\rc_{g - n_g}(\cE_0) &= 
\rc_9((\wedge^2\cU^\vee)^{\oplus 3}) = 
\rc_3(\wedge^2\cU^\vee)^3 =
(\rc_1(\cU^\vee)\rc_2(\cU^\vee) - \rc_3(\cU^\vee))^3 \\ &= 
\rc_1(\cU^\vee)^3\rc_2(\cU^\vee)^3 - 
3\rc_1(\cU^\vee)^2\rc_2(\cU^\vee)^2\rc_3(\cU^\vee) + 
3\rc_1(\cU^\vee)\rc_2(\cU^\vee)\rc_3(\cU^\vee)^2 -
\rc_3(\cU^\vee)^3.
\end{align*}
By Schubert calculus on~$\Gr(3,7)$ the terms are equal to~47, 11, 3, and~1, 
hence finally the required number is~$47 - 33 + 9 -1 = 22 = 2g - 2$.
\end{proof}

\begin{proposition}
\label{prop:hi-we}
For~$g \in \{8,9,10,12\}$
let~$\cE_0$ be the vector bundle on~$\Gr(r,V)$ defined by~\eqref{eq:def-ce-zero}.
If~$n_g - 2 \ge j \ge 0$, we have
\begin{equation*}
\rH^p(\Gr(r,V), \wedge^i\cE_0^\vee(-j)) = 
\begin{cases}
\kk, & \text{if~$p = i = j = 0$ or~$p = g$, $i = g - n_g$, $j = n_g - 2$},\\
0, & \text{otherwise}.
\end{cases}
\end{equation*}
Moreover, if~$j < 0$, the cohomology is nontrivial only if~$p = 0$ and~$i \cdot \upmu(\cE_0) \le |j|$.
\end{proposition}

\begin{proof}
If~$g = 8$ then~$\cE_0 = 0$, so we are interested in~$\rH^p(\Gr(2,6), \cO(-j))$,
and the result is obvious.

If~$g = 9$ then~$\cE_0 \cong \wedge^2\cU^\vee$, hence~$\cE_0^\vee \cong \wedge^2\cU \cong \cU^\vee(-1)$.
Therefore, $\wedge^i\cE_0^\vee(-j) \cong \wedge^i\cU^\vee(-i-j)$ is an equivariant vector bundle 
and the corresponding $\GL_6$-weight is
\begin{equation*}
\lambda = (\alpha_1-i-j, \alpha_2-i-j, \alpha_3-i-j, 0, 0, 0),
\end{equation*}
where~$1 \ge \alpha_1 \ge \alpha_2 \ge \alpha_3 \ge 0$ and~$\alpha_1 + \alpha_2 + \alpha_3 = i$.
Consider the weight
\begin{equation*}
\lambda + \rho = (\alpha_1-i-j+6, \alpha_2-i-j+5, \alpha_3-i-j+4, 3, 2, 1).
\end{equation*}
By the Borel--Bott--Weil Theorem, the cohomology is non-trivial only if all entries of this weight are distinct.
Since the first three entries are decreasing and their pairwise differences satisfy
\begin{equation*}
(\alpha_k-i-j+(7-k)) - (\alpha_{k+1}-i-j+(7-k-1)) = \alpha_k - \alpha_{k+1} + 1 \le 2,
\end{equation*}
it follows that the cohomology is trivial unless
\begin{equation*}
\alpha_3-i-j+4 \ge 4
\qquad\text{or}\qquad 
\alpha_1-i-j+6 \le 0.
\end{equation*}

In the first case we obtain~$\alpha_3 \ge i + j$, and since on the other hand~$i \ge 3\alpha_3$, we obtain~$2i + 3j \le 0$.
Since~$i \ge 0$, it follows that~$j \le 0$ and~$\tfrac23i = i \cdot \upmu(\cE_0) \le |j|$; in this case~$p = 0$.

In the second case we obtain~$\alpha_1 \le i + j - 6$, and since~$i \le 3\alpha_1$, this gives~$2i + 3j \ge 18$.
As~$i \le \rank(\cE_0) = 3$ and~$j \le n_9 - 2 = 4$, the only possibility is~$i = 3$ and~$j = 4$;
in this case~$p = 9$.

If~$g = 10$ then~$\cE_0 \cong \cU^\perp(1)$, hence~$\cE_0^\vee \cong V/\cU(-1)$.
Therefore, $\wedge^i\cE_0^\vee(-j) \cong \wedge^i(V/\cU)(-i-j)$ is an equivariant vector bundle 
and the corresponding $\GL_7$-weight is
\begin{equation*}
\lambda = (0,0,i+j-\alpha_5, i+j-\alpha_4, i+j-\alpha_3, i+j-\alpha_2, i+j-\alpha_1),
\end{equation*}
where~$1 \ge \alpha_1 \ge \alpha_2 \ge \alpha_3 \ge \alpha_4 \ge \alpha_5 \ge 0$ 
and~$\alpha_1 + \alpha_2 + \alpha_3 + \alpha_4 + \alpha_5 = i$.
Consider the weight
\begin{equation*}
\lambda + \rho = (7,6,i+j-\alpha_5+5, i+j-\alpha_4+4, i+j-\alpha_3+3, i+j-\alpha_2+2, i+j-\alpha_1+1),
\end{equation*}
By the Borel--Bott--Weil Theorem, the cohomology is non-trivial only if all entries of this weight are distinct.
Since the last five entries are decreasing and their pairwise differences satisfy
\begin{equation*}
(i+j-\alpha_{k+1}+k+1) - (i+j-\alpha_k+k) = \alpha_k - \alpha_{k+1} + 1 \le 2,
\end{equation*}
it follows that the cohomology is trivial unless
\begin{equation*}
i+j-\alpha_5+5 \le 5
\qquad\text{or}\qquad 
i+j-\alpha_1+1 \ge 8.
\end{equation*}

In the first case we obtain~$\alpha_5 \ge i + j$, and since on the other hand~$i \ge 5\alpha_5$, we obtain~$4i + 5j \le 0$.
Since~$i \ge 0$, it follows that~$j \le 0$ and~$\tfrac45i = i \cdot \upmu(\cE_0) \le |j|$; in this case~$p = 0$.

In the second case we obtain~$\alpha_1 \le i + j - 7$, and since~$i \le 5\alpha_1$, this gives~$4i + 5j \ge 35$.
As~$i \le \rank(\cE_0) = 5$ and~$j \le n_{10} - 2 = 3$, the only possibility is~$i = 5$ and~$j = 3$;
in this case~$p = 10$.

If~$g = 12$ then~$\cE_0 \cong (\wedge^2\cU^\vee)^{\oplus 3}$, hence~$\cE_0^\vee \cong (\cU^\vee(-1))^{\oplus 3}$.
Therefore, $\wedge^i\cE_0^\vee(-j)$ is a direct sum of equivariant vector bundles 
and the corresponding $\GL_7$-weights are
\begin{equation*}
\lambda = (\alpha_1-i-j, \alpha_2-i-j, \alpha_3-i-j, 0, 0, 0,0),
\end{equation*}
where~$3 \ge \alpha_1 \ge \alpha_2 \ge \alpha_3 \ge 0$ and~$\alpha_1 + \alpha_2 + \alpha_3 = i$.
Consider the weight
\begin{equation*}
\lambda + \rho = (\alpha_1-i-j+7, \alpha_2-i-j+6, \alpha_3-i-j+5, 4, 3, 2, 1).
\end{equation*}
By the Borel--Bott--Weil Theorem, the cohomology is non-trivial only if all entries of this weight are distinct.
Since the first three entries are decreasing and their pairwise differences satisfy
\begin{equation*}
(\alpha_k-i-j+(8-k)) - (\alpha_{k+1}-i-j+(8-k-1)) = \alpha_k - \alpha_{k+1} + 1 \le 4,
\end{equation*}
it follows that the cohomology is trivial unless
\begin{equation*}
\alpha_3-i-j+5 \ge 5
\qquad\text{or}\qquad 
\alpha_1-i-j+7 \le 0.
\end{equation*}

In the first case we obtain~$\alpha_3 \ge i + j$, and since on the other hand~$i \ge 3\alpha_3$, we obtain~$2i + 3j \le 0$.
Since~$i \ge 0$, it follows that~$j \le 0$ and~$\tfrac23i = i \cdot \upmu(\cE_0) \le |j|$; in this case~$p = 0$.

In the second case we obtain~$\alpha_1 \le i + j - 7$, and since~$i \le 3\alpha_1$, this gives~$2i + 3j \ge 21$.
As~$i \le \rank(\cE_0) = 9$ and~$j \le n_{12} - 2 = 1$, the only possibility is~$i = 9$ and~$j = 1$;
in this case~$p = 12$.
\end{proof}

\begin{corollary}
\label{cor:bbw}
For~$g \in \{8,9,10,12\}$ we have~$\rH^p(\Gr(r,V), \wedge^{q}\cE^\vee) = 0$ for all~$p \le q$ and~$q \ge 1$.
\end{corollary}

\begin{proof}
Since~$\cE \cong \cO(1)^{\oplus (n_g - 2)} \oplus \cE_0$ by~\eqref{eq:def-ce}, it follows that
\begin{equation*}
\wedge^q\cE^\vee \cong \bigoplus_{i = 0}^q \wedge^i \cE_0^\vee \otimes \cO(i - q)^{\oplus \binom{n_g - 2}{q-i}}.
\end{equation*}
Clearly, $0 \le q - i \le n_g - 2$, hence Proposition~\ref{prop:hi-we} applies 
and shows that the cohomology of~$\wedge^i \cE_0^\vee(i - q)$ may be nontrivial 
only if~$i = q = 0$ or~$i = g - n_g$ and~$q - i = n_g - 2$.
The first case is obviously out of range~$q \ge 1$, 
while in the second the nontrivial cohomology sits in the degree~$p = g = q + 2$,
hence out of range~$p \le q$.
\end{proof}


\section{Mukai varieties of genus 12}
\label{sec:g-12}

In this section~$V_7$ is a vector space of dimension~$7$.
Consider a \emph{nondegenerate} net of skew-forms
\begin{equation*}
\sigma_0 = (\sigma_{0,1},\sigma_{0,2},\sigma_{0,3}) \in 
\wedge^2V_7^\vee \oplus \wedge^2V_7^\vee \oplus \wedge^2V_7^\vee 
\end{equation*}
where non-degeneracy means that any linear combination of~$\sigma_{0,i}$ has rank~6, see Definition~\ref{def:nondegenerate}.
We will often think of this net as an embedding~$S_3 \hookrightarrow \wedge^2V_7^\vee$ from a 3-dimensional space~$S_3$
such that the plane~$\P(S_3) \subset \P(\wedge^2V_7^\vee)$ is contained in the open~$\GL(V_7)$-orbit of skew-forms of rank~6.

\subsection{The triple Veronese surface}

With a nondegenerate net of skew-forms we associate a (triple) Veronese surface in~$\P(V_7)$.

\begin{lemma}
\label{lem:net}
Let~$S_3 \hookrightarrow \wedge^2V_7^\vee$ be a nondegenerate net of skew-forms.
The map 
\begin{equation}
\label{eq:kappa}
\kappa \colon \P(S_3) \to \P(V_7),
\qquad 
\sigma \mapsto \sigma \wedge \sigma \wedge \sigma \in \wedge^6V_7^\vee \cong V_7
\end{equation}
is a closed embedding with~$\kappa^*\cO_{\P(V_7)}(1) \cong \cO_{\P(S_3)}(3)$.
Moreover, there is a self-dual exact sequence
\begin{equation}
\label{eq:net-sequence}
0 \to \cO_{\P(S_3)}(-4) \to V_7 \otimes \cO_{\P(S_3)}(-1) \to V_7^\vee \otimes \cO_{\P(S_3)} \to \cO_{\P(S_3)}(3) \to 0,
\end{equation}
where the middle map is induced by the net and the other two maps are induced by~$\kappa$.
In particular, the surface~$\kappa(\P(S_3)) \subset \P(V_7)$ is not contained in a hyperplane.
\end{lemma}

\begin{proof}
Since the net is nondegenerate, any~$0 \ne \sigma \in S_3 \subset \wedge^2V_7^\vee$ has rank~6, 
hence~$\sigma \wedge \sigma \wedge \sigma \ne 0$ generates~$\Ker(\sigma)$.
Therefore, the map~$\kappa$ is well-defined.
Moreover, by definition~$\kappa$ is induced by the composition
\begin{equation*}
\tilde\kappa \colon \Sym^3(S_3) \to \Sym^3(\wedge^2V_7^\vee) \to V_7,
\qquad 
\sigma_1 \otimes \sigma_2 \otimes \sigma_3 \mapsto \sigma_1 \wedge \sigma_2 \wedge \sigma_3,
\end{equation*}
hence~$\kappa$ is the composition of the isomorphism of~$\P(S_3) \cong \P^2$ onto the triple Veronese surface in~$\P(\Sym^3(S_3))$ and a linear projection.
Assume~$v \coloneqq \kappa(\sigma_1) = \kappa(\sigma_2)$ for~$\sigma_1,\sigma_2 \in \P(S_3)$ (possibly, infinitely close).
Then it follows that~$S_2 \coloneqq \langle \sigma_1, \sigma_2 \rangle \subset \wedge^2(v^\perp) \subset \wedge^2V_7^\vee$,
and therefore the pencil~$\P(S_2)$ intersects the hypersurface of degenerate forms in~$\P(\wedge^2v^\perp)$ 
in contradiction to~\eqref{eq:rank-6}.
This proves that the map~\eqref{eq:kappa} is a closed embedding.

Now consider the morphism of sheaves
\begin{equation*}
V_7 \otimes \cO_{\P(S_3)}(-1) \to V_7^\vee \otimes \cO_{\P(S_3)} 
\end{equation*}
on~$\P(S_3)$ induced by the net.
The assumption that the net is nondegenerate implies that this morphism has constant rank~6,
hence its kernel and cokernel are line bundles.
Moreover, the morphism is self-dual, hence the kernel is dual to the cokernel up to twist.
Using Riemann--Roch it is easy to deduce that the cokernel is~$\cO_{\P(S_3)}(3)$ and the kernel is~$\cO_{\P(S_3)}(-4)$,
which gives us the required self-dual sequence~\eqref{eq:net-sequence}.
Comparing this construction with the definition of the map~$\kappa$ above, 
we see that the first and last maps in the sequence are induced by~$\tilde\kappa$.

The hypercohomology spectral sequence of~\eqref{eq:net-sequence} takes the form of an exact sequence
\begin{equation*}
0 \to V_7^\vee \xrightarrow{\ \ \tilde\kappa^\vee\ } \Sym^3(S_3)^\vee \xrightarrow{\qquad} S_3 \to 0.
\end{equation*}
In particular, the map~$\tilde\kappa^\vee \colon V_7^\vee \to \Sym^3(S_3)^\vee$ is injective,
hence the surface~$\kappa(\P(S_3))$ is not contained in a hyperplane.
\end{proof}

\subsection{The zero locus of~$\sigma_0$}

Let~$X \subset \Gr(3,V_7)$ be the zero locus of~$\sigma_0$ 
considered as a global section of the bundle~$\cE_0 = \wedge^2\cU^\vee \oplus \wedge^2\cU^\vee \oplus \wedge^2\cU^\vee$.
Recall that~$V_4^\perp \subset V_7^\vee$ denotes the annihilator of a subspace~$V_4 \subset V_7$.

\begin{lemma}
\label{lem:no-gr34}
If the net~$\sigma_0$ is nondegenerate, $X$ contains no~$\Gr(3,V_4)$. 
\end{lemma}

\begin{proof}
Assume~$\Gr(3,V_4) \subset X$.
This means that~$V_4$ is totally isotropic for~$S_3$, hence~$S_3 \subset V_4^\perp \wedge V_7^\vee$.
Consider the composition
\begin{equation*}
S_3 \hookrightarrow V_4^\perp \wedge V_7^\vee \twoheadrightarrow V_4^\perp \otimes V_4^\vee.
\end{equation*}
Since~$S_3 \cap \wedge^2V_4^\perp = 0$
this gives an embedding of~$\P(S_3) = \P^2$ into the projective space of 3-by-4 matrices.
Since the locus of matrices of rank~$\le 2$ has codimension~2, 
it follows that there is a point~$0 \ne \sigma \in S_3$ whose image is a matrix of rank~$\le 2$.
Then it is clear that the corresponding skew-form has rank~$\le 5$, 
in contradiction to non-degeneracy of the net.
\end{proof}

\begin{lemma}
\label{lem:no-plane}
If the net~$\sigma_0$ is nondegenerate, $X$ contains no planes.
\end{lemma}

\begin{proof}
Assume~$X$ contains the plane in~$\Gr(3,V_7)$ corresponding to a flag~$V_1 \subset V_4$.
Then any subspace~$U_3$ such that~$V_1 \subset U_3 \subset V_4$ is totally isotropic for the net.
Clearly, any~$U_2 \subset V_4$ is contained in at least one of these subspaces, 
hence~$V_4$  must be totally isotropic, hence all~$U_3$ must be totally isotropic.
But this contradicts Lemma~\ref{lem:no-gr34}.

Assume~$X$ contains the plane in~$\Gr(3,V_7)$ corresponding to a flag~$V_2 \subset V_5$.
Then the restriction of any skew form from the net to~$V_5$ contains~$V_2$ in the kernel.
Thus, this restriction has rank~2, and therefore it has a 3-dimensional kernel space.
Then it follows that the kernel of this skew-form on~$V_7$ is contained in this 3-dimensional subspace of~$V_5$,
and therefore~$\kappa(\P(S_3)) \subset \P(V_5)$, in contradiction to Lemma~\ref{lem:net}.
\end{proof}

For each vector~$0 \ne v \in V_7$ we denote by~$X_v$ the zero locus of the corresponding section of the bundle~$V_7/\cU$ on~$X$.
In other words, $X_v = \{[U_3] \in X \mid v \in U_3 \}$.

\begin{lemma}
\label{lem:v22-conics}
Assume the net~$\sigma_0$ is nondegenerate.
For any~$0 \ne \sigma \in S_3$ the subscheme~$X_{\kappa(\sigma)} \subset X$ is a conic.
Conversely, every conic on~$X$ is equal to~$X_{\kappa(\sigma)}$ for some~$0 \ne \sigma \in S_3$.
\end{lemma}

\begin{proof}
Set~$v \coloneqq \kappa(\sigma)$ and let~$V_5 \subset V_7$ be the orthogonal complement of~$v$ with respect to all skew-forms in the net 
(it is 5-dimensional, because~$v$ is the kernel of the unique skew-form in the net).
Obviously, if~$[U_3] \in X_v$ then~$v \in U_3 \subset V_5$.
Moreover, $v$ is contained in the kernel for the restriction of any skew-form from the net to~$V_5$;
therefore we have a net of skew-forms~$S_3 \hookrightarrow \wedge^2(V_5/\kk v)^\vee$ and 
\begin{equation*}
X_v = \Gr(2,V_5/\kk v) \cap \P(S_3^\perp).
\end{equation*}
Since~$X$ contains no planes by Lemma~\ref{lem:no-plane}, this is a conic.

Conversely, any conic on~$\Gr(3,V_7)$ is contained in~$\Gr(2,V_7/\kk v) \subset \Gr(3,V_7)$ for some~\mbox{$0 \ne v \in V_7$}.
It follows, therefore, that any conic on~$X$ is contained in some~$X_v$, 
so it remains to show that~\mbox{$v = \kappa(\sigma)$} for some~$0 \ne \sigma \in S_3$.
Indeed, otherwise the orthogonal complement of~$v$ with respect to all skew-forms in the net
is a 4-dimensional subspace~$V_4 \subset V_7$,
hence~$X_v \subset X \cap \Gr(3,V_4)$.
On the other hand, it is easy to see that the zero locus of one skew-form in~$\Gr(3,V_4) \cong \P^3$ 
is empty (if the restriction of the form to~$V_4$ has rank~4),
or is a line (if the restriction has rank~2),
or the whole~$\Gr(3,V_4)$ (if the restriction vanishes).
Thus, if the intersection~$X \cap \Gr(3,V_4)$ of such zero loci contains the conic~$X_v$, 
it is equal to~$\Gr(3,V_4)$, which contradicts Lemma~\ref{lem:no-gr34}.
\end{proof}

There is a simple yet useful consequence that we will need later.

\begin{lemma}
\label{lem:conics-cover}
If the net~$\sigma_0$ is nondegenerate, $X$ is covered by conics of the form~$X_{\kappa(\sigma)}$.
\end{lemma}

\begin{proof}
We need to show that each point~$x \in X$ lies on a conic~$X_{\kappa(\sigma)}$ for some~$0 \ne \sigma \in S_3$.
In other words, if~$U_3 \subset V_7$ is a totally isotropic subspace then it contains the kernel space of some~$\sigma$.
For this just note that~$S_3 \subset U_3^\perp \wedge V_7^\vee$ and consider the composition
\begin{equation*}
S_3 \hookrightarrow U_3^\perp \wedge V_7^\vee \twoheadrightarrow U_3^\perp \otimes U_3^\vee.
\end{equation*}
Since~$S_3 \cap \wedge^2U_3^\perp = 0$
this gives an embedding of~$\P(S_3) = \P^2$ into the projective space of 4-by-3 matrices.
As in the proof of Lemma~\ref{lem:no-gr34} we find a point~$0 \ne \sigma \in S_3$ whose image is a matrix of rank~$\le 2$.
Then it is clear that the corresponding skew-form has kernel in~$U_3$, as required.
\end{proof}

\subsection{Invariants}

We compute some invariants of~$X$.

\begin{proposition}
\label{prop:v22}
Let~$X \subset \Gr(3,V)$ be the zero locus of a nondegenerate global section~$\sigma_0$.
Then~$X$ is a local complete intersection Fano threefold of degree~$22$ with~$\omega_X \cong \cO_{\Gr(3,V)}(-1)\vert_X$. 

Moreover, $X$ is projectively normal with~$\rH^{>0}(X, \cO_X(t)) = 0$ for~$t \ge 0$ and~$\rH^\bullet(X, \cO_X) = \kk$.

Finally, there is an exact sequence
\begin{equation}
\label{eq:v22-cox1}
0 \to \rH^0(\Gr(3,V), \cE_0^\vee(1)) \xrightarrow{\ \rH^0(\sigma_0(1))\ } 
\rH^0(\Gr(3,V), \cO_{\Gr(3,V)}(1)) \xrightarrow{\qquad} 
\rH^0(X, \cO_{X}(1)) \to 0
\end{equation}
where~$\cE_0 = (\wedge^2\cU_X^\vee)^{\oplus 3}$ is the vector bundle defined in~\eqref{eq:def-ce-zero}.
\end{proposition}

\begin{proof}
Since~$X \subset \Gr(3,V)$ is the zero locus of a global section 
of the vector bundle~$\cE_0$ of rank~9 on a smooth variety of dimension~12,
every component of~$X$ has dimension greater or equal than~$3$.
On the other hand, Lemmas~\ref{lem:v22-conics} and~\ref{lem:conics-cover}
show that~$X$ is covered by a 2-dimensional family of curves, hence~\mbox{$\dim(X) \le 3$}.
Therefore, $X$ is a local complete intersection threefold.

The formula for the canonical line bundle follows by adjunction, 
and it also follows that the degree of~$X$ is equal to the top Chern class of~$\cE_0$,
that was computed in Lemma~\ref{lem:ctop}.

To compute the cohomology of~$\cO_X$ we consider the Koszul complex
\begin{equation*}
0 \to \wedge^9\cE_0^\vee \to \wedge^8\cE_0^\vee \to \dots \to \wedge^2\cE_0^\vee \to \cE_0^\vee \to \cO_{\Gr(3,7)} \to \cO_X \to 0
\end{equation*}
and its hypercohomology spectral sequence.
By Proposition~\ref{prop:hi-we} the only nontrivial term in it is~$\rH^0(\Gr(3,7), \cO) = \kk$,
hence the sequence degenerates in the first page and gives~$\rH^\bullet(X, \cO_X) = \kk$.

Similarly, to prove projective normality of~$X$ we twist the Koszul complex by~$\cO(t)$ with~$t \ge 0$
and look at the hypercohomology spectral sequence.
By Proposition~\ref{prop:hi-we}, we see that all nontrivial terms are in the zeroth row, 
hence the restriction morphism
\begin{equation*}
\rH^0(\Gr(3,7), \cO_{\Gr(3,7)}(t)) \to \rH^0(X, \cO_{X}(t)) 
\end{equation*}
is surjective, hence~$X$ is indeed projectively normal.
The vanishing of~$\rH^{>0}(X, \cO_X(t))$ also follows.

Finally, in the case~$t = 1$ the condition~$i \cdot \upmu(\cE_0) \le 1$ in Proposition~\ref{prop:hi-we} implies~$i \le 1$,
hence the hypercohomology spectral sequence takes the form of the short exact sequence~\eqref{eq:v22-cox1}.
\end{proof}

\begin{corollary}
\label{cor:m12-acm}
The zero locus~$X \subset \Gr(3,V)$ of a nondegenerate global section~$\sigma_0$ is ACM.
\end{corollary}

\begin{proof}
It follows from Proposition~\ref{prop:v22} that $X$ is Cohen--Macaulay
(because it is a local complete intersection in the smooth variety~$\Gr(3,V_7)$)
and projectively normal. 
It also follows that the intermediate cohomology~$\rH^p(X, \cO_X(t))$ (i.e., for~$p \in \{1,2\}$) vanishes if~$t \ge 0$
and for~$t \le -1$ the similar vanishing follows by Serre duality.
All these properties imply the ACM property of~$X$.
\end{proof}


\bibliography{fano}
\bibliographystyle{alphaspecial}

\end{document}